\documentclass[11pt]{amsart}

\pdfoutput=1

\usepackage[text={380pt,640pt},centering]{geometry}

\usepackage{cancel}
\usepackage{esint,amssymb} 
\usepackage{graphicx}
\usepackage{MnSymbol}
\usepackage{mathtools} 
\usepackage[colorlinks=true, pdfstartview=FitV, linkcolor=blue, citecolor=blue, urlcolor=blue,pagebackref=false]{hyperref}
\usepackage{microtype}

\usepackage{bm}
\usepackage{dsfont}
\usepackage{mathrsfs}
\usepackage{xcolor}

\newtheorem{proposition}{Proposition}
\newtheorem{theorem}[proposition]{Theorem}
\newtheorem{lemma}[proposition]{Lemma}
\newtheorem{corollary}[proposition]{Corollary}

\theoremstyle{remark}
\newtheorem{remark}[proposition]{Remark}

\theoremstyle{definition}
\newtheorem{definition}[proposition]{Definition}

\numberwithin{equation}{section}
\numberwithin{proposition}{section}
\numberwithin{figure}{section}
\numberwithin{table}{section}

\newcommand{\N}{\mathbb{N}}
\newcommand{\Q}{\mathbb{Q}}
\newcommand{\R}{\mathbb{R}}

\newcommand{\E}{\mathbb{E}}
\renewcommand{\P}{\mathbb{P}}
\newcommand{\EE}{\mathbf{E}}
\newcommand{\PP}{\mathbf{P}}

\newcommand{\eps}{\varepsilon}

\renewcommand{\le}{\leqslant}
\renewcommand{\ge}{\geqslant}
\renewcommand{\leq}{\leqslant}
\renewcommand{\geq}{\geqslant}

\renewcommand{\subset}{\subseteq}
\renewcommand{\bar}{\overline}
\renewcommand{\tilde}{\widetilde}

\renewcommand{\hat}{\widehat}

\newcommand{\Ll}{\left}
\newcommand{\Rr}{\right}
\renewcommand{\d}{\mathrm{d}}
\newcommand{\dr}{\partial}

\newcommand{\1}{\mathbf{1}}

\newcommand{\mcl}{\mathcal}
\newcommand{\msf}{\mathsf}

\newcommand{\msc}{\mathscr}
\newcommand{\al}{\alpha}
\newcommand{\be}{\beta}
\newcommand{\ga}{\gamma}

\newcommand{\si}{\sigma}

\DeclareMathOperator{\supp}{supp}
\DeclareMathOperator{\sign}{sign}

\newcommand{\la}{\left\langle}
\newcommand{\ra}{\right\rangle}

\newcommand{\prog}{\mathsf{Prog}}

\newcommand{\bprog}{\mathbf{Prog}}
\newcommand{\bmart}{\mathbf{Mart}}

\newcommand{\cM}{\mathcal{M}}

\newcommand{\pnp}{{(n)}}
\newcommand{\gs}{\msf {gs}}

\begin{document}

\author[H.-B. Chen]{Hong-Bin Chen}
\address[Hong-Bin Chen]{NYU-ECNU Institute of Mathematical Sciences, NYU Shanghai, Shanghai, China}

\author[A. Guionnet]{Alice Guionnet}
\address[Alice Guionnet]{Department of Mathematics, ENS Lyon and CNRS, Lyon, France}

\author[J. Ko]{Justin Ko}
\address[Justin Ko]{Department of Mathematics, Syracuse University, New York, United States}

\author[B. Lacroix-A-Chez-Toine]{Bertrand Lacroix-A-Chez-Toine}
\address[Bertrand Lacroix-A-Chez-Toine]{Department of Mathematics, King's College, London, United Kingdom}

\author[J.-C. Mourrat]{Jean-Christophe Mourrat}
\address[Jean-Christophe Mourrat]{Department of Mathematics, ENS Lyon and CNRS, Lyon, France}

\subjclass[2010]{}
\date{\today}

\title[Large deviations for the ground-state energy]{One-sided large deviations for the ground-state energy of spin glasses}

\begin{abstract}
We describe the large deviations above its typical value of the maximal energy of a spin glass with $\pm 1$ spins. Thanks to the relatively explicit description of the rate function we identify, we then show that the latter is asymptotically quadratic near its minimum if and only if an external magnetic field is present. The proof starts from a Parisi-type formula for the fractional moments of the partition function, which we then leverage to obtain the limit of the Laplace transform of the maximum energy. Using convex-duality arguments, we then rewrite this Laplace transform as a supremum over martingales, and thereby deduce the large-deviation principle with explicit rate function. 
\end{abstract}

\maketitle

%
%
%
%
%
%

\tableofcontents

\section{Introduction and main results}\label{s.intro}

Let $(\beta_p)_{p \ge 2}$ be a sequence of nonnegative real numbers such that the function $\xi(r) := \sum_{p \ge 2} \be_p^2 r^p$ is finite for every $r \in \R$. For every integer $N \ge 1$, we let $(H'_N(\sigma))_{\sigma \in \R^N}$ be the centered Gaussian field with covariance such that, for every $\si, \tau \in \R^N$,
\begin{align}\label{e.H'_N(sigma)=}
\E [H'_N(\sigma) H'_N(\tau)] = N \xi \Ll( \frac{\sigma \cdot \tau}{N} \Rr) .
\end{align}
For a fixed $h\in \R$, we consider 
\begin{equation}
\label{e.def.HN}
H_N(\sigma) := H_N'(\sigma) + h \sum_{i = 1}^N \sigma_i.
\end{equation}
In order to minimize the number of minus signs, we define the ``ground state'' energy~$L_N$ as the maximum of $H_N$ over $\{\pm 1\}^N$, normalized by $N^{-1}$, that is,
\begin{equation}\label{e.L_N=}
L_N := \max_{\sigma \in \{-1,1\}^N} \frac{H_N(\sigma)}{N}. 
\end{equation}
Our main result is an upper large-deviation principle for $L_N$, with a relatively explicit rate function. We also show that this rate function is quadratic near its minimum if and only if $h \neq 0$. 

As a preliminary step, we obtain a representation of the large-$N$ limit of $L_N$ that takes the form of a supremum, in the spirit of the ``un-inverted'' formulas introduced in \cite{mourrat2025uninverting}. We denote this limit by
\begin{equation*}  
\msf{gs} := \lim_{ N \to + \infty} \E[L_N] = \lim_{ N \to + \infty} L_N \qquad (\text{a.s.}).
\end{equation*}
We let $\mathscr P = (\Omega,(\mcl F_t)_{t \in [0,1]}, \PP)$ be a filtered probability space with complete $\sigma$-algebras (that is, for every $t \in [0,1]$, the $\sigma$-algebra $\mcl F_t$ contains every subset of any null-measure set). We assume that the probability space $\mathscr P$ is sufficiently rich that one can define a Brownian motion $(W_t)_{t \in [0,1]}$ over $\mathscr P$. We denote by $\EE$ the expectation associated with $\PP$, and by $\bmart$ the space of bounded martingales over $\msc P$. 

\begin{theorem}[Un-inverted formula for the ground-state energy]
\label{t.gs}
We have
\begin{multline}  
\label{e.gs}
\msf{gs} = \sup \bigg\{ \EE \Ll[ h\al_0+ \al_1  \int_0^1 \sqrt{\xi''(t)} \, \d W_t \Rr] 
 \ \Big| \ \al \in \bmart, \ |\alpha_1| \le 1,  \text{ and }   
\\
\forall t \in [0,1], \ \int_t^1 \xi''(s)(\EE[\al_s^2] - s) \, \d s \ge 0
\bigg\}.
\end{multline}
Moreover, the supremum in \eqref{e.gs} is achieved at a unique $\alpha \in \bmart$.
\end{theorem}
We now state the large-deviation principle for deviations of $L_N$ above its typical value; we also refer to Theorem~\ref{t.LDP} for a more general statement. For every $r > \msf{gs}$, we set
\begin{multline}
\label{e.firstdef.lambda*}
\Lambda^*(r) = \inf \Bigg\{\frac{\Ll(r-\EE\Ll[h\alpha_0+\alpha_1\int_0^1\sqrt{\xi''(t)}\d W_t\Rr]\Rr)^2}{2\int_0^1\xi''(t)\Ll(\EE\Ll[\alpha^2_t\Rr]-t\Rr)\d t}
\ \bigg|\  \ 
\\
\al \in \bmart, \ |\alpha_1| \le 1,  \text{ and }   \forall t \in [0,1], \ \int_t^1 \xi''(s)(\EE[\al_s^2] - s) \, \d s \ge 0
\bigg\}.
\end{multline}
We also set $\Lambda^*(\msf{gs}) = 0$. Later in the paper, we will define $\Lambda$ to be the limit Laplace transform of $L_N$, and show that the convex dual of $\Lambda$ is given by the expression in \eqref{e.firstdef.lambda*}, hence the notation $\Lambda^*$ for this quantity. 
\begin{theorem}[Large deviations for the ground-state energy]
\label{t.ldp}
For every $r \ge \msf{gs}$, we have
\begin{equation*}  
\lim_{N \to \infty} - \frac 1 N \log \P \Ll[ L_N \ge r \Rr] 
= \Lambda^*(r).
\end{equation*}
\end{theorem}
Thanks to the relatively explicit description of the large-deviation rate function, we can obtain some information on its asymptotic behavior near its minimum. At a heuristic level, this information should be related to the fluctuations of the ground-state energy, as discussed in more details below. 
\begin{theorem}[quadratic/non-quadratic behavior near the minimum] 
\label{t.quadratic}
If $h \neq 0$, then there exists a constant $C \in(0, +\infty)$ such that for every $r \ge \msf{gs}$,
\begin{equation*}  
C^{-1} (r-\msf{gs})^2 \le \Lambda^*(r) \le C (r-\msf{gs})^2.
\end{equation*}
Conversely, if $h = 0$, then
\begin{equation*}  
\lim_{r \to \msf{gs}^+} \frac{\Lambda^*(r)}{(r-\msf{gs})^{2}} = +\infty.
\end{equation*}
\end{theorem}

In order to prove these results, we first derive a Parisi-type formula for the fractional moments of the partition function of the model. In other words, for every $s \in (0,1)$, we give a variational formula in the form of an infimum for the large-$N$ limit of 
\begin{equation}  
\label{e.disp.fractional}
\frac 1 {sN} \log \E \Ll[\Ll(\frac 1 {2^N} \sum_{\sigma \in \{-1,1\}^N} e^{H_N(\sigma)}\Rr)^s\Rr].
\end{equation}
This is Theorem~\ref{t.fractional}, whose proof occupies Section~\ref{s.fractional}. In Section~\ref{s.laplace}, we obtain the asymptotic behavior of the Laplace transform of $N L_N$. The starting point is the observation that, for every $s > 0$,
\begin{equation}  
\label{e.ln.laplace}
\Lambda(s) := \lim_{N\to\infty}\frac{1}{N}\log \E\Ll[e^{sNL_N}\Rr] = \lim_{\beta\to \infty}\lim_{N\to\infty}\frac{1}{N}\log \E \Ll[\Ll(\frac 1 {2^N} \sum_{\sigma \in \{-1,1\}^N} e^{\beta H_N(\sigma)}\Rr)^{\frac s \beta}\Rr].
\end{equation}
We use this together with the result of the previous section to obtain a variational formula for the limit of this Laplace transform.

Once this is done, we would only need to assert the $C^1$ regularity of $\Lambda$ in order to be able to state an abstract upper large-deviation principle for $L_N$, with a rate function that is given by the convex dual of $\Lambda$. However, this representation of the large-deviation rate function would be rather inexplicit. Instead, in Section~\ref{s.uninverting} we use convex-duality arguments to obtain a more direct ``un-inverted'' representation of $\Lambda$. We use this representation in Section~\ref{s.derivation} to show that $\Lambda$ is a $C^1$ function and compute an explicit representation of its dual, and we thereby deduce a general upper large-deviation principle in Theorem~\ref{t.LDP}, which in particular implies Theorem~\ref{t.ldp}. 
The proof of Theorem~\ref{t.gs} is along similar lines as those of Theorem~\ref{t.ldp}, as it formally amounts to taking $s \to 0$ throughout; we obtain it as a consequence of Theorem~\ref{t.laplace_uninvert}. In Section~\ref{s.quadratic}, we prove Theorem~\ref{t.quadratic} by showing that, for $\alpha$ the optimal martingale from Theorem~\ref{t.gs}, we have the equivalence
\begin{equation}
\label{e.equiv.conditions}
h = 0 \qquad\text{if and only if}\qquad \int_0^1 \xi''(t) \Ll( \EE \Ll[ \alpha_t^2 \Rr] - t \Rr) \, \d t = 0.
\end{equation}

We stress that our results concern the \emph{upper} large deviations of the ground-state energy (which we defined as being the maximum of $H_N/N$). The phenomenology for the \emph{lower} large deviations can differ substantially. Indeed, for models with no external field, the speed of the lower large deviations is expected to be $N^2$ rather than $N$. For spherical models, this was shown in \cite{huang2023constructive}. There, the authors also obtain a description of the upper large-deviation rate function for $1$-RSB models, using arguments based on the Kac-Rice formula. These results and more were anticipated in \cite{lacroix2024replica} using the non-rigorous replica method. We also point out that, for the spherical SK model, the ground-state energy is nothing but the largest eigenvalue of a GOE random matrix, whose large deviations have been extensively studied and are also of speed $N^2$ and $N$ for the lower and upper large deviations respectively \cite{benarous2001aging, benarous1997large}. 

We now give a brief overview of related works. Results similar to our Theorem~\ref{t.fractional} deriving a Parisi-type formula for the limit of \eqref{e.disp.fractional} were already obtained in \cite{talagrand2007large} in the case when the function $\xi$ in \eqref{e.H'_N(sigma)=} is even, i.e.\ when $\beta_p = 0$ for every odd $p \ge 3$, using an approach based on \cite{guerra2001sum, Tpaper}. In \cite{conmin}, the authors consider the SK model (so $\beta_p = 0$ for all $p \ge 3$) with no external field, which they couple with general Poisson--Dirichlet cascades. One can transform the fractional moment in~\eqref{e.disp.fractional} into a free-energy calculation for a system that is coupled with a Poisson--Dirichlet process (as in e.g.\ \cite[Proposition~5.20]{dominguez2024book}), and our proof of Theorem~\ref{t.fractional} then follows an approach similar to that in \cite{conmin}, leveraging the more recent proof of the Parisi formula from \cite{pan.aom, pan}. In Theorem~\ref{t.fractional}, we in fact cover a broader class of energy functions than those of the form in \eqref{e.def.HN}, as we also allow for the presence of a random external field. In the absence of a random external field, the formula for the limit of \eqref{e.disp.fractional} is almost the same as the classical Parisi formula for the limit free energy, the only difference being that the Parisi measure is forced to have an atom of mass at least $s$ at $0$. The presence of a random external field introduces a further modification to the classical Parisi formula, where in place of averaging the Parisi PDE $\Phi_{\zeta}(0,h_1)$ over a coordinate $h_1$ of the random field, we take the Laplace transform of this quantity with exponent $s$.

While the justification of \eqref{e.ln.laplace} is not difficult, taking the $\beta \to \infty$ limit in the variational formula obtained in Theorem~\ref{t.fractional} is more delicate. In the $s \to 0$ limit, this corresponds to the derivation of a Parisi-type formula for the ground-state energy, as was already achieved in \cite{auffinger2017parisi} and for generalized mixed $p$-spin models in \cite{JagSenGround, grantGround}. In order to perform this step, one needs to derive some estimates on the optimizer in the variational formula from Theorem~\ref{t.fractional}. Remarkably, we do not know how to derive those directly from the variational formula. Instead, and as was done in \cite{auffinger2017parisi}, we go back to the finite-$N$ spin-glass model, and obtain the required estimates in Lemma~\ref{l.tail_gamma_P,beta}. Compared with \cite{auffinger2017parisi}, the proof of this lemma requires a new argument, as the fractional moment is more delicate to manipulate.

One motivation of our study is to better understand the un-inverted formulas for the free energy appearing in \cite{chen2025convex, issa2024hopf, mourrat2025uninverting}; see also \cite{mourrat2025spin} for an informal discussion where Theorem~\ref{t.gs} was anticipated for the SK model without external field. For classical systems of statistical mechanics such as the Curie--Weiss model, the limit free energy naturally takes the form of a supremum of a functional involving an energy and an entropy term, and the entropy term can be interpreted  as a large-deviation rate function. We hence wonder if there could be a simple expression for the large-deviation rate function of the random variable $H_N(\sigma)$ (with $\sigma$ sampled uniformly over $\{-1,1\}^N$, and for each fixed realization of the randomness). While we do not know how to derive a simple expression for the large-deviation rate function in this case, Theorem~\ref{t.ldp} demonstrates that this is indeed possible for a related large-deviation principle. As discussed in \cite{mourrat2025spin}, a better understanding of the un-inverted formulas from \cite{chen2025convex, issa2024hopf, mourrat2025uninverting} could potentially help to make progress on the identification of the limit free energy of multi-type spin glasses with non-convex covariance structure. At the moment, only partial results have been obtained on this problem, see \cite{chen2025free, chen2025hamilton, mourrat2020nonconvex, mourrat2023free}, as well as \cite{sparse_PDE, sparse_prob, mutual_information} for related problems in community detection.

Heuristically, the asymptotic behavior of the large-deviation rate function of~$L_N$ near its minimum should be related to the fluctuations of $L_N$ around its typical value. For instance, if 
\begin{equation*}  
- \frac 1 N \log \P[L_N \ge r] \simeq \frac{\sigma^2}{2} (r-\msf{gs})^2 \ \qquad ( N \to +\infty, \mbox{ then }  r \to \msf{gs}^+), 
\end{equation*}
then formally replacing $r$ with $\msf{gs} + x/\sqrt{N}$ suggests a CLT-type scaling with a Gaussian tail probability. Theorem~\ref{t.quadratic} is thus heuristically consistent with the results of \cite{chen2018energy}, where for even $\xi$, it is shown that the ground-state energy has Gaussian fluctuations of order $N^{-1/2}$ if $h \neq 0$, and that these fluctuations are of order $o(N^{-1/2})$  when $h = 0$. Similar results to those in \cite{chen2018energy} were also proved for spherical models in \cite{chen2017parisi}. Remarkably, \cite[Theorem~8.4]{chatterjee} also identifies a class of covariance functions $\xi$ for which the fluctuations of the ground-state energy are shown to be $O(N^{-5/8 + o(1)})$. 

For the spherical SK model with $h = 0$, the ground-state energy is the top eigenvalue of a GOE matrix, and its fluctuations are known to be of order $N^{-2/3}$ \cite{tracy1996orthogonal}. In the case $\xi(r) = r^p$ with $p \ge 3$ and $h = 0$, the fluctuations of the ground-state energy were shown to be of order $N^{-1}$ in \cite{subag2017extremal}.

A closely related question concerns the scale of the fluctuations of the free energy.  The free energy of the SK model with $h = 0$ was shown to have fluctuations of order $o(N^{-{1/2}})$ at all temperatures in \cite{chatterjee}. For even $\xi$ and $h \neq 0$, the free energy was shown to have Gaussian fluctuations of order $N^{-1/2}$ at all temperatures in~\cite{chen2017fluctuations}. 
For the spherical SK model with $h = 0$, the fluctuations of the free energy at low temperature were found to be of order $N^{-2/3}$ \cite{baik2016fluctuations}. We also refer to \cite{collins2025free} for a recent survey concerning the fluctuations of the free energy of the SK model.

When $h = 0$, the question of figuring out the exact order of magnitude of the fluctuations of $L_N$ is debated in the physics literature. Indeed, for the SK model, some works suggest that these fluctuations are typically of order $N^{-3/4}$ \cite{boettcher2005extremal, bouchaud2003energy, kim2007ground, palassini2003ground}, while others suggest that they are of order $N^{-5/6}$ instead \cite{aspelmeier2008finite, crisanti1992replica, palassini2008ground, parisi2008large, parisi2009phase, parisi2010universality}. As can be seen from \cite{boettcher2010simulations, palassini2003ground, palassini2008ground}, this question is difficult to settle using only the results of numerical simulations. Most relevant to the present work are the papers \cite{parisi2008large, parisi2009phase, parisi2010universality}, which argue that the large-deviation rate function should be such that 
\begin{equation}  
\label{e.fractional.power}
\Lambda^*(r) \simeq (r - \mathsf{gs})^{\frac 6 5} \qquad (r \to \mathsf{gs}^+).
\end{equation}
It would be very interesting to investigate whether a more refined analysis of our explicit formula for $\Lambda^*$ would allow us to bridge the gap between Theorem~\ref{t.quadratic} and the prediction in \eqref{e.fractional.power}. 
The rate of convergence to zero of $\E[L_N] - \mathsf{gs}$ is also a matter of debate in the physics literature \cite{aspelmeier2008finite, boettcher2005extremal, boettcher2010simulations, franz2026overlap, kim2007ground, palassini2008ground, parisi2019study}. 

Our work is inspired by \cite{lacroix2024replica}, which focuses on spherical models. In particular, the authors employ non-rigorous arguments to predict the limit of the fractional moments of the partition function, they use \eqref{e.ln.laplace} to deduce the limit of the Laplace transform of $L_N$, and they then derive and study the corresponding large-deviation principle.

%
%
%
%
%
%

\section{Fractional moments of the partition function}
\label{s.fractional}

Throughout this section, we consider the following generalization of the Hamiltonian \eqref{e.def.HN} with a random external field:
\begin{equation}\label{e.randomHamil}
    H_N(\sigma) = H'_N(\sigma) + \sum_{i=1}^N h_i \sigma_i,
\end{equation}
where $H'_N(\sigma)$ is given as in~\eqref{e.H'_N(sigma)=} and $(h_i)$ is a collection of i.i.d.\ random variables satisfying $ \E e^{\lambda |h_i|}<\infty$ for all $\lambda\geq0$. Notice that if $h_1$ is a fixed constant, then \eqref{e.randomHamil} is equivalent to \eqref{e.def.HN}.

We denote the uniform measure on $\{-1,1\}$ by $\mu$, and set $\mu_N := \mu^{\otimes N}$. 
We define the partition function $Z_N$ associated with the Hamiltonian in \eqref{e.randomHamil}  by 
\begin{align}\label{e.Z_N=}
Z_N := \int e^{H_N(\sigma) } \, \d\mu_N(\sigma). 
\end{align}

\begin{theorem}[Fractional moments of $Z_N$]\label{t.fractional}
For every $s \in (0,1)$, we have
\begin{equation}
\label{e.fractional}
\lim_{N\to \infty} \frac{1}{sN} \log \E \Ll[ Z_N^s \Rr] =  \inf_{\zeta \in \Pr_s[0,1]}\bigg(  \frac 1 s \log \E_{h_1} \Ll[ e^{s \Phi_\zeta(0,h_1)} \Rr] - \frac{1}{2} \int_0^1 t \xi''(t) \zeta([0,t]) \, \d t \bigg),
\end{equation}
where $\Phi_\zeta$ solves
\begin{equation}\label{eq:parisipdefinite}
	\begin{cases}
		-\partial_t \Phi_{\zeta} =  \frac{\xi''(t)}{2} ( \partial_{x}^2 \Phi_\zeta + \zeta([0,t]) ( \partial_x \Phi_\zeta )^2), & \quad \text{for } (t,x) \in (0,1) \times \R \\
		\Phi_{\zeta}(1,x) = \log \int e^{x \sigma} \, \d\mu(\sigma), & \quad \text{for } x \in \R,
	\end{cases}
	\, .
\end{equation}
and where $\Pr_s[0,1]$ is the space of probability measures $\zeta$ on $[0,1]$ such that $\zeta(\{0\}) \ge s$. Moreover, the infimum in \eqref{e.fractional} is achieved at a unique $\zeta \in \Pr_s[0,1]$. 
\end{theorem}
\begin{remark}  
Formally taking $s = 0$ in \eqref{e.fractional}, we recover the Parisi formula for the limit free energy, that is,
\begin{equation*}
\lim_{N\to \infty} \frac{1}{N} \E\Ll[ \log  Z_N \Rr] =  \inf_{\zeta \in \Pr[0,1]}\bigg(  \E\Ll[\Phi_\zeta(0,h)\Rr] - \frac{1}{2} \int_0^1 t \xi''(t) \zeta([0,t]) \, \d t \bigg).
\end{equation*}
\end{remark}

As was already mentioned, special cases of Theorem~\ref{t.fractional} were already obtained in \cite{talagrand2007large} and in \cite{conmin}. In this section, we provide a direct proof of this result, much in the spirit of \cite{conmin}. 

The key observation is that the fractional moments can be accessed by coupling the Hamiltonian with a Poisson--Dirichlet process. Let $(H_N(\cdot, n) )_{n \in \N}$ be independent copies of $H_N(\cdot)$, and $(h_i^{(n)})_{i,n \in \N}$ be independent copies of $h_1$. 
We define
\[
Z_N(n) = \int e^{H_N(\sigma,n) } \, \d\mu_N(\sigma) = \int e^{H'_N(\sigma,n) + \sum_{i=1}^N h_i^\pnp \sigma_i} \, \d\mu_N(\sigma)
\]
to be the partition function associated with independent copies $H_N(\sigma,n)$ of $H_N(\sigma)$ and $h_i^\pnp $ of $h_i$. Let $(v_n)_{n \in \N}$ denote the weights of the Poisson--Dirichlet process (see \cite[Section~2.2]{pan} or \cite[Section~5.5]{dominguez2024book}) associated with the parameter $s\in (0,1)$.  By the invariance property of Poisson--Dirichlet processes, see \cite[Proposition~5.20]{dominguez2024book}, we have
\[
\frac{1}{N} \E \Ll[ \log \sum_{n} v_n Z_N(n) \Rr] =  \frac{1}{Ns} \log \E \Ll[ e^{s \log Z_N } \Rr] =  \frac{1}{Ns} \log \E \Ll[ Z_N^s \Rr].
\]
It remains to compute the limit of 
\begin{equation}\label{e.feRPC}
\frac{1}{sN} \log \E \Ll[ Z_N^s \Rr] = \frac{1}{N} \E \Ll[ \log \sum_{n} v_n Z_N(n) \Rr].
\end{equation}
In the special case of the SK model without an external field, the identification of this limit follows as a special case of the general formula proved for multiscale spin glasses in \cite[Theorem~2.1]{conmin}.

\begin{proof}[Proof of Theorem~\ref{t.fractional} from \cite{conmin} for the SK model with $h_i=0$]
We first show how the results in \cite{conmin} can be applied to the current setting. Our setting corresponds to the special case when the sequence of $\gamma$ defined in \cite[Equation~8]{conmin} consists of the single point $\gamma_1 = s$. For such models, the set $X_\beta$ defined in  \cite[Equation~2.2]{conmin} consists of precisely the space of CDFs corresponding to probability measures such that $\zeta( \{0\} ) \geq s$. 
	
	The result of  Theorem~\ref{t.fractional} then follows from representing the recursive definition of \cite[Equation~3.3]{conmin}  using the Parisi PDE. This is a direct consequence of \cite[Chapter~4]{pan}. 
\end{proof}

For completeness, a detailed proof of this result using properties of the marginals of Poisson--Dirichlet processes and synchronization is described below. The proof also includes the case when the external fields $h_i$ are random. 

\subsection{Upper bound on the free energy}

This follows from the classical Guerra's interpolation argument. We want to find an upper bound of \eqref{e.feRPC}.
\begin{lemma}\label{l.UPBDFE}For independent copies $H_N(\sigma,n)$ of $H_N(\sigma)$ defined in 
\eqref{e.randomHamil},
\begin{align}
&\limsup_{N \to \infty} \frac{1}{N} \E \log \sum_{n}   \int e^{H_N(\sigma,n)} \, \d\mu_N(\sigma) v_n \notag
\\&\leq \inf_{\zeta \in \Pr_s[0,1]}\bigg(   \frac 1 s \log \E_{h_1} \Ll[ e^{s \Phi_\zeta(0,h_1)} \Rr] - \frac{1}{2} \int_0^1 t \xi''(t) \zeta([0,t]) \, \d t \bigg) .\label{e.upperbound_finite}
\end{align}
\end{lemma}
\begin{proof} The proof follows from Guerra's interpolation argument \cite{gue03}, where we interpolate the free energy with the usual limiting object with a Poisson--Dirichlet cascade corresponding to a CDF in $\Pr_s[0,1]$. 

We first fix a discrete $\zeta \in \Pr_s[0,1]$ with $r \geq 1$ steps such that $\zeta(\{ 0\}) = s$. Let $\alpha \in \N^{r- 1}$, and let $\tilde \alpha = (n,\alpha) \in \N^{r}$. Let $$0 < \zeta_0 = s < \zeta_1 < \dots < \zeta_{r - 1} < 1$$ and $$0 = q_0 < q_1 < \dots < q_r = 1$$
denote the sequences of order parameters, and let $v_{n\alpha}$ denote the weights of the Poisson--Dirichlet cascades corresponding to $\zeta$. Notice that we have constrained that $\zeta_0 = s$ to be equal to the $s$th moment of the partition function.
 
Let $Z(\tilde \alpha)$ denote  Gaussian processes indexed by $\tilde \alpha$ with covariance $\xi'(q_{\tilde \alpha^1 \wedge \tilde \alpha^2})$. Likewise, we denote by $Y(\tilde \alpha)$ Gaussian processes indexed by $\tilde \alpha$ with covariance $\theta(q_{\tilde \alpha^1 \wedge \tilde \alpha^2}) = q_{\tilde\alpha^1 \wedge \tilde \alpha^2} \xi'(q_{\tilde \alpha^1 \wedge \tilde \alpha^2}) - \xi(q_{\alpha^1 \wedge \tilde \alpha^2})$. For $1 \leq i \leq N$, we let $Z_i(\alpha,n)$ and $Y_i(\alpha,n)$ denote independent copies of $Z(\tilde \alpha)$ and $Y(\tilde \alpha)$ respectively. We consider the interpolating Hamiltonians
\[
H'_N(\sigma,\alpha,n,t) = \sqrt{t} H'_N(\sigma,n) + \sqrt{1 - t} \sum_{i = 1}^N Z_i(\alpha,n) \sigma_i + \sqrt{t}  \sum_{i = 1}^N Y_i(\alpha,n)
\]
which are independent over the $n$ coordinates. 
 
 We define the interpolating free energy
 \[
 \phi(t) = \frac{1}{N} \E \log \sum_{n} \sum_{\alpha} \int e^{H'_N(\sigma,\alpha,n,t) + \sum_{i=1}^N h_i^\pnp \sigma_i} \, \d\mu_N(\sigma) v_{n\alpha}
 \]
By Gaussian integration by parts, we have that
\begin{align}
\phi'(t) &= -\frac{1}{2}\E\la\xi(R_{12}\1(n^1=n^2))-R_{12}\xi'(\tilde Q_{12})+\theta(\tilde Q_{12})\ra_t \notag
\\&= -\frac{1}{2} \E \langle \1(n^1 = n^2) ( \xi(R_{1,2}) - \xi(Q_{1,2}) - \xi'(Q_{12}) ( R_{1,2} - Q_{1,2} )  ) \rangle_t  \label{e.upboundphi}
\end{align}
where we defined the overlaps $R_{1,2} = \frac{\sigma^1 \cdot \sigma^2}{N}$, $\tilde Q_{1,2} = \tilde \alpha^1 \wedge  \tilde \alpha^2$ and $ Q_{1,2} =  \alpha^1 \wedge  \alpha^2$. The simplification in the second line follows from the fact that $\tilde Q_{1,2} = (\alpha^1 \wedge \alpha^2) \1(n^1=n^2)$ and $\xi(0)=\xi'(0)=0$. If $\xi$ is convex then the sign of $\phi'$ is negative for all $N$. 

If $\xi$ is only convex on $\R_+$, then the argument is slightly more involved and will hold in the limit as $N \to \infty$ if we apply Talagrand's positivity principle as in \cite[Chapter 3.2]{pan}. Since the details of this modification are identical to \cite[Chapter 3.2]{pan}, we sketch the main ideas of the argument. The usual perturbation on the $\sigma$ coordinates implies that the probability of the event $\{ R_{1,2} \leq -\epsilon \}$ under the perturbed Gibbs measure tends to zero \cite[Theorem 3.4]{pan}. We can repeat the interpolation proof with the perturbed Hamiltonian to arrive at \eqref{e.upboundphi}  where $\langle \cdot \rangle_t$ is replaced with the perturbed average. We now split \eqref{e.upboundphi} along the events   $\{ R_{1,2} \leq -\epsilon \}$ and  $\{ R_{1,2} > - \epsilon \}$. On the event  $\{ R_{1,2} > - \epsilon \}$, the upper bound can be uniformly bounded above over $N$ by $L \epsilon$ for some constant $L$ depending only on $\xi$ using the convexity of $\xi$ on $\R_+$. On the other hand, the event $\{ R_{1,2} \leq -\epsilon \}$ has probability $0$ in the limit, so it can be made arbitrarily small for each fixed $\epsilon$ by taking $N$ to infinity. Therefore, taking the limit as $N \to \infty$ then taking $\epsilon \to 0$ implies that $\limsup_{N \to \infty}\phi'$ is negative.

In either case,  we can integrate $\phi'(t)$ to conclude that $\limsup_{N \to \infty} \phi(1) \leq \limsup_{N \to \infty} \phi(0)$, so
\begin{align*}
    &\limsup_{N \to \infty} \frac{1}{N} \E \log \sum_{n} \sum_{\alpha}  \int e^{H'_N(\sigma,n) + \sum_{i=1}^N h^\pnp_i \sigma_i} \, \d\mu(\sigma) v_{n \alpha} 
    \\
    &\leq \limsup_{N \to \infty}  \E \log \sum_{n} \sum_{\alpha}  \int e^{Z(\alpha,n) \sigma + \sum_{i=1}^N h^\pnp_i \sigma_i} \, \d\mu(\sigma) v_{n \alpha} -  \E \log \sum_{n} \sum_{\alpha} e^{Y(\alpha,n) } \, v_{n \alpha}
\end{align*}
 
It remains to compute the terms on the right hand side. After summing over the $\alpha$ terms and applying the fact about the marginal distribution of the Poisson--Dirichlet cascades (see Proposition~\ref{prop:marginals} below), the term on the left simplifies to
 \begin{equation}\label{e.marginals}
 	\frac{1}{N} \E \log \sum_{n} \sum_{\alpha} \int e^{H_N(\sigma,n) } \, \d\mu_N(\sigma) v_{n\alpha} = \frac{1}{N} \E \log \sum_{n}  \int e^{H_N(\sigma,n) } \, \d\mu_N(\sigma) v_{n}.
 \end{equation} 
 The right hand side can also be computed explicitly and follows from the recursive formulation of averages with respect to the Poisson--Dirichlet cascades \cite[Theorem~2.9]{pan}, which is valid because of the integrability assumptions on the external field. Consider the function
 \begin{equation}\label{e.recursive}
 X_r = \log \int e^{ \sum_{k = 1}^{r} \sqrt{ \xi'(q_{k}) - \xi'(q_{k-1}) }z_{k}  \sigma + h_1 \sigma } \, \d\mu(\sigma) 
 \end{equation}
 where the $(z_{k})$ are independent Gaussians and $h_1$ is a Gaussian random variable with the same law as $h$ independent from $(z_k)$. 
 We define for $k \geq 1$,
 \[
 X_k = \frac{1}{\zeta_k} \log \E_{z_{k + 1}} \Ll( \exp \zeta_k X_{k + 1} \Rr) \qquad\text{and}\qquad  X_0 = \frac{1}{s} \log \E_{h_1}\E_{z_{1}} \Ll( \exp s X_{1}  \Rr),
 \]
 where the averages are only with respect to the randomness in the subscript. Notice that the last step of the recursion requires an average over both $h_1$ and $z_1$ because they depend on the first index in the Poisson--Dirichlet cascades. Using the relation between the recursive formulation and the Parisi PDE \cite[Bonus Chapter A.2]{pan}, we have that
 \[
\frac{1}{s} \log \E_{z_{1}} \Ll( \exp s X_{1}  \Rr) = \Phi_{\zeta}(0,h_1).
 \]
Thus,
\[
X_0 = \frac{1}{s} \log \E_{h_1} \Ll( e^{ s\cdot \frac{1}{s} \log \E_{z_{1}} \Ll( \exp s X_{1}  \Rr)} \Rr) = \frac{1}{s} \log \E_{h_1} e^{s \Phi_{\zeta}(0,h_1)}.
\]
Finally, we arrive at the formula
\begin{equation}\label{e.upbd1}
\E \Ll[ \log \sum_{n} \sum_{\alpha}  \int e^{Z(\alpha,n) \sigma + h_1^\pnp \sigma } \, \d\mu(\sigma) v_{n \alpha} \Rr] = \frac{1}{s} \log \E_{h_1} e^{s \Phi_{\zeta}(0,h_1)}.
\end{equation}
Similarly, the recursion implies that
\begin{equation}\label{e.upbd2}
\E \Ll[ \log \sum_{n} \sum_{\alpha} e^{Y(\alpha,n) } \, v_{n \alpha} \Rr] = \frac{1}{2} \sum_{0 \leq k \leq r - 1} \zeta_k ( \theta(q_{k + 1})- \theta(q_k) ).
\end{equation}
Combining \eqref{e.upbd1} and \eqref{e.upbd2} explicit computations gives \eqref{e.upperbound_finite} for every $N$ and finite $\zeta$, provided the covariance $\xi$ in \eqref{e.H'_N(sigma)=} is convex. We emphasize that the main difference from the usual Parisi formula is that we constrain $q_1 = 0$ and $\zeta_1 = s$.

We now extend to the case when $\zeta$ is a general measure, which may not necessarily be discrete. It is known that the terms appearing on the right hand side of \eqref{e.upperbound_finite} are Lipschitz functions in $\zeta$ \cite[Bonus Chapter Lemma~A.6]{pan}. Furthermore, notice that any $\zeta \in \Pr_s[0,1]$ can be approximated in $L^1$ by discrete $\zeta^*$ so that $\zeta^*(\{0 \}) \geq s$. We can also further approximate the $\zeta^*$ in $L^1$ by splitting the mass at $0$ to get a measure $\zeta^*_\delta$ such that $\zeta^*_\delta( \{ 0\}) = s$ and $\zeta^*_\delta( \{ \delta \} ) = \zeta^*( \{ 0 \} ) - s \geq 0$, provided that $\delta$ is taken sufficiently small. This implies that we can extend the discrete upper bound proved for discrete measures satisfying $\zeta( \{ 0\}) = s$ by continuity to any $\zeta \in \Pr_s[0,1]$. 
\end{proof}

We now prove that the $n$ marginals of $v_{n \alpha}$ appearing in \eqref{e.marginals} is the Poisson--Dirichlet process. We adopt the notation of the Poisson--Dirichlet cascades as in \cite[Chapter~2]{pan}. 

\begin{proposition}\label{prop:marginals}
Let  $\tilde \alpha = (n,\alpha) \in \N^{r}$ and $v_{n\alpha}$ denote the weights of the Poisson--Dirichlet cascades corresponding to some discrete probability measure $\zeta$ such that $\zeta(\{0\}) = s$. Similarly, let $v_{n}$ denote the weights of a Poisson--Dirichlet process with parameter $s$. For any sequence of random variables $(F(n))_{n \in \N}$ that is independent of $(v_{n\alpha})$ and $(v_n)$, we have that 
\[
\E \log \sum_{n} \sum_{\alpha} F(n) v_{n\alpha} = \E \log \sum_{n} F(n) v_{n}. 
\]
\end{proposition}
\begin{proof}
This proof follows from the characterization of the Poisson--Dirichlet cascades by the Ghirlanda--Guerra identities \cite[Chapter~5.6]{dominguez2024book}. Consider i.i.d. $\tilde \alpha^\ell =  (n^{\ell},\alpha^{\ell})$ sampled according to the weights $v_{n\alpha}$. By construction, the corresponding overlap array $\tilde R_{\ell,\ell'} =  \tilde \alpha^\ell \wedge \tilde \alpha^{\ell'}$ satisfies the Ghirlanda--Guerra identities \cite[Theorem~5.28]{dominguez2024book} or \cite[Theorem~2.10]{pan}: for any $n \geq 1$, bounded and measurable function $f$ of the overlaps $\tilde R^{\leq n} = (\tilde R_{\ell,\ell'})_{\ell,\ell' \leq n}$ and any bounded measurable function $\psi: \R \to \R$, 
\[
\E \langle f(\tilde R^{\leq n}) \psi(\tilde R_{1,n + 1}) \rangle = \frac{1}{n} \E \langle f(\tilde R^{\leq n}) \rangle \E \langle \psi(\tilde R_{1,n + 1}) \rangle  + \frac{1}{n} \sum_{\ell = 2}^n \E \langle f(\tilde R^{\leq n}) \psi(\tilde R_{1,\ell}) \rangle. 
\]

We now consider the distribution of the marginal overlaps $S_{\ell,\ell'} = \1(n^{\ell} = n^{\ell'})$. It follows that the $S$ overlaps inherit the Ghirlanda--Guerra identities satisfied by the full overlaps $\tilde R$, by taking $f(S^{\leq n}) = f( \1(\tilde R ^{\leq n} \geq 0) )$ and $\psi(S_{\ell,\ell'}) = \psi( \1(\tilde  R_{\ell,\ell'} \geq 0 ) )$, for some functions $f$ and $\psi$. Furthermore, we have that the overlap array $S_{\ell,\ell'} $ takes values $0$ or $1$, and
\[
\E \langle S_{1,2} = 0 \rangle  = \E \langle \1(n^{\ell} \neq n^{\ell'}) \rangle = \E \langle \1(\tilde R_{1,2} = 0) \rangle = s
\]
and $\E \langle S_{1,2} = 1 \rangle = 1 - \E \langle S_{1,2} = 0 \rangle  = 1 - s$. Since the overlaps $S$ takes values $0$ or $1$, and they satisfy the Ghirlanda--Guerra identities, it must be the overlaps corresponding to samples according to a Poisson--Dirichlet point process with parameter $s$, see \cite[Exercise 5.10]{dominguez2024book} or \cite[Theorem~2.8]{pan}. 

In particular, if $F(n)$ is a function that only depends on the first coordinate of $\tilde \alpha = (n,\alpha) \in \N^{r}$ independent of $v_{\tilde \alpha}$, then the average of $\E \log \sum_{n} \sum_{\alpha} F(n) v_{n\alpha}$ only depends on paths in the first level corresponding to $n$ \cite[Theorem~5.25]{dominguez2024book}. Since the $n$ level is determined by a Poisson--Dirichlet process with parameter $s$, its average is equivalent to $\E \log \sum_{n} F(n) v_{n}$.  
\end{proof}

\begin{remark}
The validity of \eqref{e.marginals} can also be observed from the recursive representation of the Poisson--Dirichlet cascades \cite[Theorem~2.9]{pan}. Notice that the left hand side of \eqref{e.marginals} can be computed recursively starting at
\[
X_r(\omega_r, \dots, \omega_1) = \log \int e^{ H'_N(\sigma,\omega_1) } \, \d\mu_N(\sigma) 
\]
where the $\omega_1, \dots, \omega_r$ are independent uniform random variables encoding the randomness at each level of the RPC. The random variable $H'_N(\sigma,\omega)$ is a Gaussian process with covariance $$\E\Ll[ H'_N(\sigma^1,\omega^1) H'_N(\sigma^2,\omega^2) \Rr] =  N \xi \Ll( \frac{\sigma^1 \cdot \sigma^2}{N} \1(\omega^1 = \omega^2) \Rr) .$$ Since the right hand side of $X_r(\omega_r, \dots, \omega_1)$ does not depend on $\omega_r, \dots, \omega_2$, we have that the recursive terms given in \eqref{e.recursive} satisfy
\[
X_r = X_k = \dots = X_1.
\]
Next, notice that the right hand side
\[
\frac{1}{N} \E \log \sum_{n}  \int e^{H_N(\sigma,n) } \, \d\mu_N(\sigma) v_{n}
\]
of \eqref{e.marginals} can be computed recursively starting at $ X_1$. These observations imply that both the left and right hand side of \eqref{e.marginals} give the same value when computing the average over the RPC. 

Yet another proof of  \eqref{e.marginals} involves explicitly computing the marginals using the invariance of the Poisson--Dirichlet process \cite[Theorem~2.6]{pan}. This computation is direct, but also tedious. 
\end{remark}

\subsection{Lower bound on the free energy}

We now prove the matching lower bound.
\begin{lemma}\label{l.LOBDFE}For independent copies $H_N(\sigma,n)$ of $H_N(\sigma)$ defined in 
\eqref{e.randomHamil},
\begin{align*}
&\liminf_{N \to \infty} \frac{1}{N} \E \log \sum_{n}   \int e^{H_N(\sigma,n)} \, \d\mu_N(\sigma) v_n 
\\&\geq \inf_{\zeta \in \Pr_s[0,1]}\bigg(  \frac 1 s \log \E_{h_1} \Ll[ e^{s \Phi_\zeta(0,h_1)} \Rr] - \frac{1}{2} \int_0^1 t \xi''(t) \zeta([0,t]) \, \d t \bigg)    .
\end{align*}
\end{lemma}
\begin{proof}
The proof of the lower bound is mostly identical to the usual proof of the lower bound for the mixed $p$-spin models in \cite[Chapter~3]{pan}. The main challenge is that we have to justify why the lower bound can be written in terms of the infimum over a single order parameter characterized by the Poisson--Dirichlet cascades.  To simplify the presentation, we will focus on the key modifications and explain where synchronization \cite{pan.multi} and the structure of the covariance of $H_N(\sigma,n)$ plays its critical role. 

The starting point is the Aizenman--Sims--Starr scheme \cite{AS2}. We define the cavity fields $Z(\sigma,n)$ and $Y(\sigma,n)$ to be Gaussian processes indexed by points in $\Sigma_N \times \N$ with covariance
\[
\E \Ll[ Z(\sigma^1,n^1) Z(\sigma^2,n^2) \Rr] = \xi'\Ll(\frac{\sigma^1\cdot\sigma^2}{N} \1(n^1 = n^2) \Rr)  = \xi'(R_{12}S_{12}) 
\]
and
\[
\E \Ll[ Y(\sigma^1,n^1) Y(\sigma^2,n^2) \Rr] = \theta\Ll(\frac{\sigma^1\cdot\sigma^2}{N} \1(n^1 = n^2) \Rr)  = \theta(R_{12}S_{12}).
\]
We use the notation $R_{\ell\ell'} = \frac{\sigma^\ell\cdot\sigma^{\ell'}}{N}$ and $S_{\ell\ell'} = \1(n^\ell = n^{\ell'})$ to denote the two overlaps.  The appearance of the $S$ overlap $S_{\ell\ell'} = \1(n^\ell = n^{\ell'})$ terms in these cavity fields is the only difference with the usual cavity fields in \cite[Equation~3.60]{pan}. 

To understand the distribution of the arrays 
\[
C_{\ell\ell'} = \xi'(R_{\ell \ell'} S_{\ell\ell'}) \text{ and } D_{\ell\ell'} = \theta(R_{\ell \ell'} S_{\ell\ell'}) ,
\]
it suffices to understand the limiting behavior of $R_{\ell \ell'} S_{\ell\ell'}$, since both $C_{\ell\ell'}$ and $D_{\ell\ell'}$ are continuous functions of this array.

We now argue that we can synchronize the overlaps $R_{12}$ and $S_{12}$ to recover the distribution of $R_{\ell \ell'} S_{\ell\ell'}$, without changing the limit of the free energy. This is achieved by adding a perturbation Hamiltonian with covariances given by monomials $R_{12}^n S_{12}^m$ and apply \cite[Theorem~4]{pan} or \cite[Section~5]{mourrat2020nonconvex} to conclude that the limiting joint law of $(R_{12},S_{12})$ along any converging subsequence can be written as $(F^{-1}_R(U), F^{-1}_S(U))$ where $F^{-1}_R$ and $F^{-1}_S$ are the quantile transforms of the marginal distributions of $R_{12}$ and $S_{12}$ and $U$ is uniform on $[0,1]$. Furthermore, both the marginals distribution of $R_{\ell\ell'}$ and $S_{\ell\ell'}$ satisfy the Ghirlanda--Guerra identities in the limit, so their distributions are encoded by their corresponding  order parameters (the distribution of the first off diagonal entry) by \cite[Theorem~2.14 and Theorem~2.15]{pan}. We conclude that the limiting distribution of $R_{\ell\ell'}S_{\ell\ell'}$ can be encoded by one order parameter. 

It remains to argue that the order parameter  of $C_{\ell\ell'}$ and $D_{\ell\ell'}$ has a point mass of at least $s$ at $0$, i.e.\ it lives in the space $\Pr_s[0,1]$. By the properties of the Poisson--Dirichlet process \cite[Equation~2.34]{pan} we have that
\[
\E\Ll[ \langle \1(S_{12} = 0) \rangle' \Rr] = s \text{ and } \E\Ll[ \langle \1(S_{12} = 1) \rangle' \Rr] = 1-s
\]
where $\langle \cdot \rangle'$ is the average with respect to the perturbed Gibbs measure that will arise from the cavity computations. Therefore, we have that
\[
\E\Ll[ \langle  \1(R_{12}S_{12} = 0) \rangle' \Rr] \geq \E\Ll[ \langle \1(S_{12} = 0) \rangle' \Rr] = s 
\]
so the distribution of $R_{12}S_{12}$ belongs to $\Pr_s[0,1]$. 

We have shown that the limiting joint behavior of the overlaps $(R_{12},S_{12})$ is encoded by an order parameter $\zeta \in \Pr_s[0,1]$. This gives us a complete characterization of all possible limit points of the distribution of $R_{\ell\ell'}S_{\ell\ell'}$ under the perturbed Gibbs measure. The rest of the proof is now a standard application of the Aizenman--Sims--Starr scheme \cite[Chapter~3]{pan}. By the Aizenman--Sims--Starr representation \cite[Theorem~3.6]{pan}, we can express the lower bound as 
\begin{multline*}
	\liminf_{N \to \infty} \frac{1}{N} \E\Ll[ \log \sum_{n} v_n Z_N(n) \Rr]
	\\\geq \liminf_{N \to \infty} \bigg( \E \Ll[ \log \bigg\langle \int e^{ Z(\sigma,n) \epsilon + h^\pnp \epsilon} \d\mu(\epsilon)  \bigg\rangle'\Rr]  -  \E \Ll[ \log \langle e^{Y(\sigma,n)} \rangle'  \Rr] \bigg)
\end{multline*}
where the $\langle \cdot \rangle'$ is a Gibbs average with respect to the bulk measure plus a perturbation term that enforces the Ghirlanda--Guerra identities and synchronization. From above, we have shown that the covariance of the cavity fields $Z(\sigma,n)$, $h^\pnp$ and $Y(\sigma,n)$ are encoded by an order parameter $\zeta \in \Pr_s[0,1]$. The simplification of these cavity fields using the recursive computation along the levels of the Poisson--Dirichlet cascades is identical to the computations in the proof of Lemma~\ref{l.UPBDFE}, which gives us the matching lower bound. 
\end{proof}

\begin{proof}[Proof of Theorem~\ref{t.fractional}]
Lemmas~\ref{l.UPBDFE} and~\ref{l.LOBDFE} yield~\eqref{e.fractional}. 
To see the uniqueness of minimizers, we recall from~\cite[Theorem~4]{auffinger2015parisi} that, for distinct probability measures $\zeta_0,\zeta_1$ on $[0,1]$, 
\begin{align}\label{e.Phi_convexity}
    \Phi_{\zeta_\lambda}(0,x)<(1-\lambda) \Phi_{\zeta_0}(0,x) + \lambda \Phi_{\zeta_1}(0,x)
\end{align}
for every $\lambda\in(0,1)$ and $x\in\R$, where we have set $\zeta_\lambda = (1-\lambda)\zeta_0 + \lambda \zeta_1$. Applying this and Hölder's inequality, we can get
\begin{align}\label{e.logEePhi_convexity}
    \log\E_{h_1}\Ll[e^{s\Phi_{\zeta_\lambda}(0,h_1)}\Rr] <(1-\lambda) \log\E_{h_1}\Ll[e^{s\Phi_{\zeta_0}(0,h_1)}\Rr] + \lambda \log\E_{h_1}\Ll[e^{s\Phi_{\zeta_1}(0,h_1)}\Rr]
\end{align}
implying the strict convexity of the functional inside the infimum in~\eqref{e.fractional}. Therefore, the minimizer must be unique.
\end{proof}

%
%
%
%
%
%
\section{Laplace transform of the ground-state energy}
\label{s.laplace}

Compared with the previous section, we now slightly restrict the class of admissible distributions for the random external field, by specializing to Hamiltonians of the form
\begin{align}\label{e.H_N=H'_N+g+h}
    H_N(\sigma) = H'_N(\sigma) + \sum_{i=1}^N g_i \sigma_i + \sum_{i=1}^N h_i \sigma_i,
\end{align}
where $H'_N(\sigma)$ is given as in~\eqref{e.H'_N(sigma)=}, $(g_i)$ is a collection of i.i.d.\ centered Gaussian random variables, and $(h_i)$ is a collection of i.i.d.\ bounded random variables. Here, the variance of $g_i$ is not necessarily one. We assume $(H'_N(\sigma))_{\sigma\in\Sigma_N}$, $(g_i)_{i\in\N}$, and $(h_i)_{i\in\N}$ are independent from each other. The goal of this section is to prove the following result. 

\begin{theorem}[Laplace transform of the ground-state energy]
\label{t.laplace}
For every $s \ge 0$, we have
\begin{equation}
\label{e.laplace}
\begin{split}
    &\lim_{N\to \infty} \frac{1}{sN} \log \E \Ll[ \exp(sN L_N) \Rr]  \\
&=  \inf_{\gamma \in \cM_{s }}\bigg( \frac 1 s \log \E_{g_1,h_1}\Ll[e^{s\Psi_{\gamma}(0,g_1+h_1)}\Rr] - \frac{1}{2} \int_0^1 t \xi''(t) \gamma(t) \, \d t \bigg),
\end{split}
\end{equation}
where $\Psi_\ga$ solves
\begin{equation}\label{eq:parisipsi}
	\begin{cases}
		-\partial_t \Psi_{\ga} =  \frac{\xi''(t)}{2} ( \partial_{x}^2 \Psi_\ga + \ga(t) ( \partial_x \Psi_\ga )^2), & \quad \text{for } (t,x) \in (0,1) \times \R \\
		\Psi_{\ga}(1,x) = |x| , & \quad \text{for } x \in \R,
	\end{cases}
	\, .
\end{equation}
and where 
\begin{multline}  \label{e.M_s=}
\mcl M_s := \big\{ \gamma : [0,1) \to \R \ \mid \ \gamma \text{ is right-continuous, nondecreasing, }
\\ \text{integrable, and } \gamma(0) \ge s \big\} .
\end{multline}
When $s=0$, $\frac{1}{sN}\log\E\Ll[\exp(sNL_N)\Rr]$ and $\frac 1 s \log \E_{g_1,h_1}\Ll[e^{s\Psi_{\gamma}(0,g_1+h_1)}\Rr]$ are understood as $\E\Ll[L_N\Rr]$ and $\E_{g_1,h_1}\Ll[\Psi_{\gamma}(0,g_1+h_1)\Rr]$, respectively.
Moreover, the infimum in \eqref{e.laplace} is achieved at a unique $\ga \in \mcl M_s$. 
\end{theorem}

\begin{remark}  
For future reference, we rewrite explicitly the Parisi-type variational formula for the ground-state energy, namely
\begin{equation}  
\label{e.def.gs}
\gs = \inf_{\gamma \in \cM_{0 }}\bigg( \E_{g_1,h_1}\Ll[\Psi_{\gamma}(0,g_1 + h_1)\Rr] - \frac{1}{2} \int_0^1 t \xi''(t) \gamma(t) \, \d t \bigg).
\end{equation}
\end{remark}

\begin{remark}[Deterministic external fields]\label{r.determ_ext_field}
As mentioned in the introduction, our main focus is on the case where the external field is deterministic. Let $h\in\R$ and we set $g_i=0$ and $h_i=h$ in Theorem~\ref{t.laplace} to recover the results needed in Sections~\ref{s.intro}, \ref{s.uninverting}, and~\ref{s.derivation}. In particular, we have $\frac 1 s \log \E_{g_1,h_1}\Ll[e^{s\Psi_{\gamma}(0,g_1+h_1)}\Rr] = \Psi_{\gamma}(0,h)$.
\end{remark}

The well-posedness of the PDE~\eqref{eq:parisipsi} has been established in~\cite{auffinger2017parisi,chen2018energy}.

Let $\beta\geq0$ be the inverse temperature (unrelated to the coefficients $(\beta_p)_{p\geq 2}$ introduced above \eqref{e.H'_N(sigma)=}) and set
\begin{align*}
    Z_N(\beta) := \int e^{\beta H_N(\sigma) } \, \d\mu_N(\sigma).
\end{align*}
In particular, we have $Z_N(1)=Z_N$ as defined in~\eqref{e.Z_N=}.
Using $\log Z_N(\beta)\leq \beta NL_N\leq \log \Ll(2^NZ_N(\beta)\Rr)$, we get
\begin{align*}
     \frac{1}{N}  \log \E\Ll[ e^{sN ( L_N - \frac{\log 2}{\beta}) }\Rr]  \leq  \frac{1}{N} \log \E \Ll[e^{sN( \frac{1}{N\beta}  \log Z_N(\beta) )} \Rr] \leq \frac{1}{N}  \log \E \Ll[e^{sN L_N} \Rr],
\end{align*}
which implies
\begin{align}\label{e.lim1/sNloge^sNL_N}
    \lim_{N\to\infty}\frac{1}{sN}\log \E\Ll[e^{sNL_N}\Rr] = \lim_{\beta\to \infty}\lim_{N\to\infty}\frac{1}{sN}\log \E\Ll[Z_N(\beta)^\frac{s}{\beta}\Rr].
\end{align}

Henceforth, let $\beta\geq s \geq 0$. Recall $\mcl M_s$ from~\eqref{e.M_s=} and set $\mcl M_s^\beta:=\{\gamma\in\mcl M_s:\: \gamma(1)=\beta\}$.
For each $\gamma\in\mcl M_s^\beta$ and $\beta>0$, we consider the solution $\Psi_{\ga,\be}$ of the equation
\begin{align}\label{e.Psi_gamma,beta}
\begin{cases}
    -\partial_t \Psi_{\ga,\be} =  \frac{\xi''(t)}{2} \Ll( \partial_{x}^2 \Psi_{\ga,\be} + \ga(t) ( \partial_x \Psi_{\ga,\be} )^2\Rr), & \quad \text{for } (t,x) \in (0,1) \times \R \\
    \Psi_{\ga,\be}(1,x) = \tfrac{1}{\beta}\log \int e^{\beta x\sigma}\d \mu(\sigma) , & \quad \text{for } x \in \R.
\end{cases}
\end{align}
Then, applying Theorem~\ref{t.fractional} with $s,\xi$ substituted by $\frac{s}{\beta},\beta^2\xi$ and rewriting the variational formula in~\eqref{e.fractional} via the relation
\begin{align}\label{e.gamma=,Psi_gamma,beta=}
    \gamma=\beta\zeta([0,\cdot])\qquad\text{and}\qquad \Psi_{\ga,\be}(\cdot,\cdot) = \frac{1}{\be}\Phi_\zeta(\cdot,\beta\cdot),
\end{align}
we get
\begin{align}\label{e.lim1/sNlogE[Z_N(beta)^s/beta]}
    \lim_{N\to\infty}\frac{1}{sN}\log \E\Ll[Z_N(\beta)^\frac{s}{\beta}\Rr] = \inf_{\gamma \in \cM_s^\beta}\bigg( \frac{1}{s}\log\E\Ll[e^{s\Psi_{\gamma,\beta}(0,g_1+h_1)}\Rr] - \frac{1}{2} \int_0^1 t \xi''(t) \gamma(t) \, \d t \bigg).
\end{align}
Moreover, this theorem also ensures that the infimum in~\eqref{e.lim1/sNlogE[Z_N(beta)^s/beta]} admits a unique minimizer $\gamma_{P,\beta}$. 

Taking the limit $\beta \to \infty$, we see that $\Psi_{\gamma,\beta}(1,\cdot)$ converges pointwise to $\Psi_\gamma(1,\cdot)$, and thus we expect $\Psi_{\gamma,\beta}$ to converge pointwise to $\Psi_\gamma$. Combining this with~\eqref{e.lim1/sNlogE[Z_N(beta)^s/beta]} and~\eqref{e.lim1/sNloge^sNL_N} gives Theorem~\ref{t.laplace}, heuristically.

To justify this rigorously, we can follow the strategy from~\cite{auffinger2017parisi}, which carries over with only minor adjustments, except for the proof of Lemma~\ref{l.tail_gamma_P,beta} below. The following is a modification of~\cite[Lemma~2]{auffinger2017parisi}.

\begin{lemma}[Upper bound]\label{l.laplace_<}
For every $s \in (0,1)$, we have
\begin{equation}
\label{e.laplace_<}
\lim_{N\to \infty} \frac{1}{sN} \log \E \Ll[ e^{sN L_N} \Rr]   \leq  \inf_{\gamma \in \cM_{s }}\bigg( \frac 1 s \log \E_{g_1,h_1}\Ll[e^{s\Psi_{\gamma}(0,g_1+h_1)}\Rr] - \frac{1}{2} \int_0^1 t \xi''(t) \gamma(t) \, \d t \bigg).
\end{equation}
\end{lemma}

\begin{proof}
We first consider $\gamma\in\cM_s$ satisfying $\gamma(1-)<\infty$. For $\beta>\gamma(1-)$, we consider $\gamma_\beta(s) = \gamma(s)\mathds{1}_{[0,1)}(s)+ \beta \mathds{1}_{\{1\}}(s)$ which belongs to $\cM^\beta_s$ defined above~\eqref{e.Psi_gamma,beta}. Then, we see that $\Psi_{\gamma_\beta,\beta}$ solves
\begin{align*}
    -\partial_t \Psi_{\ga_\be,\be} =  \frac{\xi''(t)}{2} \Ll( \partial_{x}^2 \Psi_{\ga_\be,\be} + \ga_\be(t) ( \partial_x \Psi_{\ga_\be,\be} )^2\Rr)
\end{align*}
for $(t,x) \in (0,1) \times \R$ with terminal condition $\Psi_{\ga_\be,\be}(1,x) = \tfrac{1}{\beta}\log \int e^{\beta x\sigma}\d \mu(\sigma)$. Since this terminal condition converges to $\Psi_\gamma(1,\cdot)=|\cdot|$ (in~\eqref{eq:parisipsi}) pointwise, we can deduce that $\Psi_{\gamma_\beta,\beta}$ converges pointwise to $\Psi_\gamma$ given in~\eqref{eq:parisipsi}. We recall that $\Psi_{\gamma}(0,\cdot)$ is Lipschitz as proved in~\cite[Proposition~2~(ii)]{chen2018energy} and that $\Psi_{\gamma_\beta,\beta}(0,\cdot)$ is $1$-Lipschitz as proved in~\cite[Proposition~2]{auffinger2015properties}\footnote{The $1$-Lipschitzness is proved for $\Phi_\zeta(0,\cdot)$ in~\eqref{eq:parisipdefinite} and the same property for $\Psi_{\gamma_\beta,\beta}(0,\cdot)$ can be deduced from~\eqref{e.gamma=,Psi_gamma,beta=}.}. Therefore, we can find a linear growth bound uniformly for $\Psi_\gamma$ and $\Psi_{\gamma_\beta,\beta}$, which allows us to use the dominated convergence theorem to get $\lim_{N\to\infty}\E_{g_1,h_1}\Ll[e^{s\Psi_{\gamma_\beta,\beta}(0,g_1+h_1)}\Rr]= \E_{g_1,h_1}\Ll[e^{s\Psi_{\gamma}(0,g_1+h_1)}\Rr]$. On the other hand, we also have $\int_0^1 t \xi''(t) \gamma_\beta(t) \, \d t = \int_0^1 t \xi''(t) \gamma(t) \, \d t$. 
Hence, we obtain
\begin{align*}
    &\bigg( \frac 1 s \log \E_{g_1,h_1}\Ll[e^{s\Psi_{\gamma}(0,g_1+h_1)}\Rr] - \frac{1}{2} \int_0^1 t \xi''(t) \gamma(t) \, \d t \bigg)
    \\
    &=\lim_{\beta\to\infty}\bigg( \frac 1 s \log \E_{g_1,h_1}\Ll[e^{s\Psi_{\gamma_\beta,\beta}(0,g_1+h_1)}\Rr] - \frac{1}{2} \int_0^1 t \xi''(t) \gamma_\beta(t) \, \d t \bigg) 
    \\
    &\stackrel{\eqref{e.lim1/sNloge^sNL_N}\eqref{e.lim1/sNlogE[Z_N(beta)^s/beta]}}{\geq} \lim_{N\to \infty} \frac{1}{sN} \log \E \Ll[ e^{sN L_N} \Rr].
\end{align*}
For $\gamma\in\cM_s$ with $\gamma(1-)=\infty$, we can still have the bound as in the display (without the second line) by approximation using the following continuity of $\Psi_\gamma$ in $\gamma$.
For every $\gamma, \gamma' \in \mcl M_0$, $t \in [0,1]$, and $x \in \R$, we have
\begin{equation*}  
\Ll| \Psi_\ga(t,x) -  \Psi_{\gamma'}(t,x) \Rr| \le 2\xi''(1)  \|\ga  - \gamma'\|_{L^1}
\end{equation*}
which is proved in~\cite[Corollary~2]{auffinger2017parisi}. Then, we can deduce~\eqref{e.laplace_<}.
\end{proof}

To show the lower bound, we need more preparation. 
For each $\beta$, we take $\gamma_{P,\beta}$ to be the unique minimizer of the right-hand side in~\eqref{e.lim1/sNlogE[Z_N(beta)^s/beta]}, which is given by Theorem~\ref{t.fractional} and~\eqref{e.gamma=,Psi_gamma,beta=} as aforementioned. To adapt the argument in~\cite{auffinger2017parisi}, the only significant modification is to control the behavior of $\gamma_{P,\beta}$ uniformly in $\beta$, which parallels the estimate~\cite[(3.2)]{auffinger2017parisi} and the displayed one above it. This is done in the next lemma. We define
\begin{align*}
    \bar\xi(r) :=\xi(r) + \E[g_1^2]r,\qquad \forall r\in\R,
\end{align*}
which is the covariance of the Gaussian part in $H_N(\sigma)$ as in~\eqref{e.H_N=H'_N+g+h}. Notice that $\xi''=\bar\xi''$.

\begin{lemma}[Uniform control on $\gamma_{P,\beta}$]\label{l.tail_gamma_P,beta}
Let $s \geq 0$.
Then, there is a constant $C$ depending only on $\bar\xi$, $s$, and $\|h_1\|_{L^\infty}$ such that
\begin{align}\label{e.gamam_L^1_bd}
    \int_0^1\bar\xi'(t)\gamma_{P,\beta}(t)\d t \leq C,\quad\forall \beta\geq s
\end{align}
and consequently
\begin{align}\label{e.l.tail_gamma_P,beta}
    \gamma_{P,\beta}(t)\leq \frac{C}{\bar\xi(1)-\bar\xi(t)},\quad\forall t\in[0,1),\ \forall \beta\geq s. 
\end{align}
\end{lemma}

\begin{proof}
Define $\bar H_N(\sigma) = H'_N(\sigma) + \sum_{i=1}^Ng_i\sigma_i$.
For each $n\in\N$, let $(\bar H_N(\sigma,n))_{\sigma\in\Sigma_N}$ and $(h^\pnp_i)_{i\in\N}$ be independent copies of $(\bar H_N(\sigma))_{\sigma\in\Sigma_N}$ and $(h_i)_{i\in\N}$, respectively. 
Let $(\nu_n)_{n\in\N}$ be the random weights of the Poisson--Dirichlet cascade with parameter $\frac{s}{\beta}$, independent from other randomness.
Similarly as in~\eqref{e.feRPC}, we have
\begin{align}\label{e.logE[Z..]=Elogsum}
    \log \E \Ll[Z_N(\beta)^\frac{s}{\beta}\Rr] = \frac{s}{\beta}\E \log \sum_{n\in\N}\nu_n \int e^{\beta \bar H_N(\sigma,n)+\beta \sum_{i=1}^N h^\pnp_i\sigma_i}\d \mu_N(\sigma).
\end{align}
For $\ell,\ell'\in\N$, we denote the spin overlap and the cascade overlap by $R_{\ell,\ell'} = \frac{\sigma^\ell\cdot\sigma^{\ell'}}{N}$ and $S_{\ell,\ell'} = \mathds{1}_{n^\ell = n^{\ell'}}$, respectively. Then, due to $\bar\xi(0)=0$, we have 
\begin{align*}
    \E \Ll[\bar H_N(\sigma^1,n^1)\bar H_N(\sigma^2,n^2)\Rr] = N\bar\xi (R_{1,2}S_{1,2}).
\end{align*}
Let $\la\cdot\ra_\beta$ be the (tensorized) Gibbs measure associated with the right-hand side of~\eqref{e.logE[Z..]=Elogsum}, with canonical random variables $(\sigma^\ell,n^\ell)_{\ell\in\N}$ (where $(\sigma^1,n^1):=(\sigma,n)$). Gaussian integration by parts gives
\begin{align}\label{e.E<H/N>=()}
    \E \la \frac{\bar H_N(\sigma,n)}{N}\ra_\beta = \beta \Ll(\bar\xi(1) - \E \la \bar\xi(R_{1,2}S_{1,2})\ra_\beta\Rr).
\end{align}
Then, writing $M_N(n) = \max_{\sigma\in\Sigma_N}\frac{\bar H_N(\sigma,n)}{N}$ and $M_N = \max_{\sigma\in\Sigma_N}\frac{\bar H_N(\sigma)}{N}$, we have 
\begin{align}\label{e.E<H/N><}
    \E \la \frac{\bar H_N(\sigma,n)}{N}\ra_\beta \leq \E \la M_N(n)\ra_\beta = 
    \frac{\E\Ll[M_NZ_N(\beta)^\frac{s}{\beta}\Rr]}{\E\Ll[Z_N(\beta)^\frac{s}{\beta}\Rr]},
\end{align}
where the equality follows from the invariance property of the Poisson--Dirichlet process (e.g., see~\cite[Theorem~5.19]{dominguez2024book}).
We claim that 
\begin{align}\label{e.claim<C}
    \text{r.h.s. in~\eqref{e.E<H/N><}} \leq C
\end{align}
for a constant $C$ only depending on $\bar\xi$, $s$, and $\|h_1\|_{L^\infty}$. Let us first use this to finish the proof.

Next, by using the same argument as in \cite{Panchenko2008} (or~\cite[Section~3.7]{pan}) for proving the differentiability of the Parisi formula and then extracting properties of the Parisi measure, we can have that the function
\begin{align*}
    \lambda \mapsto \lim_{N\to\infty} \E \log \sum_{n\in\N}\nu_n \int e^{\lambda \bar H_N(\sigma,n)+\beta \sum_{i=1}^Nh_i\sigma_i}\d \mu_N(\sigma)
\end{align*}
is differentiable and, as a further result,
\begin{align}\label{e.limE<>=int}
    \lim_{N\to\infty} \E \la \bar\xi(R_{1,2}S_{1,2})\ra_\beta = \int_0^1 \bar\xi(t) \beta^{-1}\d\gamma_{P,\beta}(t).
\end{align}
This was also done in~\cite{chen2017fluctuations} between displays (3.1) and (3.2) therein.
We recall that $\beta^{-1}\gamma_{P,\beta}$ corresponds to the usual Parisi measure $\alpha_{P,\beta}$ as in~\cite{auffinger2017parisi}.
Then, integrating by parts, we have
\begin{align*}
    &\int_0^1 \bar\xi'(t) \gamma_{P,\beta}(t)\d t = \beta\bar\xi(1) - \int_0^1\bar\xi(t)\d \gamma_{P,\beta}(t) 
    \\
    &\stackrel{\eqref{e.limE<>=int}}{=} \lim_{N\to\infty}\beta \Ll(\bar\xi(1) - \E \la \bar\xi(R_{1,2}S_{1,2})\ra_\beta\Rr) \stackrel{\eqref{e.E<H/N>=()}\eqref{e.E<H/N><}\eqref{e.claim<C}}{\leq} C
\end{align*}
which gives~\eqref{e.gamam_L^1_bd}.
Since $\int_0^1 \bar\xi'(t) \gamma_{P,\beta}(t)\d t \geq \gamma_{P,\beta}(t)\int_t^1\bar\xi'(r)\d r $ for every $t$, as $\gamma_{P,\beta}$ is increasing, the above inequality implies the desired bound in~\eqref{e.l.tail_gamma_P,beta}.

It remains to verify~\eqref{e.claim<C}.
For $\beta\geq 1$ and $a\geq 0$, consider the function
\begin{equation*}  
\chi_N(\beta, a) := \frac 1 {N} \log \E \Ll[\Ll(2^NZ_N(\beta)\Rr)^a\Rr].
\end{equation*}
This function is clearly convex in $a$. By $2^NZ_N(\beta)\geq e^{\beta H_N(\sigma)}$ and Jensen's inequality, we get $\chi_N(\beta,a)\geq -a\beta\|h_1\|_{L^\infty}$.
For every $\beta\geq1$, we have
\begin{align*}  
\chi_N\Ll(\beta, \frac s \beta \Rr) 
& = \frac 1 N \log \E \Ll[ \Ll( \sum_{\si \in \Sigma_N} e^{\beta H_N(\sigma)} \Rr)^\frac s \beta  \Rr] 
 \le \frac 1 N \log \E \Ll[ \Ll( \sum_{\si \in \Sigma_N} e^{ H_N(\sigma)} \Rr)^s  \Rr],
\end{align*}
where we used that for $\ell^p$ norms on sequences and for $\beta \ge 1$, we have $\|\cdot\|_{\ell^\beta} \le \|\cdot\|_{\ell^1}$. 
Notice that the right-hand side is now independent of $\beta$.
By Jensen's inequality, we can raise the power $s$ to an integer value and then perform explicit Gaussian integration to see that the right-hand side is bounded by some constant $C(s)=C(\bar\xi,s,\|h_1\|_{L^\infty})$ independent of $N$. We also observe that
\begin{equation*}  
\partial_a \chi_N(\beta,a) = \frac 1 {N} \frac{\E\Ll[\log \Ll(2^NZ_N(\beta)\Rr) Z_N(\beta)^a \Rr]}{\E[Z_N(\beta)^a ]}.
\end{equation*}
Since $\beta N M_N \le \log \Ll(2^NZ_N(\beta)\Rr)+\beta N\|h_1\|_{L^\infty}$, we deduce that
\begin{equation*}  
\frac{\E\Ll[M_N Z_N(\beta)^\frac s \beta\Rr]}{\E\Ll[Z_N(\beta)^\frac s \beta\Rr]} \le \frac 1 \beta \partial_a \chi_N\Ll(\beta,\frac s \beta \Rr) +\frac{\|h_1\|_{L^\infty}}{\beta}. 
\end{equation*}
Meanwhile, we use the convexity of $\chi_N$ in $a$ to get
\begin{equation*}  
\partial_a \chi_N \Ll( \beta, \frac s \beta \Rr) \le \beta \Ll(\chi_N \Ll( \beta, \frac {s+1} \beta \Rr) - \chi_N \Ll( \beta, \frac s \beta \Rr)\Rr) \le \beta C(s+1)+\beta s\|h_1\|_{L^\infty},
\end{equation*}
which along with the previous display verifies~\eqref{e.claim<C} and completes the proof.
\end{proof}

From the two estimates in Lemma~\ref{l.tail_gamma_P,beta}, we can use Helly's selection theorem and a diagonalization argument to pass $\gamma_{P,\beta}$ to the limit along a subsequence. Without loss of generality, we assume that the entire sequence converges to some $\gamma_0$ vaguely on $[0,1)$ and that 
\begin{align}\label{e.L_0=}
    L_0 : = \lim_{\beta\to\infty}\int_0^1\xi''(t)\gamma_{P,\beta}(t)\d t
\end{align}
exists. 
We define
\begin{align}\label{e.nu_0=}
    \d \nu_0(t):=\mathds{1}_{[0,1)}(t)\gamma_0(t)\d t + \frac{1}{\xi''(1)}\Ll(L_0 - \int_0^1\xi''(t)\gamma_0(t)\d t\Rr)\d\boldsymbol{\delta}_1( t)
\end{align}
where $\boldsymbol{\delta}_1$ is the Dirac measure at $1$. 
The following lemma is a restatement of~\cite[Lemma~1]{auffinger2017parisi}, with the same proof but based now on the bounds from Lemma~\ref{l.tail_gamma_P,beta}.

\begin{lemma}[\cite{auffinger2017parisi}]
If $\phi$ is any measurable function with $|\phi|\leq 1$ and $\lim_{t\to 1-}\phi(t) = \phi(1)$, then
\begin{align}\label{e.int..phi->int..phidnu}
    \lim_{\beta\to\infty}\int_0^1\xi''(t)\gamma_{P,\beta}(t)\phi(t)\d t= \int_0^1\xi''(t)\phi(t)\d \nu_0(t).
\end{align}
\end{lemma}

The last ingredient for the lower bound is the following variational representation of Parisi PDE solutions. 
Let $W=(W_t)_{t\in[0,1]}$ be the Wiener process. Let $\prog_1$ be the collection of progressively measurable processes $\alpha=(\alpha_t)_{t\in[0,1]}$ adapted to the filtration generated by $W$ and satisfying $\sup_{t\in[0,1]}|\alpha(t)|\leq1$. Then, for every $x\in\R$, we have
\begin{align}
    \Psi_{\gamma,\beta}(0,x) = \sup_{\alpha\in\prog_1}\bigg(\frac{1}{\beta}\E\log\cosh \beta\Ll( x + \int_0^1\xi''(t)\gamma(t)\alpha_t\d t +\int_0^1\sqrt{\xi''(t)}\d W_t\Rr)\notag
    \\
    -\frac{1}{2}\int_0^1 \xi''(t)\gamma(t)\E[\alpha^2_t]\d t\bigg) \label{e.AC_rep_1}
\end{align}
for every $\gamma \in \cM^\beta_s$, and we have
\begin{align}\label{e.AC_rep_2}
\begin{split}
    \Psi_{\gamma}(0,x) = \sup_{\alpha\in\prog_1}\bigg(\E\Ll| x + \int_0^1\xi''(t)\gamma(t)\alpha_t\d t +\int_0^1\sqrt{\xi''(t)}\d W_t\Rr|
    \\
    -\frac{1}{2}\int_0^1 \xi''(t)\gamma(t)\E[\alpha^2_t]\d t\bigg)
\end{split}
\end{align}
for every $\gamma \in \cM_s$. 
We also take $W$ independent from $g_1$ and $h_1$.
These two representations can be found in~\cite[Corollaries~1 and 2]{auffinger2017parisi}. The following lemma is a modification of~\cite[Lemma~3]{auffinger2017parisi}.

\begin{lemma}[Lower bound]\label{l.laplace_>}
For every $s \in (0,1)$, we have
\begin{equation}
\label{e.laplace_>}
\lim_{N\to \infty} \frac{1}{sN} \log \E \Ll[ e^{sN L_N} \Rr]   \geq  \inf_{\gamma \in \cM_{s }}\bigg( \frac 1 s \log \E\Ll[e^{s\Psi_{\gamma}(0,g_1+h_1)}\Rr] - \frac{1}{2} \int_0^1 t \xi''(t) \gamma(t) \, \d t \bigg).
\end{equation}
\end{lemma}

\begin{proof}
For each $n\in \N$, let $\sign_n:\R\to[-1,1]$ be defined by
\begin{align*}
    \sign_n(x) = 
    \begin{cases}
        1, \qquad &\text{if }x\geq 0,
        \\
        2nx + 1, \qquad &\text{if } -\frac{1}{n}\leq x<0,
        \\
        -1, \qquad &\text{if } x<-\frac{1}{n}.
    \end{cases}
\end{align*}
We view $(\sign_n)_{n\in\N}$ as approximations of the sign function, namely,
\begin{align*}
    \lim_{n\to\infty} \sign_n(x) = \sign(x):=
    \begin{cases}
        1  , \qquad &\text{if }x\geq 0,
        \\
        -1 , \qquad &\text{if }x< 0.
    \end{cases}
\end{align*}
Let $\gamma_0$ be given as above~\eqref{e.L_0=}. We write $\bar h= g_1+h_1$ for brevity and recall that $\bar h$ is independent from the Wiener process $W$.
Fix any $\alpha\in\prog_1$. For every $\tau\in(0,1)$ and $n\in\N$, define
\begin{align*}
    \phi_{\tau,n}(t) = \alpha_t\mathds{1}_{[0,\tau)}(t) + \sign_n\Ll(\bar h+\int_0^t\xi''(r)\gamma_0(r)\alpha_r\d r + \int_0^t\sqrt{\xi''(r)}\d W_r\Rr)\mathds{1}_{[\tau,1]}(t)
\end{align*}
for every $t\in[0,1]$. Then, we have $\phi_{\tau,n}\in\prog_1$ and $\lim_{t\to1-}\phi_{\tau,n}(t) = \phi_{\tau,n}(1)$, since $g_n$ is continuous. We also have
\begin{align}\label{e.phi_n...}
\begin{split}
    \phi_n(t) &:=\lim_{\tau\to1-}\phi_{\tau,n}(t)=\alpha_t\mathds{1}_{[0,1)}(t)+\sign_n(X)\mathds{1}_{\{1\}}(t),
    \\
    \lim_{n\to\infty}\phi_n(t) &= \alpha_t\mathds{1}_{[0,1)}(t)+\sign(X)\mathds{1}_{\{1\}}(t)
\end{split}
\end{align}
for every $t\in[0,1]$, where we have set
\begin{align}\label{e.X=}
    X:=\bar h +\int_0^1\xi''(t)\gamma_0(t)\alpha_t\d t+ \int_0^1\sqrt{\xi''(t)}\d W_t.
\end{align}
Recall that $\gamma_{P,\beta}$ is the minimizer of~\eqref{e.lim1/sNlogE[Z_N(beta)^s/beta]}. We write $\spadesuit= \lim_{N\to \infty} \frac{1}{sN} \log \E \Ll[ e^{sN L_N} \Rr] $ for brevity. Then, we have
\begin{align*}
    \spadesuit\stackrel{\eqref{e.lim1/sNloge^sNL_N}\eqref{e.lim1/sNlogE[Z_N(beta)^s/beta]}}{=}\lim_{\beta\to\infty} \bigg( \frac{1}{s}\log\E\Ll[e^{\Psi_{\gamma_{P,\beta},\beta}(0,\bar h)}\Rr] - \frac{1}{2} \int_0^1 t \xi''(t) \gamma_{P,\beta}(t) \, \d t \bigg)
    \\
    \stackrel{\eqref{e.AC_rep_1}}{\geq}\lim_{\beta\to\infty}\bigg(\frac{1}{s}\log\E_{\bar h}\Ll[e^{\frac{s}{\beta}\E_{W}\log\cosh \beta\Ll( \bar h + \int_0^1\xi''(t)\gamma_{P,\beta}(t)\phi_{\tau,n}(t)\d t +\int_0^1\sqrt{\xi''(t)}\d W_t\Rr)}\Rr]
    \\
    -\frac{1}{2}\int_0^1 \xi''(t)\gamma_{P,\beta}(t)\Ll(\E[\phi_{\tau,n}^2(t)]+t\Rr)\d t\bigg)
    \\
    \stackrel{\text{(Fatou's lemma)\eqref{e.int..phi->int..phidnu}}}{\geq} \frac{1}{s}\log\E_{\bar h}\Ll[e^{s\E_W\Ll| \bar h + \int_0^1\xi''(t)\phi_{\tau,n}(t)\d\nu_0(t) +\int_0^1\sqrt{\xi''(t)}\d W_t\Rr|}\Rr]
    \\
    -\frac{1}{2}\int_0^1 \xi''(t)\Ll(\E[\phi_{\tau,n}^2(t)]+t\Rr)\d \nu_0(t)
\end{align*}
where $\E_W$ averages over $\phi_{\tau,n}$ and $W$ and $\E_{\bar h}$ averages over $\bar h$.
Using~\eqref{e.nu_0=}, \eqref{e.phi_n...}, and~\eqref{e.X=} together with the dominated convergence theorem, we send $\tau\to1-$ and then $n\to\infty$ to get
\begin{align*}
    \spadesuit & \geq \frac{1}{s}\log\E_{\bar h}\Ll[e^{s\E_W\Ll|X + \sign(X)\xi''(1)\nu_0(1)\Rr|}\Rr]
    \\
    &-\frac{1}{2}\int_0^1 \xi''(t)\Ll(\E[\alpha^2_t]+t\Rr)\gamma_0(t)\d t - \frac{1}{2}\Ll(\E\Ll[\sign(X)^2\Rr]+1\Rr)\xi''(1)\nu_0(1).
\end{align*}
Since
\begin{align*}
    &\Ll|X+\sign(X)\xi''(1)\nu_0(1)\Rr|
    \\
    &=\Ll(X+\xi''(1)\nu_0(1)\Rr)\mathds{1}_{X>0}- \Ll(X-\xi''(1)\nu_0(1)\Rr)\mathds{1}_{X<0}+\xi''(1)\nu_0(1)\mathds{1}_{X=0}
    \\
    & =|X|+\xi''(1)\nu_0(1)
\end{align*}
and $\Ll(\E\Ll[\sign(X)^2\Rr]+1\Rr)\xi''(1)\nu_0(1)=2\xi''(1)\nu_0(1)$, we obtain
\begin{align*}
    \spadesuit \geq \frac{1}{s}\log\E_{\bar h}\Ll[e^{s\E_W\Ll|X\Rr|}\Rr]
    -\frac{1}{2}\int_0^1 \xi''(t)\Ll(\E[\alpha^2_t]+t\Rr)\gamma_0(t)\d t .
\end{align*}
Recall the expression of $X$ in~\eqref{e.X=} and that $\alpha\in\prog_1$ is chosen arbitrarily. Using these combined with~\eqref{e.AC_rep_2}, we get
\begin{align*}
    \spadesuit \geq \frac{1}{s}\log\E_{\bar h}\Ll[e^{s\Psi_{\gamma_0}(0,\bar h)}\Rr]
    -\frac{1}{2}\int_0^1 t\xi''(t)\gamma_0(t)\d t
\end{align*}
which implies~\eqref{e.laplace_>}.
\end{proof}



\begin{proof}[Proof of Theorem~\ref{t.laplace}]
Lemmas~\ref{l.laplace_<} and~\ref{l.laplace_>} together gives~\eqref{e.laplace}.
Recalling from \cite[Lemma~5]{chen2018energy} that $\gamma\mapsto \Psi_\gamma$ is strictly convex in the sense in as~\eqref{e.Phi_convexity}, we can conclude the uniqueness of minimizers as in~\eqref{e.logEePhi_convexity} and the ensuing lines there.
\end{proof}

%
%
%
%
%
%
\section{Un-inverted formula for the Laplace transform}
\label{s.uninverting}

From now on, we exclusively focus on the case in which the external field is deterministic, as in \eqref{e.def.HN}. The goal of this section is to derive an ``un-inverted'' variational representation (i.e.\ one that takes the form of a supremum) of the limit of the Laplace transform of $L_N$. This representation involves martingales defined over a probability space, and in order to clarify that the representation does not depend on the choice of probability space, we introduce the following definition. 

\begin{definition}[Admissible ambient probability spaces]\label{d.admissible}
A filtered probability space $\mathscr P = (\Omega,(\mcl F_t)_{t \in [0,1]}, \PP)$ is said to be \emph{admissible} if the following holds. First, the $\sigma$-algebras $(\mcl F_t)_{t \in [0,1]}$ are  complete: namely, they contain every subset of any null-measure set. Second, $\msc P$ is sufficiently rich that one can define a Brownian motion $(W_t)_{t \in [0,1]}$ over it (in particular, the process~$W$ is adapted and has independent increments with respect to the filtration~$(\mcl F_t)_{t \in [0,1]}$). 
We denote by $\bmart$ the space of bounded martingales over $\msc P$.
\end{definition}

We define
\begin{align}\label{e.G=}
    \msf G = \Ll\{g\in L^1([0,1]):\: \int_r^1 \xi''(t)\Ll(g(t) -t \Rr)\d t\geq 0, \ \forall r\in [0,1]\Rr\}.
\end{align}
We consider the solution $\Psi_\gamma$ of~\eqref{eq:parisipsi} with $\gamma \in \mcl M_0$.
By~\cite[Proposition~2~(i)]{chen2018energy}, for every $\gamma, \rho \in \mcl M_0$, $t \in [0,1]$ and $x \in \R$, we have
\begin{equation*}  
\Ll| \Psi_\ga(t,x) -  \Psi_{\rho}(t,x) \Rr| \le 2\xi''(1)  \|\ga  - \rho\|_{L^1}.
\end{equation*}
This ensures that even for unbounded $\gamma \in L^1([0,1])$, the function $\Psi_\gamma$ is still well-defined. Moreover, regularity properties of $\Psi_\gamma$ are also given in~\cite[Proposition~2]{chen2018energy}. 
Given an admissible probability space $\msc P$, we denote by $\bmart_1$ the space of martingales over $\msc P$ that take values in $[-1,1]$. We recall that from now on, we only consider the case of a deterministic external field, as in \eqref{e.def.HN}.

\begin{theorem}\label{t.laplace_uninvert}
Let $\msc P$ be admissible as in Definition~\ref{d.admissible}. Then, for every $s\geq 0$, we have
\begin{align}\label{e.lim=sup_mart_achieve}
\begin{alignedat}{2}
&\lim_{N\to \infty} \frac{1}{sN} \log \E \Ll[ \exp(sN L_N) \Rr] \\
&=\sup_{\alpha\in\bmart_1}\bigg\{\EE\Ll[h\alpha_0 + \alpha_1\int_0^1\sqrt{\xi''(t)}\d W_t\Rr]+\frac{s}{2}\int_0^1 \xi''(t)\Ll(\EE\Ll[\alpha_t^2\Rr] -t \Rr)\d t
\end{alignedat} \\
\begin{alignedat}{2}
& & & \Big|\ \EE\Ll[\alpha^2_\bullet\Rr]\in \msf G\bigg\} \notag
\end{alignedat}
\end{align}
where the left-hand side is understood as $\lim_{N \to +\infty} \E[L_N] $ when $s=0$. Moreover, there is a unique $\alpha\in \bmart_1$ that realizes the supremum in~\eqref{e.lim=sup_mart_achieve}, and this optimal martingale can be characterized as follows.
Let $\gamma\in\cM_s$ be the unique minimizer of the infimum in~\eqref{e.laplace} given by Theorem~\ref{t.laplace} (also see Remark~\ref{r.determ_ext_field}) and let $(X_t)_{t \in [0,1)}$ be the strong solution to 
\begin{equation}  
\label{e.AC.SDE}
\Ll\{
\begin{array}{ll}  
X_0 = h, \\
\d X_t = \xi''(t) \ga(t) \dr_x \Psi_{\ga}(t, X_t) \, \d t + \sqrt{\xi''(t)} \, \d W_t.
\end{array}
\Rr.
\end{equation}
This unique maximizer $\al \in \bmart_1$ must satisfy, for every $t \in [0,1)$,
\begin{equation}  
\label{e.identity.alpha}
\al_t = \dr_x \Psi_{\ga}(t,X_t) = \dr_x \Psi_{\ga}(0,h) +  \int_0^t \sqrt{\xi''(s)} \dr_x^2 \Psi_{\ga}(s,X_s) \, \d W_s.
\end{equation}
\end{theorem}



Let $\bprog_1$ be the space of progressively measurable processes $u$ adapted to the filtration of $\msc P$ that also satisfy $|u_t|\leq 1$ almost surely for every $t\in[0,1]$. We start with the following mild extension of \eqref{e.AC_rep_2}, where here we allow for a general admissible probability space, and we specify what the optimal processes are.

\begin{lemma}[Variational representation of $\Psi_\ga$]
\label{l.variat.psiga}
For every $h \in \R$ and $\ga \in \mcl M_0$, we have
\begin{multline}
\label{e.variat.psiga}
\Psi_\ga(0,h) = \sup_{\alpha\in\bprog_1} \EE \left[ 
\left|h+\int_{0}^{1}\xi''(t)\ga(t)\alpha_t \, \mathrm{d}t+\int_0^1\sqrt{\xi''(t)}\, \mathrm{d} W_t\right| \right. \\
\left. - \frac{1}{2}\int_{0}^{1}\xi''(t)\gamma(t)\alpha^{2}_t\, \mathrm{d}t \right].
\end{multline}
Moreover, let $(X_t)_{t \in [0,1)}$ be the strong solution to~\eqref{e.AC.SDE}.
When $\gamma$ is not constantly zero on $[0,1)$, a stochastic process $\al \in \bprog_1$ achieves the supremum in \eqref{e.variat.psiga} if and only if it satisfies~\eqref{e.identity.alpha} for almost every $t \in [0,1)$ with $\gamma(t) > 0$.
\end{lemma}


Notice that although $\gamma$ is not bounded, the integrals in~\eqref{e.variat.psiga} make sense since $|\alpha_t|\leq 1$ holds uniformly in $t$ and $\gamma$ is integrable (see~\eqref{e.M_s=}).
When $\gamma$ is constantly zero, the variational problem in~\eqref{e.variat.psiga} is independent of $\alpha$.

We exclude $t=1$ in~\eqref{e.identity.alpha} since both $\partial_x\Psi_\gamma(1,\cdot)$ and $\partial^2_x\Psi_\gamma(1,\cdot)$ have a singularity at $0$ (recall that $\Psi_\gamma(1,\cdot)=|\cdot|$).
It should be noted that $\alpha_t$ given as in~\eqref{e.identity.alpha} indeed satisfies $|\alpha_t|\leq 1$ a.s.\ due to
\begin{align}\label{e.|dxPhi|<1}
    \operatorname*{ess\,sup}_{(t,x)\in[0,1]\times\R}|\partial_x\Psi_\gamma(t,x)|\leq 1,
\end{align}
as proven in~\cite[Lemma~11]{chen2018energy} (the bound is proven for the case when $\gamma$ is a step function there, and the general case follows from approximations given in~\cite[Proposition~2~(iii)]{chen2018energy}).


If the underlying probability space $\msc P$ is the Wiener space, then the representation in~\eqref{e.variat.psiga} and the fact that $\alpha$ given in~\eqref{e.identity.alpha} achieves the supremum both follow from~\cite[Theorem~5]{chen2018energy}. 
Now, for more general $\msc P$, one can repeat the same argument used in~\cite[Lemma~2.1]{mourrat2025uninverting} to get the result (replacing $\phi$ there by $|\cdot|$). The uniqueness follows from the same reasoning as explained below~\cite[(2.8)]{mourrat2025uninverting}. Since there is some technicality at $t=1$ due to the non-differentiability of $\Psi(1,\cdot)=|\cdot|$ at $0$, we give a detailed proof below, modifying the argument in~\cite[Lemma~2.1]{mourrat2025uninverting}.

\begin{proof}[Proof of Lemma~\ref{l.variat.psiga}]
Since $\msc P$ is richer than the Wiener space, \cite[Theorem~5]{chen2018energy} (which proves \eqref{e.variat.psiga} for the Wiener space) implies ``$\leq$'' in~\eqref{e.variat.psiga}. It remains to show
\begin{align}\label{e.variat.psiga_lb}
\begin{split}
    \Psi_\gamma(0,h)\geq \EE \left[ 
\left|h+\int_{0}^{1}\xi''(t)\ga(t)\alpha_t \, \mathrm{d}t+\int_0^1\sqrt{\xi''(t)}\, \mathrm{d} W_t\right| \right. \\
\left. - \frac{1}{2}\int_{0}^{1}\xi''(t)\gamma(t)\alpha^{2}_t\, \mathrm{d}t \right]
\end{split}
\end{align}
for any $\alpha\in \bprog_1$. 
Henceforth, we write $\Psi=\Psi_\gamma$ for brevity.
Fix any $\alpha\in\bprog_1$ and set
\begin{align}\label{e.Y=}
    Y_t := h + \int_0^t \xi''(r)\gamma(r)\alpha_r\d r+ \int_0^t \sqrt{\xi''(r)}\d W_r,\quad\forall t\in[0,1].
\end{align}
The regularity of $\Psi$ on $[0,1)\times \R$ from~\cite[Proposition~2]{chen2018energy} allows us to apply Itô's formula to get, for $t\in[0,1)$,
\begin{align*}
    \Psi(t,Y_t)=\Psi(0,h) + \int_0^t\partial_t \Psi(r,Y_r)\d r+ \int_0^t \partial_x \Psi(r,Y_r)\d Y_r \\
    + \frac{1}{2}\int_0^t \partial^2_x\Psi(r,Y_r)\d \la Y\ra_r.
\end{align*}
Taking the expectation, we get, for $t\in[0,1)$,
\begin{align}\label{e.E[]=Psi+E...}
\begin{split}
    \EE \Ll[\Psi(t,Y_t)\Rr] &= \Psi(0,h)+\EE\int_0^t\Big(\partial_t\Psi(r,Y_r)
    \\
    &+\frac{\xi''(r)}{2}\Ll(2\gamma(r)\alpha_r\partial_x\Psi(r,Y_r)+\partial^2_x\Psi(r,Y_r)\Rr)\Big)\d r.
\end{split}
\end{align}
It is well known that when $\gamma$ is a step function (and thus bounded), $\Psi$ can be obtained by solving~\eqref{eq:parisipsi} using the Cole--Hopf transformation (see the proof of~\cite[Lemma~11]{chen2018energy}) on $I:=[0,1)\setminus\{\text{discontinuity points of $\gamma$}\}$. In particular, $\Psi$ is smooth on $I\times \R$ and we have
\begin{align*}
    -\partial_t \Psi(t,x) 
    & = \frac{\xi''(t)}{2} \Ll( \partial_{x}^2 \Psi(t,x) + \ga(t) ( \partial_x \Psi(t,x) )^2\Rr)
    \\
    & = \frac{\xi''(t)}{2} \sup_{a\in\R}\Ll ( \partial_{x}^2 \Psi(t,x) +2a \ga(t)\partial_x \Psi(t,x) - a^2 \gamma(t)\Rr)
\end{align*}
for every $(t,x) \in I \times \R$.
Substituting $(t,x)$ with $(r,Y_r)$ in the above and setting $a$ to be $\alpha_r$, we thus obtain
\begin{align}\label{e.Eint(...)<Eintgammaalpha^2}
\begin{split}
    \EE\int_0^t\Big(\partial_t\Psi(r,Y_r)
    +\frac{\xi''(r)}{2}\Ll(2\gamma(r)\alpha_r\partial_x\Psi(r,Y_r)+\partial^2_x\Psi(r,Y_r)\Rr)\Big)\d r 
    \\
    \leq \frac{1}{2} \EE \int_0^t \xi''(r)\gamma(r)\alpha^2_r\d r
\end{split}
\end{align}
for every $t\in[0,1)$, where ``$\leq$'' becomes an equality if and only if
\begin{align}\label{e.eqn_cond}
    \text{for almost every $r\in[0,1)$ such that $\gamma(r)>0$, }\quad \alpha_r = \partial_x \Psi(r,Y_r) \ \text{a.s.}
\end{align}
Recall that we have assumed that $\gamma$ is a step function. The general case of~\eqref{e.Eint(...)<Eintgammaalpha^2} follows by approximations enabled by~\cite[Proposition~2~(iii)]{chen2018energy}.
Inserting~\eqref{e.Eint(...)<Eintgammaalpha^2} into~\eqref{e.E[]=Psi+E...} and sending $t\nearrow1$, we get
\begin{align}\label{e.Psi(0,h)>E}
    \EE\Ll[\Psi(0,h)\Rr] \geq \EE [\Psi(1,Y_1)] -\frac{1}{2} \EE \int_0^1 \xi''(t)\gamma(t)\alpha^2_t\d t,
\end{align}
which together with~\eqref{e.Y=} and $\Psi(1,\cdot)=|\cdot|$ gives~\eqref{e.variat.psiga_lb} and thus implies~\eqref{e.variat.psiga}.

The fact that~\eqref{e.identity.alpha} is sufficient for maximality is already proved in~\cite[Theorem~5]{chen2018energy}.
Then, we show the necessity. Let $\alpha$ be a maximizer and thus~\eqref{e.Psi(0,h)>E} is an equality, which requires~\eqref{e.Eint(...)<Eintgammaalpha^2} to be also an equality. Then, inserting the equality condition~\eqref{e.eqn_cond} into~\eqref{e.Y=}, we see that $Y$ must satisfy the SDE~\eqref{e.AC.SDE}, which gives that $\alpha$ satisfies~\eqref{e.identity.alpha}.
\end{proof}

The following lemma corresponds to~\cite[Lemma~2.2]{mourrat2025uninverting} which considers spin glasses at finite temperature. 

\begin{lemma}[Supremum of affine functionals]
\label{l.prog.to.mart}
For every $\ga \in \mcl M_0$, we have
\begin{equation}  
\label{e.prog.to.mart}
\Psi_\ga(0,h) 
\\
= \sup_{\al \in \bmart_1} \EE \Ll[ h \al_0 + \frac{1}{2}\int_{0}^{1}\xi''(t)\ga(t) \alpha^{2}_t\, \d t + \al_1 \int_0^1 \sqrt{\xi''(t)}\, \d W_t \Rr] .
\end{equation}
Moreover, there exists a unique maximizer to the variational problem in~\eqref{e.prog.to.mart}. This unique maximizer $\al \in \bmart_1$ must satisfy~\eqref{e.identity.alpha} for every $t \in [0,1)$ with $(X_t)_{t \in [0,1]}$ solving~\eqref{e.AC.SDE}. 
\end{lemma}

Let $\alpha$ be a martingale satisfying~\eqref{e.identity.alpha} for every $t\in[0,1)$.
The boundedness of $\alpha$ allows us to apply Doob's martingale convergence theorem to see that there is some $\alpha_1 \in L^\infty(\msc P)$ such that $\alpha_t$ converges to $\alpha_1$ a.s.\ and in any $L^p$ for $p\in[1,\infty)$. Therefore, $\alpha_1$ and thus $(\alpha_t)_{t\in[0,1]}$ are uniquely determined.

Notice that Lemma~\ref{l.variat.psiga} only gives a characterization of the maximizer when $\gamma$ is not constantly zero on $[0,1)$. Hence, accordingly, the proof of Lemma~\ref{l.prog.to.mart} is divided into two cases.

\begin{proof}[Proof of Lemma~\ref{l.prog.to.mart} when $\gamma$ is constantly zero]
Setting $\gamma=0$ on $[0,1)$, we directly verify the claims. The equation in~\eqref{eq:parisipsi} becomes a heat equation and can be solved as $\Psi_\gamma(t,x)=\E\Ll[|x+\sigma(t)Z|\Rr]$ where $Z$ is a standard Gaussian variable and $\sigma(t):=(\int_t^1\xi''(r)\d r)^\frac{1}{2}$. Let $X_t=h+\int_0^t\sqrt{\xi''(r)}\d W_r$ be given as in~\eqref{e.AC.SDE} and thus we have $\Psi_\gamma(0,h)=\EE\Ll[|X_1|\Rr]$ as $X_1\stackrel{\d}{=}\sigma(0)Z$. Since $\alpha\in\bmart_1$ satisfies $\alpha_1X_1\leq |X_1|$, we have $\Psi_\gamma(0,h)=\sup_{\alpha\in\bmart_1}\EE\Ll[\alpha_1X_1\Rr]$ verifying~\eqref{e.prog.to.mart} with the maximum achieved at $\alpha_1$ satisfying $\alpha_1=1$ if $X_1>0$ and $\alpha_1=-1$ when $X_1<0$. Since $X_1$ is Gaussian and thus $X_1\neq0$ a.s., this relation uniquely determines $\alpha_1$ a.s. Since $\alpha$ is a martingale, $\alpha$ is uniquely determined by $\alpha_1$. It remains to verify that $\alpha$ satisfies~\eqref{e.identity.alpha}.

When $t<1$, we have $\sigma(t)>0$ and we can directly compute
\begin{align*}
    \partial_x\Psi_\gamma(t,x)=\P\Ll\{x+\sigma(t)Z>0\Rr\} - \P\Ll\{x+\sigma(t)Z<0\Rr\} = \frac{2}{\sqrt{\pi}}\int_0^{\frac{x}{\sigma(t)}}e^{-\frac{u^2}{2}}\d u.
\end{align*}
Since $X_1\neq 0$ a.s.\ and $X_t$ converges to $X_1$ a.s.\ (due to $X_t=h+\int_0^t\sqrt{\xi''(r)}\d W_r$), we can see from the above integral together with $\lim_{t\to1}\sigma(t)=0$ that $\partial_x\Psi_\gamma(t,X_t)$ converges to $\alpha_1$ a.s. Since $\Ll(\partial_x\Psi_\gamma(t,X_t)\Rr)_{t\in[0,1)}$ is also a martingale evident from the second equality in~\eqref{e.identity.alpha}, we must have $\alpha_t=\partial_x\Psi_\gamma(t,X_t)$ for every $t\in[0,1)$ as expected in~\eqref{e.identity.alpha}. This completes the proof in this special case.
\end{proof}

The second half of the proof uses arguments from the proof \cite[Lemma~2.2]{mourrat2025uninverting}.
Due to the non-differentiability of $\Psi_\gamma(1,\cdot)$ at $0$, some modification is needed.
In the proof, we need to consider the convex conjugate of $\Psi_\gamma(t,\cdot)$, which is given as follows.
For each $(t,x)\in[0,1]\times\R$, we define
\begin{align}\label{e.Psi^*(t,x)=}
    \Psi^*_\gamma(t,x) = \sup_{y\in\R}\Ll\{xy - \Psi_\gamma(t,y)\Rr\}.
\end{align}
To proceed, we need the following lemma, the proof of which requires additional inputs from convex analysis.


\begin{lemma}[Convergence of $\Psi_\gamma^*(t,\cdot)$]\label{l.Phi^*_cvg}
Let $\gamma\in\mcl M_0$ and let $\Psi_\gamma$ solve~\eqref{eq:parisipsi}. Then, for every $x\in\R$, we have
\begin{align}\label{e.l.Phi^*_cvg}
    \lim_{t\nearrow1}\Psi^*_\gamma(t,x)  = \Psi^*_\gamma(1,x)=
    \begin{cases}
        0, & \text{if $x\in [-1,1]$},
        \\
        +\infty, & \text{otherwise}.
    \end{cases}
\end{align}
\end{lemma}

\begin{proof}
We will need some inputs from convex analysis, which we briefly recall when needed.

\medskip
\noindent\textit{Step~1.}
We show that $(\Psi_\gamma^*(t,\cdot))_{t\in[0,1]}$ is \textit{equi-l.s.c.}\ (equi-lower-semicontinuous). We recall (see the paragraph following~\cite[7.9~Exercise]{rockafellar1998variational}) that a family $(f_i)_{i\in I}$ of extended-real-valued functions on $\R$ is said to be equi-l.s.c.\ if, for every $\bar x\in\R$ and every finite $\rho,\eps>0$, there is $\delta>0$ such that
\begin{align*}
    \inf_{x:\:|x-\bar x|\leq \delta}f_i(x) \geq \min\Ll\{f_i(\bar x) - \eps,\ \rho\Rr\},\quad\forall i\in I.
\end{align*}
Due to the bound in~\eqref{e.|dxPhi|<1}, we can see that the supremum in~\eqref{e.Psi^*(t,x)=} can be restricted to $y\in[-1,1]$. From this, we can deduce that
\begin{align}\label{e.Phi^*<Phi^*+||}
    \Psi_\gamma^*(t,x) \leq \Psi_\gamma^*(t,x') + |x-x'|,\quad\forall t\in[0,1],\ \forall x,x'\in\R
\end{align}
which implies that $(\Psi_\gamma^*(t,\cdot))_{t\in[0,1]}$ is equi-l.s.c.

\medskip
\noindent\textit{Step~2.}
We show that $\Psi_\gamma(t,\cdot)$ converges pointwise to $\Psi_\gamma(1,\cdot)$ on $\R$ as $t\nearrow1$ and that $(\Psi_\gamma(t,\cdot))_{t\in[0,1]}$ is equi-l.s.c.
By~\cite[Proposition~2]{chen2018energy}, $\Psi_\gamma$ is a weak solution (see~\cite[Definition~1]{chen2018energy}), which is continuous in particular. Hence, we get the pointwise convergence. Due to~\eqref{e.|dxPhi|<1}, $\Psi_\gamma(t,\cdot)$ is $1$-Lipschitz for every $t$, which implies equi-l.s.c.

\medskip
\noindent\textit{Step~3.}
We use results from convex analysis to conclude. 
We recall that \cite[7.10~Theorem]{rockafellar1998variational} states that, under the condition of equi-l.s.c., pointwise convergence is equivalent to \textit{epi-convergence} as in \cite[7.1~Definition]{rockafellar1998variational}.
Since we will not need the precise definition of epi-convergence, we choose to omit it here. But we mention that heuristically it means the convergence of epigraphs of convex functions.
By \cite[7.10~Theorem]{rockafellar1998variational} and Step~2, $\Psi_\gamma(t,\cdot)$ epi-converges to $\Psi_\gamma(1,\cdot)$ on $\R$. Next, \cite[11.34~Theorem]{rockafellar1998variational} shows that the action of convex conjugation preserves epi-convergence. Hence, $\Psi_\gamma^*(t,\cdot)$ epi-converges to $\Psi_\gamma^*(1,\cdot)$ on $\R$. Lastly, using~\cite[7.10~Theorem]{rockafellar1998variational} again and Step~1, we deduce the convergence part in~\eqref{e.l.Phi^*_cvg}. The second equality in~\eqref{e.l.Phi^*_cvg} follows from $\Psi(1,\cdot) =|\cdot|$.
\end{proof}

We also need~\cite[Lemma~4]{chen2017fluctuations} restated below, which is also useful later.

\begin{lemma}[\cite{chen2017fluctuations}]\label{l.d_xPsi_odd_increase}
For every $\gamma\in\mcl M_0$ and $t\in[0,1)$, $\partial_x\Psi_\gamma(t,\cdot)$ is odd and strictly increasing.
\end{lemma}

Now, we are ready to complete the other half of the proof.

\begin{proof}[Proof of Lemma~\ref{l.prog.to.mart} when $\gamma$ is not constantly zero]
For brevity, we introduce the following notation. For $t\in[0,1]$ and $\alpha \in \bprog_1$, we define
\begin{align*}
    C_t(\alpha) & = h+\int_0^t \xi''(r)\gamma(r)\alpha_r\d r+\int_0^t\sqrt{\xi''(r)}\d W_r,
    \\
    L_t(\alpha) &= \frac{1}{2}\int_0^t\xi''(r)\gamma(r)\alpha_r^2\d r,
    \\
    B_t(\alpha) &= h\alpha_t + \frac{1}{2}\int_0^t\xi''(r)\gamma(r)\alpha_r^2\d r+ \alpha_t\int_0^t\sqrt{\xi''(r)}\d W_r.
\end{align*}
In this notation, we have
\begin{align}
    \Psi_\gamma(0,h) & \stackrel{\text{L.\ref{l.variat.psiga}}}{=} \sup_{\alpha\in\bprog_1}\EE \Ll[|C_1(\alpha)|-L_1(\alpha)\Rr]\notag
    \\
    &\geq \sup_{\alpha\in\bmart_1}\EE \Ll[\alpha_1 C_1(\alpha)-L_1(\alpha)\Rr]\notag
    \\
    & = \sup_{\alpha\in\bmart_1}\EE \Ll[B_1(\alpha)\Rr] \label{e.Psi_gamma>supE[B]}
\end{align}
where the inequality follows from $\bmart_1 \subset \bprog_1$ and $|C_1(\alpha)|\geq \alpha_1 C_1(\alpha)$, and the last equality follows from $\EE \Ll[\alpha_1 C_1(\alpha)\Rr] = \EE \Ll[B_1(\alpha) + L_1(\alpha)\Rr]$ due to the martingale property.

Let $\bar\alpha \in\bmart_1$ satisfy~\eqref{e.identity.alpha} for every $t\in[0,1)$.
We claim that
\begin{align}\label{e.claim}
    \Psi_\gamma(0,h) = \EE \Ll[B_1(\bar\alpha)\Rr]
\end{align}
and postpone the verification after we complete the main objective.
Using this and~\eqref{e.Psi_gamma>supE[B]}, we obtain~\eqref{e.prog.to.mart}.

Now, we show the uniqueness of maximizers and the characterization. Let $\alpha$ be any maximizer of~\eqref{e.prog.to.mart}. Then, we have
\begin{align*}
    \Psi_\gamma(0,h) = \EE \Ll[B_1(\alpha)\Rr]= \EE \Ll[\alpha_1C_1(\alpha)-L_1(\alpha)\Rr]
    \\
    \leq \EE \Ll[|C_1(\alpha)|-L_1(\alpha)\Rr]\stackrel{\text{L.\ref{l.variat.psiga}}}{\leq}\Psi_\gamma(0,h),
\end{align*}
where the second equality is due to the martingale property and the third is due to $|\alpha_1|\leq 1$. So, the last inequality must be an equality and $\alpha$ is a maximizer of~\eqref{e.variat.psiga}. Since $\gamma$ is not constantly zero, Lemma~\ref{l.variat.psiga} implies that $\alpha_t$ satisfies~\eqref{e.identity.alpha} for all $t$ sufficiently close to $1$. 
Due to $|\alpha_t|\leq1$, we can see that $\alpha_t$ converges a.s.\ to some $\alpha_1$ as $t\nearrow1$.
Since $\alpha$ is a martingale (so that $\alpha_t =\EE [\alpha_1\,|\, \mcl F_t]$), we obtain the desired uniqueness of maximizers and the characterization.

It remains to verify the claim in~\eqref{e.claim}.
Fix any $t\in(0,1)$, we can apply the same argument as in the proof of~\cite[Lemma~2.2]{mourrat2025uninverting} (with $[0,1]$ and $\phi$ therein substituted with $[0,t]$ and $\Psi_\gamma(t,\cdot +h)$, respectively) to get\footnote{Setting $\phi = \Psi_\gamma(t,\cdot +h)$, we have $\phi^*(x) = \Psi_\gamma^*(t,x) -hx$ for every $x\in\R$.}
\begin{align*}
    \Psi_\gamma(0,h) = \EE \Ll[B_t(\bar\alpha) - \Psi_\gamma^*(t,\bar\alpha_t)\ \big|\ \mathcal{F}_0\Rr]
\end{align*}
where $\Psi_\gamma^*(t,\cdot)$ is defined as in~\eqref{e.Psi^*(t,x)=} and $\bar\alpha$ is given as before. 
It is important to point out that a key property used in~\cite[Lemma~2.2]{mourrat2025uninverting} is the strict convexity of $\phi$ therein, which corresponds here to the strict convexity of $\Psi_\gamma(t,\cdot)$ implied by Lemma~\ref{l.d_xPsi_odd_increase}.
Using~\eqref{e.Phi^*<Phi^*+||} in the above display, we get
\begin{align*}
    \Psi_\gamma(0,h) \leq \EE \Ll[B_t(\bar\alpha)\Rr] - \EE \Ll[\Psi_\gamma^*(t,\bar\alpha_1)\Rr]+ \EE \Ll[\Ll|\bar\alpha_t-\bar\alpha_1\Rr|\Rr].
\end{align*}
Sending $t\nearrow1$, the first term on the right converges to $\EE \Ll[B_1(\bar\alpha)\Rr]$, as is clear from the definition of $B_t$. Due to~$|\bar\alpha_1|\leq 1$ and Lemma~\ref{l.Phi^*_cvg}, the second term vanishes in the limit. The martingale convergence also nullifies the last term in the limit. Therefore, together with~\eqref{e.Psi_gamma>supE[B]}, we obtain~\eqref{e.claim} as claimed, completing the proof.
\end{proof}

We need the following observation that, for each bounded measurable function $f:[0,1]\to\R$, we have
\begin{align}\label{e.compute_inf_gamma}
    \inf_{\gamma\in\mcl M_s}\int_0^1 f(t)\gamma(t)\d t &= s \int_0^1 f(t)\d t + \inf_{\gamma\in\mcl M_0}\int_0^1 f(t)\gamma(t)\d t \notag
    \\ 
    &= s \int_0^1 f(t)\d t + 
    \begin{cases}
          0 ,\  &\text{if $\int_r^1  f(t)\d t\geq0$, $\forall r\in[0,1]$},
          \\
          -\infty ,\  &\text{otherwise}.
    \end{cases}
\end{align}
Here, the first equality follows from the fact that $\gamma\mapsto s+\gamma$ is a bijection from $\mcl M_0$ to $\mcl M_s$, and the second equality follows from the fact that if $\gamma' : [0,1] \to \R_+$ is a nonnegative bounded measurable function, then $t \mapsto \int_0^t \gamma' \in \mcl M_0$ and 
\begin{align*}  
\int_0^1 f(t) \int_0^t \gamma'(r)  \d r \,   \d t = \int_0^1 \gamma'(r) \int_r^1 f(t) \d t \, \d r.
\end{align*}

\begin{lemma}[Interchanging $\inf$ and $\sup$]\label{l.inter_inf_sup}
Let
\begin{align*}
    \msf K_0 := \Ll\{\Ll(\EE\Ll[h \al_0+\alpha_1\int_0^1\sqrt{\xi''(t)}\d W_t\Rr],\ \Ll(\EE[\alpha^2_t]\Rr)_{t\in[0,1]}\Rr)\ \bigg|\ \alpha\in\bmart_1\Rr\}
\end{align*}
and let $\msf K$ be the closure of the convex hull of $\msf K_0$ with respect to the topology of the space $\R\times L^1([0,1])$. Then,
\begin{align}
    \lim_{N\to \infty} \frac{1}{sN} \log \E \Ll[ \exp(sN L_N) \Rr] &= \sup_{(\chi,g)\in\msf K}\inf_{\gamma\in\mcl M_s}\Ll(\chi+\frac{1}{2}\int_0^1 \xi''(t)\gamma(t)\Ll(g(t) -t \Rr)\d t\Rr)\notag
    \\
    &=\sup_{(\chi,g)\in\msf K}\Ll\{\chi+\frac{s}{2}\int_0^1 \xi''(t)\Ll(g(t) -t \Rr)\d t\ \Big|\ g\in \msf G\Rr\}. \label{e.inter_inf_sup}
\end{align}
\end{lemma}
\begin{proof}
Write $G(\gamma,\chi,g) = \chi+\frac{1}{2}\int_0^1 \xi''(t)\gamma(t)\Ll(g(t) -t \Rr)\d t$.
Combining Theorem~\ref{t.laplace} and Lemma~\ref{l.prog.to.mart}, we get
\begin{align*}
    \lim_{N\to \infty} \frac{1}{sN} \log \E \Ll[ \exp(sN L_N) \Rr] = \inf_{\gamma\in\mcl M_s}\sup_{(\chi,g)\in\msf K_0} G(\gamma,\chi,g).
\end{align*}
Since $G(\gamma,\cdot,\cdot)$ is affine for each fixed $\gamma$, setting $\msf K_1$ to be the convex hull of $\msf K_0$, we have
\begin{align}\label{e.lim=infsup_K_1}
    \lim_{N\to \infty} \frac{1}{sN} \log \E \Ll[ \exp(sN L_N) \Rr] = \inf_{\gamma\in\mcl M_s}\sup_{(\chi,g)\in\msf K_1} G(\gamma,\chi,g).
\end{align}

To further replace $\msf K_1$ by its closure $\msf K$, we show that $G$ is continuous on $\cM_s \times \R \times \{g\in L^1:\:|g|\leq 1\}$ in the topology of $L^1\times\R\times L^1$ where $L^1 =L^1([0,1])$. The continuity in $\chi$ is obvious and we focus on that in $\gamma$ and $g$. For any $\gamma,\tilde \gamma \in\cM_s$ and $g,\tilde g\in L^1$ bounded by $1$, we have
\begin{align}
    &\Ll|\int_0^1 \xi''(t)\gamma(t)\Ll(g(t) -t \Rr)\d t - \int_0^1 \xi''(t)\tilde\gamma(t)\Ll(\tilde g(t) -t \Rr)\d t\Rr| \notag
    \\
    &\leq \Ll|\int_0^1 \xi''(t)\Ll(\gamma(t)-\tilde \gamma(t)\Rr)\Ll(\tilde g(t) -t\Rr)\d t\Rr| +\Ll| \int_0^1 \xi''(t)\gamma(t)\Ll(g(t) -\tilde g(t) \Rr)\d t\Rr| . \label{e.G_cts_1}
\end{align}
The first term on the right is bounded by $2\xi''(1)\|\gamma-\tilde \gamma\|_{L^1}$ and the second term is bounded by, for any $\delta\in(0,1)$, 
\begin{align}\label{e.G_cts_2}
    \xi''(1)\Ll(\gamma(1-\delta)\int_0^{1-\delta}\Ll|g(t)-\tilde g(t)\Rr|\d t + 2\int_{1-\delta}^1 \gamma(t)\d t\Rr).
\end{align}
Fixing $\gamma$ and $g$, we consider $\tilde\gamma$ and $\tilde g$ approaching them. Due to $\gamma\in L^1$, we can choose $\delta$ sufficiently close to $1$ to make $\int_{1-\delta}^1 \gamma(t)\d t$ sufficiently small and then letting $\tilde g$ tend to $g$ in $L^1$ to make $\int_0^{1-\delta}\Ll|g(t)-\tilde g(t)\Rr|\d t$ arbitrarily small. From this and the bound on the first term~\eqref{e.G_cts_1}, we get the announced continuity of $G$.
In particular, from~\eqref{e.lim=infsup_K_1}, this continuity yields
\begin{align}\label{e.lim=infsup_K}
    \lim_{N\to \infty} \frac{1}{sN} \log \E \Ll[ \exp(sN L_N) \Rr] = \inf_{\gamma\in\mcl M_s}\sup_{(\chi,g)\in\msf K} G(\gamma,\chi,g).
\end{align}

Then, we show that $\msf K$ is compact. As $\alpha\in \bmart_1$, given any sequence $(\chi^\pnp, g^\pnp)_{n\in\N}$ in $\msf K_0$, we have that $\chi^\pnp$ and $\|g^\pnp\|_{L^\infty}$ are bounded uniformly over $n$. We can thus extract a subsequence along which $\chi^\pnp$ and $g^\pnp(t)$ converge for every $t\in\Q\cap[0,1]$. By this and the monotonicity of $g^\pnp$, this subsequence converges in $\R\times L^1$. Hence, $\msf K_0$ is pre-compact and its closed convex hull $\msf K$ is compact by~\cite[Theorem~5.35]{aliprantis2006infinite}.

Recall the continuity of $G$. Also, notice that $G(\gamma,\cdot,\cdot)$ is affine for each fixed $\gamma$ and that $G(\cdot,\chi,g)$ is affine for each fixed $(\chi,g)$. These together with the convexity of $\cM_s$ and the compactness and convexity of $\msf K$ allow us to apply the minimax theorem~\cite{fan1953minimax,sion1958minimax} to interchange $\inf$ and $\sup$ in~\eqref{e.lim=infsup_K} and get the first equality in~\eqref{e.inter_inf_sup}. The second equality there follows from~\eqref{e.compute_inf_gamma}.
\end{proof}

\begin{lemma}[Existence of maximizing martingale]\label{l.lim=sup_mart_achieve}
There exists an admissible probability space $\mathscr P$ (see Definition~\ref{d.admissible}) such that \eqref{e.lim=sup_mart_achieve} is valid and the supremum therein is achieved.
\end{lemma}

\begin{proof}
Notice that ``$\geq$'' in~\eqref{e.lim=sup_mart_achieve} is an immediate consequence of Lemma~\ref{l.inter_inf_sup}. It remains to show the reverse inequality.

We start with choosing the standard Wiener space $\msc W$ as underlying probability space. To emphasize this, we write $\bmart_1(\msc W)$, $\msf K(\msc W)$, and $\msf K_0(\msc W)$.
Later, we will consider probability spaces other than $\msc W$ and extend this notation in the obvious way.

Let $(\chi_n,g_n)\in\msf K(\msc W)$ be a maximizing sequence for~\eqref{e.inter_inf_sup}, where we can assume that $g_n\in\msf G$ for each $n$. As shown in the proof of Lemma~\ref{l.inter_inf_sup}, $\msf K(\msc W)$ is compact in the topology of $\R\times L^1([0,1])$. Since the functional $(\chi,g)\mapsto \chi + \frac{s}{2}\int_0^1\xi''(t)(g(t)-t)\d t$ is continuous and the set $\msf G$ (see~\eqref{e.G=}) is closed in $L^1([0,1])$, we can extract a convergent subsequence from $(\chi_n,g_n)_{n\in\N}$ with some limit $(\chi,g)\in \msf K(\msc W)$ such that $g\in \msf G$ and the supremum in~\eqref{e.inter_inf_sup} is achieved:
\begin{align}\label{e.limL=chi+s}
    \lim_{N\to \infty} \frac{1}{sN} \log \E \Ll[ \exp(sN L_N) \Rr] = \chi+\frac{s}{2}\int_0^1 \xi''(t)\Ll(g(t) -t \Rr)\d t.
\end{align}

Since $\msf K(\msc W)$ is the closed convex hull $\msf K_0(\msc W)$, we can find, for $n\in\N$, weights $\Ll(c^\pnp_i\Rr)_{1\leq i\leq m_n}$ with $\sum_{i=1}^{m_n}c^\pnp_i=1$ for some $m_n\in\N$ and martingales $\Ll(\alpha^{\pnp,i}\Rr)_{1\leq i\leq m_n}$ on $\msc W$ such that $(\chi_n, g_n)$ from the convex hull of $\msf K_0(\msc W)$ converges to $(\chi, g)$ where
\begin{align}\label{e.chi_n=,g_n=}
\begin{split}
    \chi_n&=\sum_{i=1}^{m_n} c^\pnp_i \EE\Ll[h\alpha^{\pnp,i}_0 + \alpha^{\pnp,i}_1\int_0^1\sqrt{\xi''(t)}\d W_t\Rr]
    \\
    g_n(t) & = \sum_{i=1}^{m_n} c^\pnp_i \EE \Ll[\Ll(\alpha^{\pnp,i}_t\Rr)^2\Rr].
\end{split}
\end{align}
Let $(\mcl F_t)_{t\in[0,1]}$ be the natural filtration associated with $\msc W$ and let $\bar{\msc W}$ be a richer probability space with an augmented filtration $\Ll(\bar{\mcl F}_t\Rr)_{t\in[0,1]}$ such that the Wiener process $W$ is still supported on $\bar{\msc W}$ and that there is a $\bar{\mcl F}_0$-measurable random variable distributed uniformly over $[0,1]$. Then, for each $n$, we can define a $\bar{\mcl F}_0$-measurable random variable $N^\pnp$ with law $\PP(N^\pnp=i) = c^\pnp_i$ for $i\in\{1,\dots,m_n\}$. On $\bar{\msc W}$, we consider the martingale $\alpha^\pnp = \alpha^{\pnp, N^\pnp} \in \bmart_1\Ll(\bar{\msc W}\Rr)$ and we have
\begin{align}\label{e.chi_n=,g_n=2}
    \chi_n=\EE\Ll[h\alpha^\pnp_0 + \alpha^\pnp_1\int_0^1\sqrt{\xi''(t)}\d W_t\Rr] \quad\text{and}\quad 
    g_n(t)  =\EE \Ll[\Ll(\alpha^\pnp_t\Rr)^2\Rr].
\end{align}

Since $\alpha^\pnp$ takes values in $[-1,1]$, the family of random variables, by Prokhorov's theorem and passing to a subsequence, we have that $\big(h,W,\big(\alpha^\pnp_t\big)_{t\in\Q\cap[0,1]}\big)$ converges in law to some $\Ll(h,W,\Ll(\alpha_t\Rr)_{t\in\Q\cap[0,1]}\Rr)$ on some probability space $\msc P$ in the following sense. For every $\ell\in\N$,  $t_1,\cdots, t_\ell \in \Q\cap [0,1]$, and bounded continuous function $G:\R\times C([0,1])\times \R^\ell\to\R$, we have
\begin{align}\label{e.limE[G(W...)=}
    \lim_{n\to\infty} \EE \Ll[G\Ll(h,W, \alpha^\pnp_{t_1},\dots ,\alpha^\pnp_{t_\ell}\Rr)\Rr] = \EE \Ll[G\Ll(h,W, \alpha_{t_1},\dots ,\alpha_{t_\ell}\Rr)\Rr]. 
\end{align}
We then extend $\alpha$ by setting $\alpha_t= \EE \Ll[\alpha_1 |\bar{\mcl F}_t\Rr]$ for $t\in[0,1]$, turning it into a martingale in $\msc P$. By Skorokhod's representation theorem, we may assume that the convergence is also pointwise a.s. By this and the boundedness of these martingales, recalling the convergence of $(\chi_n,g_n)$ to $(\chi,g)$, we can get
\begin{align}\label{e.chi=,g=}
    \chi=\EE\Ll[h\alpha_0 + \alpha_1\int_0^1\sqrt{\xi''(t)}\d W_t\Rr] \quad\text{and}\quad 
    g(t)  =\EE \Ll[\Ll(\alpha_t\Rr)^2\Rr]
\end{align}
for $t\in \Q\cap [0,1]$. Since both sides are monotone, this identity can be extended to almost every $t\in[0,1]$. This along with~\eqref{e.limL=chi+s} gives ``$\leq$'' in~\eqref{e.lim=sup_mart_achieve} with $\bmart_1 = \bmart_1(\msc P)$.
As commented in beginning, this completes the proof.
\end{proof}

\begin{proof}[Proof of Theorem~\ref{t.laplace_uninvert}]
Let $\bar{\msc P}$ be the probability space in Lemma~\ref{l.lim=sup_mart_achieve} which is in general different from $\msc P$. For clarity, we write $\bmart_1(\bar{\msc P})$ and $\bmart_1 (\msc P)$ for the corresponding spaces of martingales. 

Let $\bar\alpha\in \bmart_1(\bar{\msc P})$ maximize~\eqref{e.lim=sup_mart_achieve} allowed by Lemma~\ref{l.lim=sup_mart_achieve}. Let $\bar\gamma\in \cM_s$ be the unique minimizer of~\eqref{e.laplace} prescribed by Theorem~\ref{t.laplace}. 
For brevity, we write 
\begin{align*}
    \Gamma(\gamma,\alpha):=\EE \Ll[ h \al_0 + \al_1 \int_0^1 \sqrt{\xi''(t)}\, \d W_t + \frac{1}{2}\int_{0}^{1}\xi''(t)\ga(t) \Ll(\alpha^{2}_t-t\Rr)\, \d t \Rr] 
\end{align*}
and
$\ell := \lim_{N\to\infty} \frac{1}{sN}\log\E \Ll[\exp(sNL_N)\Rr]$. Then, Theorem~\ref{t.laplace} together with Lemma~\ref{l.prog.to.mart} yields
\begin{align*}
    \ell = \inf_{\gamma\in\cM_s}\sup_{\alpha\in\bmart_1(\bar{\msc P})}\Gamma(\gamma,\alpha) = \sup_{\alpha\in\bmart_1(\bar{\msc P})}\Gamma(\bar\gamma,\alpha),
\end{align*}
while Lemma~\ref{l.lim=sup_mart_achieve} along with~\eqref{e.compute_inf_gamma} gives
\begin{align*}
    \ell = \sup_{\alpha\in\bmart_1(\bar{\msc P})}\inf_{\gamma\in\cM_s}\Gamma(\gamma,\alpha) = \inf_{\gamma\in\cM_s}\Gamma(\gamma,\bar\alpha).
\end{align*}
Hence, for every $\gamma \in\mcl M_s$ and $\alpha\in\bmart_1(\bar{\msc P})$, we have
\begin{align*}
    \Gamma(\bar\gamma,\alpha)\leq \ell \quad \text{and}\quad \ell \leq \Gamma(\gamma,\bar\alpha).
\end{align*}
From this, we can deduce $\ell= \Gamma(\bar\gamma,\bar\alpha)$ and
\begin{align}\label{e.Gamma<Gamma<Gamma}
    \Gamma(\bar\gamma,\alpha)\leq \Gamma(\bar\gamma,\bar\alpha)\leq \Gamma(\gamma,\bar\alpha).
\end{align}
The first optimality in~\eqref{e.Gamma<Gamma<Gamma} and Lemma~\ref{l.prog.to.mart} imply that $\bar\alpha$ is specified by~\eqref{e.identity.alpha} with this $\bar\gamma$ and $W$ on $\bar{\msc P}$ and, in particular, that $\bar\alpha$ is measurable with respect to the filtration generated by the Wiener process.

Now, we turn to the probability space $\msc P$ as in the statement of the theorem. We let $\hat \alpha$ be given by the same relation in~\eqref{e.identity.alpha} but with $W$ on $\msc P$. Since $\bar \alpha$ and $\hat \alpha$ have the same law, $\ell =\inf_{\gamma\in\cM_s}\Gamma(\gamma,\bar\alpha)$ together with the computation in~\eqref{e.compute_inf_gamma} implies that ``$\leq $'' holds in~\eqref{e.lim=sup_mart_achieve}. The reverse inequality ``$\geq $'' follows from Lemma~\ref{l.inter_inf_sup}. Therefore, we have verified~\eqref{e.lim=sup_mart_achieve}.

Lastly, by the same argument above over the probability space $\bar{\msc P}$ leading to the fact that $\bar\alpha$ satisfies~\eqref{e.identity.alpha}, we can see that $\hat\alpha$ also satisfies~\eqref{e.identity.alpha} with the same~$\bar\gamma$. This completes the uniqueness part.
\end{proof}

\section{Derivation of the large-deviation principle}
\label{s.derivation}

For $s\geq0$ and $\alpha\in\bmart_1$ with some fixed underlying probability space $\msc P$, we consider the functional
\begin{align}\label{e.g(s,alpha)=}
    g(s,\alpha):=s\EE\Ll[h\alpha_0 + \alpha_1\int_0^1\sqrt{\xi''(t)}\d W_t\Rr]+\frac{s^2}{2}\int_0^1 \xi''(t)\Ll(\EE\Ll[\alpha_t^2\Rr] -t \Rr)\d t.
\end{align}
For each $s\geq0$, we define the limiting cumulant generating function of $L_N$ by
\begin{align}\label{e.Lambda(s)=}
\begin{split}
    \Lambda(s)&:=\lim_{N\to\infty} \frac{1}{N}\log \E \Ll[\exp\Ll(sNL_N\Rr)\Rr] \\
&=\sup_{\alpha\in\bmart_1}\Ll\{g(s,\alpha) \ \big|\ \EE\Ll[\alpha^2_\bullet\Rr]\in \msf G\Rr\}
\end{split}
\end{align}
where $\msf G$ is defined in~\eqref{e.G=} and the second equality follows from Theorem~\ref{t.laplace_uninvert}. We also define
\begin{align}\label{e.Lambda^*=}
    \Lambda^*(r) := \sup_{s\in\R_+}\Ll\{sr- \Lambda(s)\Rr\},\quad\forall r\in\R.
\end{align}

\begin{theorem}\label{t.LDP}
For every Borel measurable subset $A\subset [\gs,\infty)$, we have
\begin{align}\label{e.t.LDP}
    -\inf_{A^\circ}\Lambda^*\leq \liminf_{N\to\infty}\frac{1}{N}\log \P\Ll\{L_N\in A\Rr\}\leq \limsup_{N\to\infty}\frac{1}{N}\log \P\Ll\{L_N\in A\Rr\} \leq  -\inf_{\bar A}\Lambda^*
\end{align}
where $A^\circ$ and $\bar A$ are the interior (relative to $[\gs,\infty)$) and the closure of $A$, respectively, and $\Lambda^*$ given as in~\eqref{e.Lambda^*=} is continuous and increasing on $[\gs,\infty)$.

Moreover, we have $\Lambda^*(\gs)=0$ and, for every $r>\gs$,
\begin{align}\label{e.Lambda^*(r)=inf_mart}
    0<\Lambda^*(r)= \inf_{\alpha\in\bmart_1}\Ll\{\frac{\Ll(r-\EE\Ll[h\alpha_0+\alpha_1\int_0^1\sqrt{\xi''(t)}\d W_t\Rr]\Rr)^2}{2\int_0^1\xi''(t)\Ll(\EE\Ll[\alpha^2_t\Rr]-t\Rr)\d t}\ \Bigg|\ \EE\Ll[\alpha^2_\bullet\Rr]\in \msf G\Rr\}
\end{align}
where the probability space $\msc P$ underlying $\bmart_1$ can be chosen to be any admissible one  (see Definition~\ref{d.admissible}).
For every $r>\gs$, there is a unique maximizer $s$ of~\eqref{e.Lambda^*=}, and this maximizer satisfies $r=\Lambda'(s)$ and $s>0$. The infimum in~\eqref{e.Lambda^*(r)=inf_mart} is achieved at a unique $\alpha$ that is also the unique maximizer of $\frac{\Lambda(s)}{s}$ given by Theorem~\ref{t.laplace_uninvert} at this $s$. Lastly, for the said optimal $s$ and $\alpha$, we have
\begin{align}\label{e.Lambda^*(r)=..s..alpha}
    \Lambda^*(r)= \frac{s^2}{2}\int_0^1\xi''(t)\Ll(\EE\Ll[\alpha^2_t\Rr]-t\Rr)\d t.
\end{align}
\end{theorem}

\subsection{Analytic properties of \texorpdfstring{$\Lambda$}{Λ} and \texorpdfstring{$\Lambda^*$}{Λ*}}

\begin{lemma}[Basic properties]\label{l.basic_Lambda}
The function $\Lambda:\R_+\to\R$ is convex, locally Lipschitz, and grows super-linearly. The function $\Lambda^*:\R\to\R$ is convex and continuous on $\R$ and is increasing on $\R_+$. 
\end{lemma}

\begin{proof}

Since $g(\cdot,\alpha)$ is convex for each $\alpha$, $\Lambda$ is convex. Due to $|\alpha_t|\leq 1$, $g(\cdot,\alpha)$ is locally Lipschitz uniformly over $\alpha\in\bmart_1$, which implies that $\Lambda$ is also locally Lipschitz. Fix $\mathbf{1}$ to be the martingale constantly equal to $1$. Then, $\frac{\Lambda(s)}{s}\geq \frac{g(s,\mathbf{1})}{s}= h +\frac{s}{2}\int_0^1\xi''(t)(1-t)\d t$, where the last integral is strictly positive. Hence, $\Lambda$ grows super-linearly as $s$ tends to $+\infty$, which implies that $\Lambda^*$ is finite everywhere. The convexity of $\Lambda^*$ is clear from the definition. Since any convex function that is finite everywhere is continuous (e.g., see~\cite[Proposition~2.9]{dominguez2024book}), we deduce the continuity of $\Lambda^*$. Since the supremum in $\Lambda^*$ is taken over $s\in\R_+$, we see that $\Lambda^*$ is increasing on $\R_+$.
\end{proof}


\begin{lemma}[Continuous differentiability of $\Lambda$]\label{l.Lambda'(s)=}
Let $\msc P$ be admissible as in Definition~\ref{d.admissible}.
The function $\Lambda :\R_+\to\R$ is continuously differentiable, and at every $s\in\R_+$ its derivative is given by
\begin{align}\label{e.Lambda'(s)=}
    \Lambda'(s) = \partial_s g(s,\alpha)= \EE\Ll[h\alpha_0 + \alpha_1\int_0^1\sqrt{\xi''(t)}\d W_t\Rr]+s\int_0^1 \xi''(t)\Ll(\EE\Ll[\alpha_t^2\Rr] -t \Rr)\d t
\end{align}
where $\alpha$ is the maximizer given by Theorem~\ref{t.laplace_uninvert} at $s$.
\end{lemma}

In view of the variational formula in~\eqref{e.Lambda(s)=}, this lemma follows from the envelope theorem (see, e.g.,~\cite[Theorem~2.21]{dominguez2024book}) adapted to our setting. However, since the space $\bmart_1$ is infinite-dimensional and compactness issues are delicate, we provide a complete proof for clarity.

\begin{proof}
We first show that $\Lambda$ is differentiable by providing an upper and then a lower bound. After this, we show that the derivative is continuous.

\medskip
\noindent\textit{Step~1.}
We show the upper bound in the definition of differentiability. Fix any $s\in\R_+$ and let $\alpha$ be the maximizer for $\Lambda(s)$ as in~\eqref{e.Lambda(s)=}. Due to $|\alpha_t|\leq1$ and that $g(\cdot,\alpha)$ is quadratic, we can find a constant $C$ depending only on $\xi''$ such that, for every $s'\in\R_+$,
\begin{align*}
    \Ll|g(s',\alpha)-g(s,\alpha) - (s'-s)\partial_s g(s,\alpha)\Rr|\leq C|s'-s|^2.
\end{align*}
Using this, $g(s',\alpha)\leq \Lambda(s')$ and $g(s,\alpha)=\Lambda(s)$, we get
\begin{align*}
    \Lambda(s')\geq \Lambda(s) + (s'-s)\partial_s g(s,\alpha) - C|s'-s|^2.
\end{align*}

\medskip
\noindent\textit{Step~2.}
We derive a matching lower bound and deduce the differentiability. Fix any $s\in\R_+$ and let $\alpha$ be the maximizer for $\Lambda(s)$. It suffices to show that for every sequence $(s_n)_{n\in\N}$ in $\R_+\setminus\{s\}$ converging to $s$, we can always extract a subsequence $(s_{n_k})_{k\in\N}$ such that
\begin{align}\label{e.l.Lambda_diff_upper}
    \limsup_{k\to\infty}\frac{1}{\Ll|s_{n_k}-s\Rr|}\Ll(\Lambda(s_{n_k})-\Lambda(s) -(s_{n_k}-s)\partial_s g(s,\alpha)\Rr)=0.
\end{align}
Indeed, this along with the first step implies that $\Lambda$ is differentiable with the announced derivative. Now, let us verify this. For each $n$, let $\alpha^\pnp$ be the maximizer for $\Lambda(s_n)$. Since $\alpha^\pnp$ takes values in $[-1,1]$, using the same argument as around the display~\eqref{e.limE[G(W...)=}, we get that $(W,(\alpha^\pnp_t)_{t\in\Q\cap[0,1]})$ converges along a subsequence $(n_k)_{k\in\N}$ in law to some $(W,(\bar\alpha_t)_{t\in\Q\cap[0,1]})$ in a possibly different probability space $\bar{\msc P}$, where $\bar\alpha_t$ is a martingale. 
For convenience, we denote the subsequence $(n_k)_{k\in\N}$ still by $(n)_{n\in\N}$.

Let $\bar g$ be the corresponding functional as $g$ but with underlying probability space $\bar{\msc P}$. We also write $\EE_{\msc P}$ and $\EE_{\bar{\msc P}}$ as expectations with respect to the ambient probability spaces displayed. Using Skorokhod's representation, we may turn the convergence in law to a.s.\ convergence on another common probability space. Using this and the boundedness of these martingales, we can show that
\begin{align}\label{e.cvg_alpha^pnp_bar_alpha}
\begin{split}
    \lim_{n\to\infty}\EE_{\msc P}\Ll[h\alpha^\pnp_0 + \alpha^\pnp_1\int_0^1\sqrt{\xi''(t)}\d W_t\Rr] & = \EE_{\bar{\msc P}}\Ll[h\bar\alpha_0 + \bar\alpha_1\int_0^1\sqrt{\xi''(t)}\d W_t\Rr],
    \\
    \lim_{n\to\infty}\int_0^1 \xi''(t)\Ll(\EE_{\msc P}\Ll[\Ll(\alpha_t^\pnp\Rr)^2\Rr] -t \Rr)\d t &= \int_0^1 \xi''(t)\Ll(\EE_{\bar{\msc P}}\Ll[\bar\alpha_t^2\Rr] -t \Rr)\d t.
\end{split}
\end{align}
Then, we use this to show
\begin{align}\label{e.alpha=baralpha}
    \alpha \stackrel{\d}{=}\bar\alpha.
\end{align}
First, from~\eqref{e.cvg_alpha^pnp_bar_alpha} and the definition of $g$ in~\eqref{e.g(s,alpha)=}, we have $\lim_{n\to\infty}g(s_n,\alpha^\pnp) = \bar g(s,\bar\alpha)$.
The continuity of $\Lambda$ by Lemma~\ref{l.basic_Lambda} together with $g(s_n,\alpha^\pnp)= \Lambda(s_n)$ implies $\lim_{n\to\infty}g(s_n,\alpha^\pnp) = \Lambda(s)$. Hence, we have $\Lambda(s)= \bar g(s,\bar\alpha)$ and recall $\Lambda(s)=g(s,\alpha)$. The characterization of the unique maximizer in Theorem~\ref{t.laplace_uninvert} ensures that $\alpha$ and $\bar\alpha$ have the same law, verifying~\eqref{e.alpha=baralpha}.

Now, let us return to proving~\eqref{e.l.Lambda_diff_upper} and recall the shorthand $(n)_{n\in\N}$ for $(n_k)_{k\in\N}$.
We start with
\begin{align*}
    g\Ll(s_n,\alpha^\pnp\Rr)- g\Ll(s,\alpha^\pnp\Rr) =\int_0^1(s_n-s)\partial_sg\Ll(ts_n+(1-t)s,\alpha^\pnp\Rr)\d t.
\end{align*}
Then, using $\Lambda(s_n)= g\Ll(s_n,\alpha^\pnp\Rr)$ and $\Lambda(s)\geq g\Ll(s,\alpha^\pnp\Rr)$ on the left-hand side and comparing the right-hand side with $(s_n-s)\partial_s \bar g(s,\bar\alpha)$, we get
\begin{align*}
    &\Lambda(s_n)-\Lambda(s) \leq (s_n-s)\partial_s \bar g(s,\bar\alpha) 
    \\
    & + |s_n-s|\int_0^1 \Ll|\partial_sg\Ll(ts_n+(1-t)s,\alpha^\pnp\Rr) - \partial_s \bar g(s, \bar\alpha)\Rr|\d t.
\end{align*}
Due to~\eqref{e.cvg_alpha^pnp_bar_alpha} and~\eqref{e.g(s,alpha)=}, the integral vanishes as $n\to\infty$. Using~\eqref{e.alpha=baralpha}, we also have $ \partial_s \bar g(s, \bar\alpha) = \partial_s g(s,\alpha)$. These together yield~\eqref{e.l.Lambda_diff_upper}, completing the proof of differentiability with derivative as in~\eqref{e.Lambda'(s)=}.

\medskip
\noindent\textit{Step~3.}
We show that the derivative is continuous in $s$. Again, fix any $s$ and let $\alpha$ be the maximizer for $\Lambda(s)$. It suffices to show that for any $(s_n)_{n\in\N}$ converging to $s$ we can extract subsequence $(s_{n_k})_{k\in\N}$ such that $\lim_{k\to\infty} \partial_s g(s_{n_k},\alpha^{(n_k)}) = \partial_s g(s,\alpha)$, where $\alpha^\pnp$ is the maximizer for $\Lambda(s_n)$. We argue as in Step~2. Writing $(n_k)_{k\in\N}$ still as $(n)_{n\in\N}$, we extract a subsequential limit $\bar\alpha$ on a possibly different probability space $\bar{\msc P}$ as before. Then, we still have~\eqref{e.cvg_alpha^pnp_bar_alpha} and~\eqref{e.alpha=baralpha}. In view of the definition of $g$ in~\eqref{e.g(s,alpha)=}, these two imply the desired convergence. Now, the proof is complete.
\end{proof}

Using Lemma~\ref{l.Lambda'(s)=} and Theorem~\ref{t.laplace_uninvert}, we can compute that
\begin{align}\label{e.Lambda'(0)=gs}
    \Lambda'(0)= \gs= \lim_{N\to\infty}\E [L_N] \geq 0
\end{align}
where the inequality is due to $\E [L_N]\geq N^{-1}\E[H_N(\sigma)]=N^{-1} h\sum_{i=1}^N\sigma_i\geq 0$ for a suitable $\sigma$.
In the next lemma, we show that to the right of zero, the function $\Lambda$ may have a flat piece, and is then strictly convex. We will also show in Corollary~\ref{c.bars.zero} below that the flat piece is necessarily absent when $h \neq 0$. 

\begin{lemma}[Strict convexity of $\Lambda$ after some linear portion]\label{l.cond_strict_convex}
Let $\underline s: = \sup\{s\in\R_+:\: \Lambda(s)=s \, \gs\}$. Then, we have that $\Lambda(s)=s \, \gs$ for every $s\in[0,\underline s]$ and that $\Lambda$ is strictly convex on $(\underline s ,\infty)$ and satisfies $\Lambda(s)> s \, \gs$ for every $s>\underline s$.

Moreover, the following holds, where we denote by $\alpha^s$ the unique optimizer of $\frac{\Lambda(s)}{s}$ (given by Theorem~\ref{t.laplace_uninvert}) for each $s\geq0$. For every $s\in(0,\underline s]$, we have $\alpha^s=\alpha^0$ and
\begin{align}\label{e.cond_strict_convex}
    \int_0^1\xi''(t) \Ll(\EE\Ll[\Ll(\alpha^s_t\Rr)^2\Rr]-t\Rr)\d t = \int_0^1\xi''(t) \Ll(\EE\Ll[\Ll(\alpha^0_t\Rr)^2\Rr]-t\Rr)\d t = 0.
\end{align}
For every $s>\underline s$, we have
\begin{align}\label{e.cond_strict_convex_2}
    \int_0^1\xi''(t) \Ll(\EE\Ll[\Ll(\alpha^s_t\Rr)^2\Rr]-t\Rr)\d t >0.
\end{align}
\end{lemma}

Since $\Lambda(0)=0$ and $\Lambda$ grows super-linearly, we have $\underline s \in[0,\infty)$.
In view of~\eqref{e.cond_strict_convex}, we can see that if $\int_0^1\xi''(t) \Ll(\EE\Ll[\Ll(\alpha^0_t\Rr)^2\Rr]-t\Rr)\d t>0$, then $\underline s=0$ and thus $\Lambda$ is strictly convex on $\R_+$.

\begin{proof}
We first give a characterization of the condition $\Lambda(s)=s \, \gs$ and then prove the statement on different intervals. Throughout, let $\alpha^s$ be the unique optimizer for $\Lambda(s)$ and we express $g$ in~\eqref{e.g(s,alpha)=} as $g(s,\alpha) = s g_1(\alpha)+\frac{s^2}{2}g_2(\alpha)$ where
\begin{align*}
    g_1(\alpha):= \EE\Ll[h\alpha_0 + \alpha_1\int_0^1\sqrt{\xi''(t)}\d W_t\Rr]\quad \text{and}\quad g_2(\alpha) :=\int_0^1 \xi''(t)\Ll(\EE\Ll[\alpha_t^2\Rr] -t \Rr)\d t.
\end{align*}
We often use that $g_2(\alpha)\geq 0$ if $\EE[\alpha^2_\bullet]\in\msf G$.
In particular, we have $g_2(\alpha^s)\geq 0$ for every $s\geq0$.

\medskip
\noindent{Step 1.}
For $s>0$, we show that
\begin{align}\label{e.equiv_Lambda(s)=sgs}
    \int_0^1\xi''(t) \Ll(\EE\Ll[\Ll(\alpha^s_t\Rr)^2\Rr]-t\Rr)\d t=0\qquad \Longleftrightarrow \qquad \Lambda(s)=s \, \gs.
\end{align}
We notice that the l.h.s.\ is equivalent to $g_2(\alpha^s)=0$.
To see ``$\Longrightarrow$'', we start with
\begin{align*}
    s \, \gs=sg_1(\alpha^0)\leq s g_1(\alpha^0)+\frac{s^2}{2}g_2(\alpha^0)
    \leq \Lambda(s) = s g_1(\alpha^s)+\frac{s^2}{2}g_2(\alpha^s) = s g_1(\alpha^s) \leq s \, \gs
\end{align*}
yielding the r.h.s. Next, to show ``$\Longleftarrow$'', we start with
\begin{align*}
    s \, \gs=sg_1(\alpha^0)\leq s g_1(\alpha^0)+\frac{s^2}{2}g_2(\alpha^0)
    \leq \Lambda(s) =  s \, \gs
\end{align*}
which implies that $\alpha^0$ is a maximizer for $\Lambda(s)$. By the uniqueness of maximizers given by Theorem~\ref{t.laplace_uninvert}, we must have $\alpha^s=\alpha^0$. Using this, $s>0$, and that the first inequality in the above display is in fact an equality, we can get $g_2(\alpha^0)=g_2(\alpha^s)=0$. Hence, we have verified~\eqref{e.equiv_Lambda(s)=sgs}.

Notice that in the proof of ``$\Longleftarrow$'', we have also shown
\begin{align}\label{e.=>alpha^s=alpha^0}
    \Lambda(s)=s \, \gs\qquad \Longrightarrow \qquad \alpha^s=\alpha^0
\end{align}
which will be needed below.

\medskip
\noindent{Step 2.}
We show that $ \Lambda(s)=s \, \gs$ for $s\in [0,\underline s]$. By continuity, we have $ \Lambda(\underline s)=\underline s \, \gs$.
Then, convexity implies $\Lambda(s)\leq \frac{s}{\underline s}\Lambda(\underline s)+ \frac{\underline s -s }{\underline s }\Lambda(0)=s \, \gs$.
By convexity and~\eqref{e.Lambda'(0)=gs}, we have $\Lambda(s)\geq \Lambda(0)+s\Lambda'(0)=s \, \gs$. These give the desired result, which along with~\eqref{e.equiv_Lambda(s)=sgs} and~\eqref{e.=>alpha^s=alpha^0} also implies the ``moreover'' part for $s\in(0,\underline s]$ and~\eqref{e.cond_strict_convex}.

\medskip
\noindent{Step 3.}
We show that $\Lambda$ is strictly convex on $(\underline s ,\infty)$ and $\Lambda(s)> s \, \gs $ for every $s>\underline s$. First, the latter must hold, because otherwise (namely, $\Lambda(s)= s \, \gs$) the same argument as in Step~2 would imply $\Lambda(s') = s'\gs$ for every $s'\in[0,s]$ contradicting the definition of $\underline s$. Now, fix any distinct $s_0,s_1>\underline s$ and any $\lambda\in(0,1)$. Set $s_\lambda = (1-\lambda)s_0+\lambda s_1$ and we have $\Lambda(s_\lambda) = s_\lambda g_1(\alpha^{s_\lambda})+ \frac{s_\lambda^2}{2}g_2(\alpha^{s_\lambda})$. Due to~\eqref{e.equiv_Lambda(s)=sgs}, we have $g_2(\alpha^{s_\lambda})>0$. Using the strict convexity of the function $s\mapsto \frac{s^2}{2}$, we have
\begin{align*}
    \Lambda(s_\lambda)<(1-\lambda)g(s_0,\alpha^{s_\lambda})+\lambda g(s_1,\alpha^{s_\lambda}) \leq (1-\lambda)\Lambda(s_0)+\lambda \Lambda(s_1),
\end{align*}
verifying the strict convexity.

Due to $\Lambda(s)>s \, \gs$ for every $s>\underline s$, using~\eqref{e.equiv_Lambda(s)=sgs}, we get~\eqref{e.cond_strict_convex_2}.
\end{proof}

\begin{lemma}[Maximizers of $\Lambda^*$]\label{l.max_Lambda^*}
Let $\underline s$ be given as in Lemma~\ref{l.cond_strict_convex}.
The maximizers of $\Lambda^*(\gs)$ are exactly the real numbers in $[0,\underline s]$ and we have $\Lambda^*(\gs)=0$.
For every $r>\gs$, $\Lambda^*(r)$ admits a unique maximizer $s\in(\underline s, \infty)$, and this maximizer is such that $r=\Lambda'(s)$. We also have $\Lambda^*(r)>0$ for every $r>\gs$.
\end{lemma}

\begin{proof}
Recall $\Lambda^*(\gs)=\sup_{s\in\R_+}\{s \, \gs - \Lambda(s)\}$. Lemma~\ref{l.cond_strict_convex} implies that $s \, \gs - \Lambda(s)=0$ for $s\in[0,\underline s]$ and $s \, \gs - \Lambda(s)>0$ for $s>\underline s$, which yields the first part of the lemma.

Then, let $r>\gs$. Since $\Lambda$ grows super-linearly (see Lemma~\ref{l.cond_strict_convex}), $\Lambda^*(r)$ admits maximizers. Let $s$ be any such maximizer. If $s=0$, then $\Lambda^*(r)=0-\Lambda(0)=0\geq s'r-\Lambda(s')$ for every $s'\in \R_+$, implying $r\leq \Lambda'(0)=\gs$ due to~\eqref{e.Lambda'(0)=gs}, which is not possible. Hence, we must have $s>0$. Then, the maximality condition implies that $r=\Lambda'(s)$. Since $r>\gs$, we must have $s>\underline s$, because otherwise we would have $\Lambda'(s)=\gs <r$ (see Lemma~\ref{l.cond_strict_convex}). Hence, $\Lambda^*(r)$ achieves the maximum over $(\underline s,\infty)$. Then, the strict convexity in Lemma~\ref{l.cond_strict_convex} implies the uniqueness of maximizers.

Let $r>\gs$ and suppose that $\Lambda^*(r)\leq 0$. Then, we have $s r-\Lambda(s)\leq 0$ for every $s\in\R_+$. Due to~\eqref{e.Lambda'(0)=gs}, we must have $r\leq \Lambda'(0)=\gs$, reaching a contradiction. Hence, the last part of this lemma is verified.
\end{proof}

\subsection{Inverting variations inside \texorpdfstring{$\Lambda^*$}{Λ*}}

In view of~\eqref{e.Lambda(s)=} and~\eqref{e.Lambda^*=}, $\Lambda^*$ admits a $\sup\inf$ formula. We want to invert the order. We first record
\begin{align}\label{e.sup(sr-schi-s^2a/2=}
    \chi<r,\quad a \geq 0 \quad \Longrightarrow \quad \sup_{s\in\R_+}\Ll\{sr-s\chi - \frac{s^2}{2}a\Rr\} =  \frac{(r-\chi)^2}{2a}.
\end{align}
Indeed, if $a=0$, then the supremum is equal to $+\infty$ as $r-\chi>0$; if $a>0$, then the maximum is achieved at $s= \frac{r-\chi}{a}$. To invert $\Lambda^*$, we follow similar arguments as in Lemmas~\ref{l.inter_inf_sup} and~\ref{l.lim=sup_mart_achieve}.

\begin{lemma}[Interchanging $\sup$ and $\inf$]\label{l.Lambda^*_inter_inf_sup}
Let
\begin{align*}
    \msf K_0 := \Ll\{\Ll(\EE\Ll[h \al_0+\alpha_1\int_0^1\sqrt{\xi''(t)}\d W_t\Rr],\ \Ll(\EE[\alpha^2_t]\Rr)_{t\in[0,1]}\Rr)\ \bigg|\ \alpha\in\bmart_1\Rr\}
\end{align*}
and let $\msf K$ be the closure of the convex hull of $\msf K_0$ with respect to the topology of the space $\R\times L^1([0,1])$. Then, for every $r>\gs$, we have
\begin{align}
    \Lambda^*(r) &= \inf_{(\chi,g)\in\msf K}\sup_{s\in\R_+}\Ll\{sr-s\chi-\frac{s^2}{2}\int_0^1 \xi''(t)\Ll(g(t) -t \Rr)\d t\ \Big|\ g\in\msf G\Rr\}\notag
    \\
    &=\inf_{(\chi,g)\in\msf K}\Ll\{ \frac{(r-\chi)^2}{2\int_0^1\xi''(t)(g(t)-t)\d t}  \ \bigg|\ g\in \msf G\Rr\}. \label{e.Lambda^*_inter_inf_sup}
\end{align}
\end{lemma}
\begin{proof}
Write $G(s,\chi,g) = sr - s\chi -\frac{s^2}{2}\int_0^1 \xi''(t)\Ll(g(t) -t \Rr)\d t$.
From~\eqref{e.Lambda(s)=} and \eqref{e.Lambda^*=}, we get
\begin{align*}
    \Lambda^*(r) = \sup_{s\in\R_+}\inf_{(\chi,g)\in\msf K_0}\Ll\{ G(s,\chi,g)\ \big|\ g\in\msf G\Rr\}.
\end{align*}
Since for each fixed $s$, $G(s,\cdot,\cdot)$ is affine, setting $\msf K_1$ to be the convex hull of $\msf K_0$, we have
\begin{align*}
    \Lambda^*(r) = \sup_{s\in\R_+}\inf_{(\chi,g)\in\msf K_1}\Ll\{ G(s,\chi,g)\ \big|\ g\in\msf G\Rr\}.
\end{align*}

To further replace $\msf K_1$ by its closure $\msf K$, we can use the same arguments around~\eqref{e.G_cts_1} and~\eqref{e.G_cts_2} to see that $\msf G$ is continuous on $\R_+\times[0,\gs] \times \{g\in L^1:\:|g|\leq 1\}$ in the topology of $\R\times\R\times L^1$ where $L^1 =L^1([0,1])$. This continuity together with the obvious closedness of $\msf G$ in $L^1$ yields
\begin{align}\label{e.Lambda^*=infsup_K}
    \Lambda^*(r) = \sup_{s\in\R_+}\inf_{(\chi,g)\in\msf K}\Ll\{ G(s,\chi,g)\ \big|\ g\in\msf G\Rr\}.
\end{align}
Due to the formula for $\gs$ given in Theorem~\ref{t.laplace_uninvert} (at $s=0$), we have
\begin{align}\label{e.E[..alpha..]<gs}
    \EE\Ll[h\alpha_0 + \alpha_1\int_0^1\sqrt{\xi''(t)}\d W_t\Rr] \leq \gs 
\end{align}
and thus
\begin{align}\label{e.chi<gs<r}
    \chi \leq \gs < r,\qquad \forall (\chi,g)\in \msf K.
\end{align}
In particular, we can see $\msf K\subset [0,\gs]\times \Ll\{|g|\leq 1\Rr\}$.

As argued below~\eqref{e.lim=infsup_K}, $\msf K$ is compact under the norm topology of $\R\times L^1$, which implies that $\msf K\cap (\R\times \msf G)$ is compact. Also, notice that $G(s,\cdot,\cdot)$ is affine for each fixed $s$ and that $G(\cdot,\chi,g)$ is concave for each fixed $(\chi,g)$. These together with the convexity of $\R_+$ and the compactness and convexity of $\msf K\cap (\R\times \msf G)$ allow us to apply the minimax theorem~\cite{fan1953minimax,sion1958minimax} to interchange $\sup$ and $\inf$ in~\eqref{e.Lambda^*=infsup_K} and get the first equality in~\eqref{e.Lambda^*_inter_inf_sup}. The second equality there follows from~\eqref{e.chi<gs<r}, $g\in\msf G$, and~\eqref{e.sup(sr-schi-s^2a/2=}.
\end{proof}

The following uses a similar argument as in Lemma~\ref{l.lim=sup_mart_achieve}.

\begin{lemma}[Existence of minimizing martingale]\label{l.Lambda^*=inf_mart_achieve}
There is an admissible probability space $\mathscr P$ (see Definition~\ref{d.admissible}) such that, for every $r>\gs$,
\eqref{e.Lambda^*(r)=inf_mart} is valid and the supremum therein is achieved.
\end{lemma}

\begin{proof}
Notice that ``$\leq$'' in~\eqref{e.Lambda^*(r)=inf_mart} is an immediate consequence of Lemma~\ref{l.Lambda^*_inter_inf_sup}. It remains to show the reverse inequality.

We start with choosing the standard Wiener space $\msc W$ as underlying probability space. To emphasize this, we write $\bmart_1(\msc W)$, $\msf K(\msc W)$, and $\msf K_0(\msc W)$.
Later, we will consider probability spaces other than $\msc W$ and extend this notation in the obvious way.

Let $(\chi_n,g_n)\in\msf K(\msc W)$ be a minimizing sequence for~\eqref{e.Lambda^*_inter_inf_sup}, where we can assume that $g_n\in\msf G$ for each $n$. 
Hence, for sufficiently large $n$, we have
\begin{align*}
    \frac{(r-\chi_n)^2}{2\int_0^1\xi''(t)(g_n(t)-t)\d t} \leq \Lambda^*(r)+1
\end{align*}
which along with~\eqref{e.chi<gs<r} implies
\begin{align*}
    \int_0^1\xi''(t)(g_n(t)-t)\d t \geq \frac{(r-\gs)^2}{2\Lambda^*(r)+2}=:c_0>0.
\end{align*}
Then, we take $\msf G_0=\msf G\cap\Ll\{\int_0^1\xi''(t)(g(t)-t)\d t\geq c_0\Rr\}$.

As shown in the proof of Lemma~\ref{l.inter_inf_sup}, $\msf K(\msc W)$ is compact in the topology of $\R\times L^1$ (where $L^1=L^1([0,1])$. Since the functional $(\chi,g)\mapsto \frac{(r-\chi)^2}{2\int_0^1\xi''(t)(g(t)-t)\d t}$ is continuous on $\R\times \msf G_0$ equipped with the $\R\times L^1$-topology and the set $\msf G_0$ is closed in $L^1$, we can extract a convergent subsequence from $(\chi_n,g_n)_{n\in\N}$ with some limit $(\chi,g)\in \msf K(\msc W)$ such that $g\in \msf G_0$ and the infimum in~\eqref{e.Lambda^*_inter_inf_sup} is achieved:
\begin{align}\label{e.Lambda^*(r)=..chi..g}
    \Lambda^*(r) = \frac{(r-\chi)^2}{2\int_0^1\xi''(t)(g(t)-t)\d t}.
\end{align}

By the same argument around~\eqref{e.chi_n=,g_n=} and~\eqref{e.chi_n=,g_n=2}, we can find an enriched probability space $\bar{\msc W}$ from $\msc W$ and martingales $\alpha^\pnp \in\bmart_1(\bar{\msc W})$ such that
\begin{align*}
    \chi_n=\EE\Ll[h\alpha^\pnp_0 + \alpha^\pnp_1\int_0^1\sqrt{\xi''(t)}\d W_t\Rr] \quad\text{and}\quad 
    g_n(t)  =\EE \Ll[\Ll(\alpha^\pnp_t\Rr)^2\Rr].
\end{align*}
Then, by the same argument around~\eqref{e.limE[G(W...)=} and~\eqref{e.chi=,g=}, we can pass the above equations to the limit and find a probability space $\msc P$ and a martingale $\alpha\in\bmart_1(\msc P)$ such that
\begin{align*}
    \chi=\EE\Ll[h\alpha_0 + \alpha_1\int_0^1\sqrt{\xi''(t)}\d W_t\Rr] \quad\text{and}\quad 
    g(t)  =\EE \Ll[\Ll(\alpha_t\Rr)^2\Rr].
\end{align*}
This along with~\eqref{e.Lambda^*(r)=..chi..g} gives ``$\geq$'' in~\eqref{e.Lambda^*(r)=inf_mart} with $\bmart_1 = \bmart_1(\msc P)$.
As commented in the beginning, this completes the proof.
\end{proof}

\subsection{One-sided large-deviation principle}

Since we only consider $\Lambda(s)$ for $s\geq0$, we cannot derive the most general form of the LDP with testing sets being arbitrary subsets of $\R$. Instead, we need a ``one-sided'' version, stated in the following proposition, whose proof is elementary but included here for completeness. 
The setting is general and independent of the rest of the paper.

\begin{proposition}[One-sided LDP]\label{p.one-sided_LDP}

Let $(L_N)_{N\in\N}$ be a sequence of random variables and, for each $N\in\N$, denote by $\Lambda_N:\R_+\to\R$ the function:
\begin{align}\label{e.Lambda_N=}
    \Lambda_N(s) := \frac{1}{N}\log \E \exp\Ll(s N L_N\Rr),\qquad\forall s\in \R_+.
\end{align}
Suppose that there exists a continuously differentiable function $\Lambda:\R_+\to\R$ such that $\Lambda_N$ converges pointwise to $\Lambda$ on $\R_+$, and that $m := \Lambda'(0) \geq 0$. Then, for every Borel measurable subset $A\subset [m,\infty)$, we have
\begin{align}
\label{e.l.one-sided_LDP}
    -\inf_{A^{\circ_\R}}\Lambda^*\leq \liminf_{N\to\infty}\frac{1}{N}\log \P\Ll\{L_N\in A\Rr\}\leq \limsup_{N\to\infty}\frac{1}{N}\log \P\Ll\{L_N\in A\Rr\} \leq  -\inf_{\bar A}\Lambda^*
\end{align}
where $A^{\circ_\R}$ and $\bar A$ are the interior relative to $\R$ and the closure of $A$, respectively, and $\Lambda^*:\R\to\R\cup\{+\infty\}$ is given as in~\eqref{e.Lambda^*=}.

If, in addition, $\Lambda^*$ is right-continuous at $m$ (i.e.\ $\lim_{r\searrow m}\Lambda^*(r)=\Lambda^*(m)$), then the lower bound in~\eqref{e.l.one-sided_LDP} can be improved to $-\inf_{A^\circ}\Lambda^*$, where $A^\circ$ is the interior relative to $[m,\infty)$ of $A$.
\end{proposition}

\begin{proof}

\textit{Upper bound. }
For any $x,s\geq 0$, Chebyshev's inequality gives
\begin{align*}
    \P\Ll\{L_N\geq x\Rr\}\leq e^{-sNx}\E e^{sN L_N} = e^{-sNx+N\Lambda_N(s)}.
\end{align*}
Since $s$ can be taken arbitrary over $\R_+$, we have
\begin{align}\label{e.limsup1/NlogP(L_N>x)}
    \limsup_{N\to\infty}\frac{1}{N}\log\P\Ll\{L_N\geq x\Rr\}\leq \inf_{s\in\R_+}\Ll\{-sx + \Lambda(s)\Rr\} = - \Lambda^*(x),\quad\forall x \geq 0.
\end{align}
Let $A\subset [0,\infty)$ and set $x_A = \inf A\in\bar A\subset [0,\infty)$. Notice that $\P\Ll\{L_N\in A\Rr\}\leq \P\Ll\{L_N\geq x_A\Rr\}$ and that $\inf_{\bar A} \Lambda^*\geq \Lambda^*(x_A)$ since $\Lambda^*$ is increasing on $\R_+$ (due to~\eqref{e.Lambda^*=}). Using these and~\eqref{e.limsup1/NlogP(L_N>x)}, we get the desired upper bound
\begin{align*}    \limsup_{N\to\infty}\frac{1}{N}\log \P\Ll\{L_N\in A\Rr\}\leq -\Lambda^*(x_A) \leq - \inf_{\bar A}  \Lambda^*.
\end{align*}
Notice that this upper bound holds for any Borel measurable $A\subset[0,\infty)$. Due to $m\geq0$, this clearly holds for any $A\subset [m,\infty)$.

\textit{Lower bound.}
Since $\Lambda'$ is continuous and increasing and $\Lambda'(0)=m$, the image of $\Lambda'(\R_+)$  is an interval with endpoint $m\leq d \leq \infty$. Fix any $x\in (m,d)$ and let $s\in\R_+$ be such that $x= \Lambda'(s)$. Notice that we must have $s>0$.
For every $\eps>0$ and $\mu\in(0,s]$, we have
\begin{align}\label{e.Ee^sNL_None_L_N>x+eps}
    \E e^{sNL_N}\mathds{1}_{\{L_N> x+\eps\}} \leq e^{-N\mu(x+\eps)}\E e^{(s+\mu)NL_N} = e^{-N\mu(x+\eps)+N\Lambda_N(s+\mu)}
\end{align}
and similarly
\begin{align}\label{e.Ee^sNL_None_L_N<x-eps}
    \E e^{sNL_N}\mathds{1}_{\{L_N< x-\eps\}} \leq e^{-N\mu(x-\eps)}\E e^{(s-\mu)NL_N} = e^{-N\mu(x-\eps)+N\Lambda_N(s-\mu)}.
\end{align}
Since $\Lambda$ is continuously differentiable, we have, for sufficiently small $\mu$,
\begin{align*}
    \Ll|\Lambda(s+\mu)-\Lambda(s)-\mu\Lambda'(s)\Rr| + \Ll|\Lambda(s-\mu)-\Lambda(s)+\mu\Lambda'(s)\Rr|\leq \tfrac{\eps\mu}{4}.
\end{align*}
Fix any such $\mu$. Recall that $\Lambda'(s)=x$. For $N$ sufficiently large, we have
\begin{align*}
    \Ll|\Lambda_N(s+\mu)-\Lambda_N(s)-\mu x\Rr| + \Ll|\Lambda_N(s-\mu)-\Lambda_N(s)+\mu x\Rr|\leq \tfrac{\eps\mu}{2}.
\end{align*}
Using this in~\eqref{e.Ee^sNL_None_L_N>x+eps} and~\eqref{e.Ee^sNL_None_L_N<x-eps}, we get
\begin{align*}
    \E e^{s NL_N}\mathds{1}_{\{L_N\not\in [x-\eps,x+\eps]\}}\leq 2 e^{-N\eps\mu/2}\E e^{s NL_N}.
\end{align*}
Hence, for sufficiently large $N$, we have
\begin{align*}
    \E e^{s NL_N}\mathds{1}_{\{L_N\in [x-\eps,x+\eps]\}}\geq \tfrac{1}{2}\E e^{s NL_N}
\end{align*}
and thus
\begin{align*}
    \P\{L_N\in [x-\eps,x+\eps]\} 
    \geq \tfrac{1}{2}e^{-sN(x+\eps)}\E e^{sNL_N}=\tfrac{1}{2}e^{-sN(x+\eps)+N\Lambda_N(s)}.
\end{align*}
Notice that the left-hand side gets smaller as $\eps$ decreases. Hence, letting $N$ tend to infinity and then $\eps$ to zero on the right, we get
\begin{align}\label{e.liminf1/NlogP(x-eps,x+eps)}
    \liminf_{N\to\infty} \frac{1}{N}\log \P\{L_N\in [x-\eps,x+\eps]\} \geq -sx + \Lambda(s) \geq -\Lambda^*(x)
\end{align}
for $x\in (m,d)$ and every $\eps>0$.

We want to extend~\eqref{e.liminf1/NlogP(x-eps,x+eps)} to every $x \in (m,\infty)$. Due to $d = \sup_{\R_+}\Lambda'$, we can verify that $\Lambda^*(x)=\infty$ if $x\not\in (m,d]$ and thus~\eqref{e.liminf1/NlogP(x-eps,x+eps)} holds at $x\in(d,\infty)$. It remains to verify~\eqref{e.liminf1/NlogP(x-eps,x+eps)} at $x=d$ when $d<\infty$. 
In this case, for every $\eps>0$, we can find $x\in(m,d)$ and $\eps'>0$ such that $[x-\eps',x+\eps']\subset [d-\eps,d+\eps]$. Since $\Lambda^*$ is increasing on $\R_+$ (due to~\eqref{e.Lambda^*=}), we have $\Lambda^*(d) \geq \Lambda^*(x)$. Using these and~\eqref{e.liminf1/NlogP(x-eps,x+eps)} at $x$, we get~\eqref{e.liminf1/NlogP(x-eps,x+eps)} at $d$.

For any $A\subset [m,\infty)$, let $x\in A^{\circ_\R}$ (taken relative to $\R$ and thus $m\not\in A^{\circ_\R}$). We can find sufficiently small $\eps$ so that $[x-\eps,x+\eps]\subset A$. Now, using the extended~\eqref{e.liminf1/NlogP(x-eps,x+eps)}, we get
\begin{align*}
    \liminf_{N\to\infty} \frac{1}{N}\log \P\{L_N\in A\}\geq -\Lambda^*(x).
\end{align*}
Optimizing over $x\in A^{\circ_\R}$, we obtain the desired lower bound. Notice that the lower bound holds trivially when $A^{\circ_\R}=\emptyset$.

\textit{The additional part.}
We always have $\inf_{A^\circ}\Lambda^*\leq \inf_{A^{\circ_\R}}\Lambda^*$.
They differ only when $A^\circ\setminus A^{\circ_\R}=\{m\}$ and $\inf_{A^\circ}\Lambda^*= \Lambda^*(m)$. However, in this case, we can find a sequence in $A^{\circ_\R}$ converging to $m$ from the right, which along with the right-continuity of $\Lambda^*$ at $m$ ensures that $\inf_{A^{\circ_\R}}\Lambda^* = \Lambda^*(m)$. Therefore, under the additional assumption, we must have $\inf_{A^\circ}\Lambda^*=\inf_{A^{\circ_\R}}\Lambda^*$.
\end{proof}

\subsection{Proof of Theorem~\ref{t.LDP}}

Clearly, the theorem can be divided into three parts and we prove each of them separately.

\begin{proof}[Proof of~\eqref{e.t.LDP} and the regularity of $\Lambda^*$]
Let $\Lambda_N$ be defined as in~\eqref{e.Lambda_N=} with $L_N$ given as in~\eqref{e.L_N=}. Theorem~\ref{t.laplace_uninvert} implies that $\Lambda_N$ converges to $\Lambda$ as in~\eqref{e.Lambda(s)=} pointwise on $\R_+$. Lemma~\ref{l.Lambda'(s)=} ensures that $\Lambda$ is continuously differentiable and~\eqref{e.Lambda'(0)=gs} gives $\Lambda'(0)=\gs \geq0$. Hence, we can apply Proposition~\ref{p.one-sided_LDP} with $m=\gs$ to get~\eqref{e.t.LDP}. Lemma~\ref{l.basic_Lambda} gives that $\Lambda^*$ is continuous and increasing on $[\gs,\infty)$.
This continuity also allows us to use the additional part in Proposition~\ref{p.one-sided_LDP} so that $A^\circ$ in~\eqref{e.t.LDP} is taken relative to $[\gs,\infty)$.
\end{proof}

\begin{proof}[Proof of~\eqref{e.Lambda^*(r)=inf_mart} and the uniqueness of minimizers]

Fix any $r>\gs$.
Lemma~\ref{l.max_Lambda^*} gives $\Lambda^*(r)>0$ in~\eqref{e.Lambda^*(r)=inf_mart} and we focus on the equality in~\eqref{e.Lambda^*(r)=inf_mart}.

We follow similar steps as in the proof of Theorem~\ref{t.laplace_uninvert}.
Let $\bar{\msc P}$ be the probability space in Lemma~\ref{l.lim=sup_mart_achieve} which is in general different from $\msc P$ in the statement of the theorem to prove. For clarity, we write $\bmart_1(\bar{\msc P})$ and $\bmart_1 (\msc P)$ for the corresponding spaces of martingales. Recall $\msf G$ from~\eqref{e.G=} and we write $\bmart_{1,\msf G}(\msc P)$ to be the subcollection of martingales $\alpha$ satisfying $\EE[\alpha^2_\bullet]\in\msf G$ and similarly for $\bmart_{1,\msf G}(\bar{\msc P})$.

Let $\bar\alpha\in \bmart_{1,\msf G}(\bar{\msc P})$ minimize~\eqref{e.Lambda^*(r)=inf_mart}, allowed by Lemma~\ref{l.Lambda^*=inf_mart_achieve}. Let $\bar s\in\R_+$ be the unique maximizer of~\eqref{e.Lambda^*=} given by Lemma~\ref{l.max_Lambda^*}, which satisfies $r=\Lambda'(\bar s)$ and $\bar s>0$ as desired in Theorem~\ref{t.LDP}.

In the following, we write $\Gamma(s,\alpha) = rs - g(s,\alpha)$ for $g(s,\alpha)$ appearing in~\eqref{e.g(s,alpha)=}.
Then, the definition of $\Lambda^*(r)$ as in~\eqref{e.Lambda^*=} along with~\eqref{e.Lambda(s)=} gives
\begin{align*}
    \Lambda^*(r) = \sup_{s\in\R_+}\inf_{\alpha\in\bmart_{1,\msf G}(\bar{\msc P})}\Gamma(s,\alpha) = \inf_{\alpha\in\bmart_{1,\msf G}(\bar{\msc P})}\Gamma(\bar s,\alpha),
\end{align*}
while Lemma~\ref{l.Lambda^*=inf_mart_achieve} along with~\eqref{e.sup(sr-schi-s^2a/2=} (allowed by~\eqref{e.E[..alpha..]<gs}) gives
\begin{align*}
    \Lambda^*(r)  = \inf_{\alpha\in\bmart_{1,\msf G}(\bar{\msc P})}\sup_{s\in\R_+}\Gamma(s,\alpha) = \sup_{s\in\R_+}\Gamma(s,\bar\alpha).
\end{align*}
Hence, for every $s\in\R_+$ and $\alpha\in\bmart_{1,\msf G}(\bar{\msc P})$, we have
\begin{align*}
    \Lambda^*(r)\leq \Gamma(\bar s,\alpha) \quad \text{and}\quad \Gamma(s,\bar\alpha)\leq \Lambda^*(r).
\end{align*}
From this, we can deduce $\Lambda^*(r)= \Gamma(\bar s,\bar\alpha)$ and
\begin{align}\label{e.Lambda_Gamma<Gamma<Gamma}
    \Gamma(s,\bar\alpha)\leq \Gamma(\bar s,\bar\alpha)\leq \Gamma(\bar s,\alpha).
\end{align}
The second optimality in~\eqref{e.Lambda_Gamma<Gamma<Gamma} implies that $\bar\alpha$ maximizes $\Lambda(s)$, which by Theorem~\ref{t.laplace_uninvert} gives that $\bar\alpha$ is the unique maximizer satisfying~\eqref{e.identity.alpha} as described therein. In particular, $\bar \alpha$ is measurable with respect to the filtration generated by the Wiener process.

Now, we turn to the probability space $\msc P$ as in the statement of Theorem~\ref{t.LDP}. We let $\hat \alpha$ be given by the same relation in~\eqref{e.identity.alpha} but with $W$ now on $\msc P$. Now since $\bar \alpha$ and $\hat \alpha$ have the same law, $\Lambda^*(r) =\sup_{s\in\R_+}\Gamma(s,\bar\alpha)$ together with the computation in~\eqref{e.sup(sr-schi-s^2a/2=} implies that ``$\geq $'' holds in~\eqref{e.Lambda^*(r)=inf_mart}. The reverse inequality ``$\leq $'' follows from Lemma~\ref{l.Lambda^*_inter_inf_sup}. Therefore, we have verified the equality in~\eqref{e.Lambda^*(r)=inf_mart}.

Lastly, applying the same argument above (leading to~\eqref{e.Lambda_Gamma<Gamma<Gamma}) to $\msc P$ and $\hat \alpha$, we can see that $\hat\alpha$ is uniquely determined via~\eqref{e.identity.alpha} given by Theorem~\ref{t.laplace_uninvert} at $\bar s$. This gives the uniqueness of minimizers.
\end{proof}

\begin{proof}[Proof of~\eqref{e.Lambda^*(r)=..s..alpha}]
Fix any $r>\gs$ and let $s$ and $\alpha$ be the optimizers as described in Theorem~\ref{t.LDP}. By $r=\Lambda'(s)$ and maximality, we have $\Lambda^*(r)= s\Lambda'(s)-\Lambda(s)$. 
Since $\alpha$ is the maximizer given in Theorem~\ref{t.laplace_uninvert} at $s$, we can see that $\alpha$ maximizes~\eqref{e.Lambda(s)=}.
Hence, we have $\Lambda(s)= g(s,\alpha)$ and $\Lambda'(s)=\partial_sg(s,\alpha)$ by Lemma~\ref{l.Lambda'(s)=}. Using the expression of $g(s,\alpha)$ in~\eqref{e.g(s,alpha)=}, we can thus compute $s\Lambda'(s)-\Lambda(s)$ and verify~\eqref{e.Lambda^*(r)=..s..alpha}.
\end{proof}

Now, the proof of Theorem~\ref{t.LDP} is complete.

\section{Necessary and sufficient condition for quadratic behavior}
\label{s.quadratic}

The main goal of this section is to prove Theorem~\ref{t.quadratic}. 
To understand the effect of the external field on the Parisi measure, namely the minimizer $\gamma$ for $\gs$ given in~\eqref{e.def.gs}\footnote{Recall that in this section we take $h \in \mathbb{R}$ to be deterministic, so that we have $\Psi_\gamma(0,h)$ instead of $\mathbb{E}_{g_1,h_1}\Psi_\gamma(0,g_1 + h_1)$ in~\eqref{e.def.gs}.}, we recall~\cite[Proposition~3]{chen2017fluctuations} as follows. We view $\mathrm{d}\gamma$ as a finite positive measure on $[0,1]$ and denote its support by $\supp \mathrm{d}\gamma$.

\begin{proposition}[\cite{chen2017fluctuations}]
Let $h\in\R$, let $\gamma$ be the minimizer of the right-hand side in~\eqref{e.def.gs}, and let $X$ be the solution to \eqref{e.AC.SDE}.
For every $q\in\supp\d\gamma$, we have
\begin{align}
    \E \Ll(\partial_x\Psi_\gamma(q,X_q)\Rr)^2 & =q, \label{e.E(dPsi)^2=q}
    \\
    \xi''(q) \E\Ll(\partial_x^2\Psi_\gamma(q,X_q)\Rr)^2 & \leq 1. \label{e.xi''(q)E(d^2Psi)<1}
\end{align}
\end{proposition}

This proposition allows one to deduce the following well-known result. Heuristically, when $h=0$, the spin distribution is balanced, and hence $0$ belongs to the support of $\mathrm{d}\gamma$.

\begin{proposition}\label{p.h=0,0in_supp}
Let $h=0$ and let $\gamma$ be the minimizer of the right-hand side in~\eqref{e.def.gs}. Then, we have $0\in\supp\d\gamma$.
\end{proposition}

Such a result has been established for mixed-$p$ spin models at finite temperature in~\cite{auffinger2015properties}. At zero temperature, it is proved at the beginning of~\cite[Section~3.2]{chen2018energy} using~\cite[Theorem~5]{chen2014chaos}. Although both works are formulated for mixed even spin models, the argument in fact applies to general mixed-$p$ spin models. For completeness and clarity, we choose to reproduce the proof here. We begin with an analytic result, which is a simplification of~\cite[Theorem~5]{chen2014chaos}.

\begin{lemma}\label{l.chi_psi}
Let $h\in\R$, let $\gamma$ be the minimizer of the right-hand side in~\eqref{e.def.gs}, and let $q=\inf\supp\d\gamma$. For each $u\in[0,q]$, let $\chi_1(u)$ and $\chi_2(u)$ be two centered jointly Gaussian random variables with covariances
\begin{align*}
    \E\chi_1^2(u) = \E\chi_2^2(u) = \xi'(q) \quad\text{and}\quad \E\chi_1(u)\chi_2(u)= \xi'(u)
\end{align*}
and define
\begin{align*}
    \psi(u) = \E \partial_x\Psi_\gamma(q,h+\chi_1(u))\partial_x\Psi_\gamma(q,h+\chi_2(u)).
\end{align*}
Then, $|\psi'(u)|<1$ for every $u\in[0,q)$.
\end{lemma}

\begin{proof}
We can realize $\chi_1$ and $\chi_2$ by choosing $g_0,g_1,g_2$ to be i.i.d.\ standard Gaussian and set, for $i\in\{1,2\}$,
\begin{align*}
    \chi_i(u) = \sqrt{\xi'(u)}g_0 + \sqrt{\xi'(q)-\xi'(u)}g_i.
\end{align*}
As a consequence of $q=\inf\supp\d\gamma$, we have $\gamma=0$ on $[0,q]$ and thus $X_q = h+ \int_0^q\sqrt{\xi''(t)}\d W_t$ in view of~\eqref{e.AC.SDE}. Hence, we have
\begin{align}\label{e.chi=X_q}
    h+\chi_1(u)\stackrel{\d}{=}h+\chi_2(u)\stackrel{\d}{=}X_q,\quad\forall u\in[0,q].
\end{align}
For brevity, we write $F=\partial_x\Psi_\gamma(q,\cdot)$ and often omit $u$ in writing $\chi_1$ and $\chi_2$.
Then, using the symmetry between $\chi_1$ and $\chi_2$ and Gaussian integration by parts, we can differentiate $\psi$ and compute
\begin{align*}
    \psi'(u) &= 2\E \Ll[F'(h+\chi_1)F(h+\chi_2)\chi'_2\Rr]
    \\
    &= \xi''(u)\E \Ll[F'(h+\chi_1)F(h+\chi_2)\Ll(\tfrac{g_0}{\sqrt{\xi'(u)}}-\tfrac{g_1}{\sqrt{\xi'(q)-\xi'(u)}}\Rr)\Rr]
    \\
    & = \xi''(u)\E \Ll[F'(h+\chi_1)F'(h+\chi_2)\Rr].
\end{align*}
For $i\in\{1,2\}$, we define $Z_i = \sqrt{\xi''(q)}F'(h+\chi_i)$. Since $F'=\partial_x^2\Psi_\gamma(q,\cdot)>0$ due to Lemma~\ref{l.d_xPsi_odd_increase}, we have $Z_i>0$. By~\eqref{e.xi''(q)E(d^2Psi)<1} and~\eqref{e.chi=X_q}, we also have $\E Z_i^2\leq 1$. Using this, $\xi''(u)\leq \xi''(q)$, and the Cauchy--Schwarz inequality, we get
\begin{align*}
    |\psi'(u)|\leq \E Z_1Z_2\leq \Ll(\E Z_1^2\Rr)^\frac{1}{2}\Ll(\E Z_2^2\Rr)^\frac{1}{2}\leq 1.
\end{align*}
Suppose that $|\psi'(u)|=1$ for some $u\in[0,q)$. Then, we have $\E Z_1Z_2=1$ and thus 
\begin{align*}
    \E (Z_1-Z_2)^2 = \E Z_1^2 + \E Z_2^2 - 2 \E Z_1 Z_2 \leq 0,
\end{align*}
implying that $Z_1=Z_2$ a.s. However, since $F'$ is not constant, we can find open intervals $O_1$ and $O_2$ such that $\sup_{O_1} F' < \inf_{O_2} F'$. Due to $u<q$, $\chi_1(u)$ and $\chi_2(u)$ are not identical and the event that $h+\chi_1(u)\in O_1$ and $h+\chi_2(u)\in O_2$ has positive probability. This implies $Z_1(u)\neq Z_2(u)$ and reaches a contradiction. Therefore, we must have $|\psi'(u)|<1$ for every $u\in[0,q)$.
\end{proof}

\begin{proof}[Proof of Proposition~\ref{p.h=0,0in_supp}]
Set $q=\inf\supp\d\gamma$.
Let $\chi_1$, $\chi_2$, and $\psi$ be given as in Lemma~\ref{l.chi_psi} at this $q$ and $h=0$. Notice that $\chi_1(q)=\chi_2(q)$ a.s. By this, \eqref{e.chi=X_q}, and~\eqref{e.E(dPsi)^2=q}, we get $q=\psi(q)$. On the other hand, due to $\xi'(0)=0$, $\chi_1(0)$ and $\chi_2(0)$ are independent and thus we have $\psi(0)= \Ll(\E \partial_x\Psi_\gamma(q,\chi_1(0))\Rr)^2=0$ since $\partial_x\Psi_\gamma(q,\cdot)$ is odd by Lemma~\ref{l.d_xPsi_odd_increase}. Therefore, both $0$ and $q$ are fixed points of $\psi$ on $[0,q]$. Suppose $q\neq 0$. Then, the mean value theorem gives some $u\in(0,q)$ such that $\psi'(u)=1$, which contradicts Lemma~\ref{l.chi_psi}. Hence, we must have $q=0$.
\end{proof}

As announced in \eqref{e.equiv.conditions}, we now show the following equivalence.
\begin{lemma}\label{l.h=0_int=0}
Let $\alpha$ be the optimal martingale from Theorem~\ref{t.gs}. We have 
\begin{align}\label{e.h=0_int=0}
    h =0 \qquad\text{if and only if}\qquad  \int_0^1\xi''(t)\Ll(\EE\Ll[\alpha^2_t\Rr]-t\Rr)\d t = 0.
\end{align}
\end{lemma}

\begin{proof}
Recall that Theorem~\ref{t.gs} is derived from Theorem~\ref{t.laplace_uninvert} at $s=0$. Let $\gamma\in\cM_0$ be the unique minimizer of the Parisi formula~\eqref{e.laplace} (recall Remark~\ref{r.determ_ext_field}) given by this theorem at $s=0$, whose relation with $\alpha$ is described by~\eqref{e.AC.SDE} and~\eqref{e.identity.alpha}. So, $\gamma$ is the Parisi measure as often called in the spin glass literature. We recall a useful property of $\gamma$ from~\cite[Proposition~4]{chen2017fluctuations}. Denote the functional over $\cM_0$ inside the $\inf$ in~\eqref{e.def.gs} by $\mcl P$. Then, for every $\gamma'\in\cM_0$, we have
\begin{align}\label{e.Parisi_functional}
    \frac{\d\mcl P(\gamma+\lambda(\gamma'-\gamma))}{\lambda}\Big|_{\lambda=0}= \frac{1}{2}\int_0^1\xi''(t)(\gamma'(t)-\gamma(t))(\EE[\al_t^2] - t) \, \d t \geq 0.
\end{align}

Now, we are ready to derive ``$\Longrightarrow$'' in~\eqref{e.h=0_int=0}. When $h=0$, Proposition~\ref{p.h=0,0in_supp} gives $0\in\supp\d\gamma$. Hence, there is a vanishing sequence $(\delta_n)_{n\in\N}$ of positive numbers such that $\gamma(\delta_n)>0$. For each $n$, substituting $(\gamma(\cdot)-\gamma(\delta_n))\mathds{1}_{[\delta_n,1)}\in\cM_0$ for $\gamma'$ in~\eqref{e.Parisi_functional}, we get
\begin{align*}
    \gamma(\delta_n)\int_{\delta_n}^1\xi''(t)(\EE[\al_t^2] - t) \, \d t \leq -\int_0^{\delta_n}\xi''(t)\gamma(t)(\EE[\al_t^2] - t) \, \d t \leq 2\xi''(1)\gamma(\delta_n)\delta_n
\end{align*}
Divide both sides by $\gamma(\delta_n)$ and send $n\to\infty$ to get $\int_0^1\xi''(t)(\EE[\al_t^2] - t) \, \d t \leq 0$. On the other hand, the condition on $\alpha$ in~\eqref{e.gs} requires this integral to be nonnegative. Therefore, we have verified ``$\Longrightarrow$'' in~\eqref{e.h=0_int=0}.

Next, we turn to ``$\Longleftarrow$''. It is equivalent to assuming $h\neq 0$ and showing $\int_0^1\xi''(t)(\EE[\al_t^2] - t) \, \d t>0$. Since $\partial_x\Psi_\gamma(0,\cdot)$ (recall the PDE~\eqref{eq:parisipsi}) is odd and strictly increasing by Lemma~\ref{l.d_xPsi_odd_increase}, we have $\partial_x\Psi_\gamma(0,h)\neq 0$ when $h\neq 0$. Due to $\alpha_0=\partial_x\Psi_\gamma(0,h)$ in~\eqref{e.identity.alpha}, this implies $\EE[\alpha_0^2]>0$ and, by continuity, there is $\delta>0$ such that $\int_0^\delta\xi''(t)\Ll(\EE[\alpha_t^2]-t\Rr)\,\d t >0$. Meanwhile, the condition on $\alpha$ in~\eqref{e.gs} gives $\int_{\delta}^1\xi''(t)\Ll(\EE[\alpha_t^2]-t\Rr)\,\d t \geq0$. Combined, they yield $\int_0^1\xi''(t)\Ll(\EE[\alpha_t^2]-t\Rr)\,\d t > 0$ as desired.
\end{proof}

We record an immediate corollary on $\Lambda$ defined in~\eqref{e.Lambda(s)=}. 

\begin{corollary}[Strict convexity of $\Lambda$ at $h\neq 0$]
\label{c.bars.zero}
Let $\underline s$ be given as in Lemma~\ref{l.cond_strict_convex}. If $h\neq 0$, then $\underline s=0$ and thus $\Lambda$ is strictly convex on $[0,\infty)$ and satisfies $\Lambda(s)>s \, \gs$ for every $s>0$.
\end{corollary}

\begin{proof}
Suppose $\underline s>0$. Then, \eqref{e.cond_strict_convex} from Lemma~\ref{l.cond_strict_convex} implies $\int_0^1\xi''(t)\big(\EE\Ll[\alpha^2_t\Rr]-t\big)\d t = 0$ for the optimal martingale $\alpha$ from Theorem~\ref{t.gs}. By Lemma~\ref{l.h=0_int=0}, this implies $h=0$, reaching a contradiction. Now with $\underline s=0$, the properties of $\Lambda$ follow from Lemma~\ref{l.cond_strict_convex}.
\end{proof}

We conclude by proving Theorem~\ref{t.quadratic}, treating the two distinct cases separately.

\begin{proof}[Proof of Theorem~\ref{t.quadratic} in the case $h\neq0$]
Let $\bar\alpha$ be the unique maximizer given in Theorem~\ref{t.gs}, which satisfies $\gs = \EE \Ll[ h\bar\al_0+ \bar\al_1  \int_0^1 \sqrt{\xi''(t)} \, \d W_t \Rr]$.
Lemma~\ref{l.h=0_int=0} gives $c_0:= \int_0^1\xi''(t)\big(\EE\Ll[\bar\alpha^2_t\Rr]-t\big)\d t>0$. Using these and \eqref{e.firstdef.lambda*}, we get $\Lambda^*(r)\leq \frac{(r-\gs)^2}{2c_0}$.

On the other hand, we set $c_1 = \int_0^1\xi''(t)(1-t)\d t>0$. Then, for any $\alpha$ admissible in the variational formula for $\Lambda^*(r)$ in~\eqref{e.firstdef.lambda*}, we have
\begin{align*}
    \frac{\Ll(r-\EE\Ll[h\alpha_0+\alpha_1\int_0^1\sqrt{\xi''(t)}\d W_t\Rr]\Rr)^2}{2\int_0^1\xi''(t)\Ll(\EE\Ll[\alpha^2_t\Rr]-t\Rr)\d t}  \stackrel{\eqref{e.gs}}{\geq} \frac{(r-\gs)^2}{2c_1}
\end{align*}
implying $\Lambda^*(r)\geq \frac{(r-\gs)^2}{2c_1}$ and completing the proof.
\end{proof}

\begin{proof}[Proof of Theorem~\ref{t.quadratic} in the case $h=0$]

It suffices to show that for any sequence $(r_n)_{n\in\N}$ satisfying $r_n>\gs$ and $\lim_{n\to\infty}r_n=\gs$, we can extract a subsequence $(r_{n_k})_{k\in\N}$ such $\frac{\Lambda^*(r_{n_k})}{(r_{n_k}-\gs)^2}$ diverges to $+\infty$. For each $n$, we invoke Theorem~\ref{t.LDP} to get $s_n$ satisfying $r_n=\Lambda'(s_n)$ and minimizing $\Lambda^*(r_n)$, and $\alpha^\pnp$ maximizing $\frac{\Lambda(s_n)}{s_n}$.

First, we identify the limit of $s_n$.
Let $\underline s$ be given in Lemma~\ref{l.cond_strict_convex}, which gives that $\Lambda$ is strictly convex on $(\underline s,\infty)$. Hence, $\Lambda'$ is strictly increasing and bijective from $(\underline s,\infty)$ to $(\Lambda'(\underline s), \Lambda'(\infty))$ where $\Lambda'(\infty):=\lim_{s\to\infty}\Lambda'(s)$. The definition of $\underline s$ in Lemma~\ref{l.cond_strict_convex} also gives $\Lambda'(\underline s)=\gs$. This along with $\Lambda'(s_n)=r_n\to \gs$ and the continuity of $\Lambda'$ due to Lemma~\ref{l.Lambda'(s)=} implies
\begin{align}\label{e.lim_s_n=s}
    \lim_{n\to\infty}s_n=\underline s.
\end{align}

Next, we turn to $\alpha^\pnp$. By a similar argument used in Step~2 of the proof of Lemma~\ref{l.Lambda'(s)=}, we can extract a subsequence $(W,(\alpha^{(n_k)}_t)_{t\in\Q\cap[0,1]})$ that converges in law to $(W,(\bar\alpha_t)_{t\in\Q\cap[0,1]})$ for some martingale $\bar\alpha$ defined on a possibly larger probability space. Recall the expression of $\frac{\Lambda(s)}{s}$ from~\eqref{e.g(s,alpha)=} and \eqref{e.Lambda(s)=} which is exactly the left-hand side in~\eqref{e.lim=sup_mart_achieve} of Theorem~\ref{t.laplace_uninvert}. The choice of $\alpha^\pnp$ ensures $\frac{\Lambda(s_n)}{s_n}=\frac{g(s_n,\alpha^\pnp)}{s_n}$. At $s=0$, we understand $\frac{\Lambda(0)}{0}=\Lambda'(0)=\gs$ (see~\eqref{e.Lambda'(0)=gs}). Then, using the continuity of $(s,\alpha)\mapsto\frac{g(s,\alpha)}{s}$, the aforementioned convergence of $\alpha^{(n_k)}$, and~\eqref{e.lim_s_n=s}, we can pass $n_k$ to the limit to get $\frac{\Lambda(\underline s)}{\underline s} = \frac{g(\underline s,\bar\alpha)}{\underline s}$. In other words, $\bar\alpha$ is the unique optimizer of $\frac{\Lambda(\underline s)}{\underline s}$. 
If $\underline s>0$, we invoke Lemma~\ref{l.cond_strict_convex} to see
\begin{align}\label{e.int...bar_alpha=0}
    \int_0^1\xi''(t)\Ll(\EE\Ll[\bar\alpha^2_t\Rr]-t\Rr)\d t =0.
\end{align}
If $\underline s =0$, then $\bar\alpha$ is the optimizer of $\frac{\Lambda(0)}{0}=\gs$, which by Lemma~\ref{l.h=0_int=0} and  the assumption $h=0$ also implies~\eqref{e.int...bar_alpha=0}. Then, the convergence of $\alpha^{(n_k)}$ and~\eqref{e.int...bar_alpha=0} yield
\begin{align*}
    \lim_{k\to\infty}\int_0^1\xi''(t)\Ll(\EE\Ll[\Ll(\alpha^{(n_k)}_t\Rr)^2\Rr]-t\Rr)\d t =0.
\end{align*}

Recall that $\alpha^\pnp$ is the optimizer for $\Lambda^*(r_n)$ in~\eqref{e.firstdef.lambda*} which along with~\eqref{e.gs} gives
\begin{align*}
    \Lambda^*\Ll(r_{n_k}\Rr) \geq \frac{(r_{n_k}-\gs)^2}{2\int_0^1\xi''(t)\Ll(\EE\Ll[\Ll(\alpha^{(n_k)}_t\Rr)^2\Rr]-t\Rr)\d t}.
\end{align*}
Combining the above two displays, we have that $\frac{\Lambda^*\Ll(r_{n_k}\Rr)}{(r_{n_k}-\gs)^2}$ diverges to $+\infty$ as $k\to\infty$. As explained in the beginning, this gives the desired result.
\end{proof}

\medskip

\noindent \textbf{Acknowledgements.} HBC acknowledges funding from the NYU Shanghai Start-Up Fund and support from the NYU–ECNU Institute of Mathematical Sciences at NYU Shanghai. AG acknowledges funding from the European Research Council (ERC) under the European Union
Horizon 2020 research and innovation program (grant agreement No.\ 884584). JK acknowledges support from the Natural Sciences and Engineering Research Council of Canada and the Canada Research
Chairs programme (RGPIN-2020-04597, DGECR-2020 00199). JCM acknowledges the support of the ERC MSCA grant SLOHD (101203974).

\small
\bibliographystyle{plain}
\bibliography{gs_ldp}

\newcommand{\noop}[1]{} \def\cprime{$'$}
\begin{thebibliography}{10}

\bibitem{AS2}
Michael Aizenman, Robert Sims, and Shannon~L. Starr.
\newblock Extended variational principle for the {S}herrington-{K}irkpatrick
  spin-glass model.
\newblock {\em Phys. Rev. B}, 68(21):214403, 2003.

\bibitem{aliprantis2006infinite}
Charalambos~D. Aliprantis and Kim~C. Border.
\newblock {\em Infinite dimensional analysis. {A} hitchhiker's guide}.
\newblock Springer, Berlin, third edition, 2006.

\bibitem{aspelmeier2008finite}
Timo Aspelmeier, Alain Billoire, Enzo Marinari, and Michael~Arthur Moore.
\newblock Finite-size corrections in the {S}herrington-{K}irkpatrick model.
\newblock {\em J. Phys. A}, 41(32):324008, 21, 2008.

\bibitem{auffinger2015properties}
Antonio Auffinger and Wei-Kuo Chen.
\newblock On properties of {P}arisi measures.
\newblock {\em Probab. Theory Related Fields}, 161(3-4):817--850, 2015.

\bibitem{auffinger2015parisi}
Antonio Auffinger and Wei-Kuo Chen.
\newblock The {P}arisi formula has a unique minimizer.
\newblock {\em Comm. Math. Phys.}, 335(3):1429--1444, 2015.

\bibitem{auffinger2017parisi}
Antonio Auffinger and Wei-Kuo Chen.
\newblock Parisi formula for the ground state energy in the mixed {$p$}-spin
  model.
\newblock {\em Ann. Probab.}, 45(6B):4617--4631, 2017.

\bibitem{baik2016fluctuations}
Jinho Baik and Ji~Oon Lee.
\newblock Fluctuations of the free energy of the spherical
  {S}herrington-{K}irkpatrick model.
\newblock {\em J. Stat. Phys.}, 165(2):185--224, 2016.

\bibitem{benarous2001aging}
G\'erard Ben~Arous, Amir Dembo, and Alice Guionnet.
\newblock Aging of spherical spin glasses.
\newblock {\em Probab. Theory Related Fields}, 120(1):1--67, 2001.

\bibitem{benarous1997large}
G\'erard Ben~Arous and Alice Guionnet.
\newblock Large deviations for {W}igner's law and {V}oiculescu's
  non-commutative entropy.
\newblock {\em Probab. Theory Related Fields}, 108(4):517--542, 1997.

\bibitem{boettcher2005extremal}
Stefan Boettcher.
\newblock Extremal optimization for {S}herrington-{K}irkpatrick spin glasses.
\newblock {\em Eur. Phys. J. B}, 46(4):501--505, 2005.

\bibitem{boettcher2010simulations}
Stefan Boettcher.
\newblock Simulations of ground state fluctuations in mean-field {I}sing spin
  glasses.
\newblock {\em J. Stat. Mech. Theory Exp.}, 2010(7):P07002, 2010.

\bibitem{bouchaud2003energy}
Jean-Philippe Bouchaud, Florent Krzakala, and Olivier Martin.
\newblock Energy exponents and corrections to scaling in {I}sing spin glasses.
\newblock {\em Phys. Rev. B}, 68(22):224404, 2003.

\bibitem{chatterjee}
Sourav Chatterjee.
\newblock {\em Superconcentration and related topics}.
\newblock Springer Monographs in Mathematics. Springer, Cham, 2014.

\bibitem{chen2025convex}
Hong-Bin Chen, Victor Issa, and Jean-Christophe Mourrat.
\newblock The convex structure of the {P}arisi formula for multi-species spin
  glasses.
\newblock {\em Preprint, arXiv:2508.06397}, 2025.

\bibitem{chen2025free}
Hong-Bin Chen and Jean-Christophe Mourrat.
\newblock On the free energy of vector spin glasses with nonconvex
  interactions.
\newblock {\em Probab. Math. Phys.}, 6(1):1--80, 2025.

\bibitem{chen2025hamilton}
Hong-Bin Chen and Jiaming Xia.
\newblock Hamilton-{J}acobi equations from mean-field spin glasses.
\newblock {\em Probab. Theory Related Fields}, 192(3-4):803--873, 2025.

\bibitem{chen2014chaos}
Wei-Kuo Chen.
\newblock Chaos in the mixed even-spin models.
\newblock {\em Comm. Math. Phys.}, 328(3):867--901, 2014.

\bibitem{chen2017fluctuations}
Wei-Kuo Chen, Partha Dey, and Dmitry Panchenko.
\newblock Fluctuations of the free energy in the mixed {$p$}-spin models with
  external field.
\newblock {\em Probab. Theory Related Fields}, 168(1-2):41--53, 2017.

\bibitem{chen2018energy}
Wei-Kuo Chen, Madeline Handschy, and Gilad Lerman.
\newblock On the energy landscape of the mixed even {$p$}-spin model.
\newblock {\em Probab. Theory Related Fields}, 171(1-2):53--95, 2018.

\bibitem{chen2017parisi}
Wei-Kuo Chen and Arnab Sen.
\newblock Parisi formula, disorder chaos and fluctuation for the ground state
  energy in the spherical mixed {$p$}-spin models.
\newblock {\em Comm. Math. Phys.}, 350(1):129--173, 2017.

\bibitem{collins2025free}
Elizabeth Collins-Woodfin and Han~Gia Le.
\newblock Free energy fluctuations in {SK} and related spin glass models: {A}
  literature survey.
\newblock {\em Preprint, arXiv:2510.26960}, 2025.

\bibitem{conmin}
Pierluigi Contucci and Emanuele Mingione.
\newblock A multi-scale spin-glass mean-field model.
\newblock {\em Comm. Math. Phys.}, 368(3):1323--1344, 2019.

\bibitem{crisanti1992replica}
Andrea Crisanti, Giovanni Paladin, Angelo Vulpiani, and Hans-Jurgen Sommers.
\newblock Replica trick and fluctuations in disordered systems.
\newblock {\em J. Phys. I}, 2(7):1325--1332, 1992.

\bibitem{sparse_PDE}
Tomas Dominguez and Jean-Christophe Mourrat.
\newblock Infinite-dimensional {H}amilton-{J}acobi equations for statistical
  inference on sparse graphs.
\newblock {\em SIAM J. Math. Anal.}, 56(4):4530--4593, 2024.

\bibitem{sparse_prob}
Tomas Dominguez and Jean-Christophe Mourrat.
\newblock Mutual information for the sparse stochastic block model.
\newblock {\em Ann. Probab.}, 52(2):434--501, 2024.

\bibitem{dominguez2024book}
Tomas Dominguez and Jean-Christophe Mourrat.
\newblock {\em Statistical mechanics of mean-field disordered systems: a
  {H}amilton-{J}acobi approach}.
\newblock Zurich Lectures in Advanced Mathematics. European Mathematical
  Society (EMS), Z\"{u}rich, 2024.

\bibitem{fan1953minimax}
Ky~Fan.
\newblock Minimax theorems.
\newblock {\em Proc. Nat. Acad. Sci. USA}, 39:42--47, 1953.

\bibitem{franz2026overlap}
Silvio Franz, Giorgio Parisi, and Federico Ricci-Tersenghi.
\newblock Overlap locking and non-perturbative effects in spin glasses.
\newblock {\em Preprint, arXiv:2602.18399}, 2026.

\bibitem{grantGround}
Curtis Grant, Aukosh Jagannath, and Justin Ko.
\newblock Pseudo-maximum likelihood theory for high-dimensional rank one
  inference, 2025.

\bibitem{guerra2001sum}
Francesco Guerra.
\newblock Sum rules for the free energy in the mean field spin glass model.
\newblock {\em Fields Inst. Commun.}, 30:161--170, 2001.

\bibitem{gue03}
Francesco Guerra.
\newblock Broken replica symmetry bounds in the mean field spin glass model.
\newblock {\em Comm. Math. Phys.}, 233(1):1--12, 2003.

\bibitem{huang2023constructive}
Brice Huang and Mark Sellke.
\newblock A constructive proof of the spherical {P}arisi formula.
\newblock {\em Preprint, arXiv:2311.15495}, 2023.

\bibitem{issa2024hopf}
Victor Issa.
\newblock A {H}opf-like formula for mean-field spin glass models.
\newblock {\em Preprint, arXiv:2410.08754}, 2024.

\bibitem{JagSenGround}
Aukosh Jagannath and Subhabrata Sen.
\newblock On the unbalanced cut problem and the generalized
  {S}herrington-{K}irkpatrick model.
\newblock {\em Ann. Inst. Henri Poincar\'e{} D}, 8(1), 2021.

\bibitem{kim2007ground}
Seung-Yeon Kim, Sung~Jong Lee, and Jooyoung Lee.
\newblock Ground-state energy and energy landscape of the
  sherrington-kirkpatrick spin glass.
\newblock {\em Phys. Rev. B}, 76:184412, 2007.

\bibitem{mutual_information}
Anastasia Kireeva and Jean-Christophe Mourrat.
\newblock Breakdown of a concavity property of mutual information for
  non-{G}aussian channels.
\newblock {\em Inf. Inference}, 13(2):Paper No. iaae008, 21, 2024.

\bibitem{lacroix2024replica}
Bertrand Lacroix-A-Chez-Toine, Yan~V. Fyodorov, and Pierre Le~Doussal.
\newblock Replica-symmetry breaking transitions in the large deviations of the
  ground-state of a spherical spin-glass.
\newblock {\em J. Stat. Phys.}, 191(2):Paper No. 11, 76, 2024.

\bibitem{mourrat2020nonconvex}
Jean-Christophe Mourrat.
\newblock Nonconvex interactions in mean-field spin glasses.
\newblock {\em Probab. Math. Phys.}, 2(2):281--339, 2021.

\bibitem{mourrat2023free}
Jean-Christophe Mourrat.
\newblock Free energy upper bound for mean-field vector spin glasses.
\newblock {\em Ann. Inst. Henri Poincar\'e{} Probab. Stat.}, 59(3):1143--1182,
  2023.

\bibitem{mourrat2025spin}
Jean-Christophe Mourrat.
\newblock Spin glasses and the parisi formula.
\newblock {\em Preprint, arXiv:2510.01054}, 2025.

\bibitem{mourrat2025uninverting}
Jean-Christophe Mourrat.
\newblock Un-inverting the {P}arisi formula.
\newblock {\em Ann. Inst. Henri Poincar\'e{} Probab. Stat.}, 61(4):2709--2720,
  2025.

\bibitem{palassini2003ground}
Matteo Palassini.
\newblock Ground-state energy fluctuations in the {S}herrington-{K}irkpatrick
  model.
\newblock {\em Preprint, arXiv:cond-mat/0307713}, 2003.

\bibitem{palassini2008ground}
Matteo Palassini.
\newblock Ground-state energy fluctuations in the {S}herrington-{K}irkpatrick
  model.
\newblock {\em J. Stat. Mech. Theory Exp.}, 2008(10):P10005, 2008.

\bibitem{Panchenko2008}
Dmitry Panchenko.
\newblock On differentiability of the {P}arisi formula.
\newblock {\em Electron. Commun. Probab.}, 13:241--247, 2008.

\bibitem{pan.aom}
Dmitry Panchenko.
\newblock The {P}arisi ultrametricity conjecture.
\newblock {\em Ann. of Math. (2)}, 177(1):383--393, 2013.

\bibitem{pan}
Dmitry Panchenko.
\newblock {\em The {S}herrington-{K}irkpatrick model}.
\newblock Springer Monographs in Mathematics. Springer, New York, 2013.

\bibitem{pan.multi}
Dmitry Panchenko.
\newblock The free energy in a multi-species {S}herrington-{K}irkpatrick model.
\newblock {\em Ann. Probab.}, 43(6):3494--3513, 2015.

\bibitem{parisi2008large}
Giorgio Parisi and Tommaso Rizzo.
\newblock Large deviations in the free energy of mean-field spin glasses.
\newblock {\em Phys. Rev. Lett.}, 101(11):117205, 2008.

\bibitem{parisi2009phase}
Giorgio Parisi and Tommaso Rizzo.
\newblock Phase diagram and large deviations in the free energy of mean-field
  spin glasses.
\newblock {\em Phys. Rev. B}, 79(13):134205, 2009.

\bibitem{parisi2010universality}
Giorgio Parisi and Tommaso Rizzo.
\newblock Universality and deviations in disordered systems.
\newblock {\em Phys. Rev. B}, 81(9):094201, 2010.

\bibitem{parisi2019study}
Giorgio Parisi, Leopoldo Sarra, and Lorenzo Talamanca.
\newblock Study of longitudinal fluctuations of the sherrington--kirkpatrick
  model.
\newblock {\em J. Stat. Mech. Theory Exp.}, 2019(3):033302, 2019.

\bibitem{rockafellar1998variational}
R.~Tyrrell Rockafellar and Roger J.-B. Wets.
\newblock {\em Variational analysis}, volume 317 of {\em Grundlehren der
  Mathematischen Wissenschaften}.
\newblock Springer-Verlag, Berlin, 1998.

\bibitem{sion1958minimax}
Maurice Sion.
\newblock On general minimax theorems.
\newblock {\em Pacific J. Math.}, 8:171--176, 1958.

\bibitem{subag2017extremal}
Eliran Subag and Ofer Zeitouni.
\newblock The extremal process of critical points of the pure {$p$}-spin
  spherical spin glass model.
\newblock {\em Probab. Theory Related Fields}, 168(3-4):773--820, 2017.

\bibitem{Tpaper}
Michel Talagrand.
\newblock The {P}arisi formula.
\newblock {\em Ann. of Math. (2)}, 163(1):221--263, 2006.

\bibitem{talagrand2007large}
Michel Talagrand.
\newblock Large deviations, {G}uerra's and {A}.{S}.{S}. schemes, and the
  {P}arisi hypothesis.
\newblock {\em J. Stat. Phys.}, 126(4-5):837--894, 2007.

\bibitem{tracy1996orthogonal}
Craig~A. Tracy and Harold Widom.
\newblock On orthogonal and symplectic matrix ensembles.
\newblock {\em Comm. Math. Phys.}, 177(3):727--754, 1996.

\end{thebibliography}

\end{document}